\documentclass[11pt]{article}

\usepackage{fullpage,amsthm,amssymb,amsmath,amsfonts,rotate}

\usepackage{graphicx,algorithmic,algorithm}
\usepackage[]{datetime}
\usepackage{aliascnt} 
\usepackage{hyperref,todonotes}
\usepackage{lmodern}
\usepackage{tikz}
\usetikzlibrary{calc,3d}
\usetikzlibrary{intersections,decorations.pathmorphing,shapes,decorations.pathreplacing,fit}

\usepackage{subcaption}
\usepackage{cleveref}

\usepackage[greek,english]{babel}
\usepackage[utf8]{inputenc}
\usepackage{alphabeta}

\usepackage{amsmath,ifthen,color,eurosym}
\usepackage{wrapfig,hyperref,cite}
\usepackage{latexsym,amsthm,amsmath,amssymb,url}

\extrafloats{100}

\usepackage{enumerate}
\usepackage{enumitem}

\hypersetup{
	pdftitle = {Minor Obstructions for Apex Pseudoforests},
	pdfauthor=  {Alexandros Leivaditis, Alexandros Singh, Giannos Stamoulis, Dimitrios Thilikos, and Konstantinos Tsatsanis},
	colorlinks = true,
	urlcolor = black!50!red,
	linkcolor = black!50!blue,
	citecolor = black!50!green,
	menucolor = {red}
	filecolor = {cyan},
	anchorcolor = {black}
	urlcolor = {magenta},
	runcolor = {cyan},
	linkbordercolor = {blue}
}

\newcommand{\hh}{
	\bibliographystyle{plain}
	\bibliography{biblio}
\end{document}}

\newcommand{\remove}[1]{}
\usepackage{color}

\definecolor{Black}{rgb}{0,0, 0}
\definecolor{Blue}{rgb}{0, 0 ,1}
\definecolor{green}{rgb}{0.05,0.9,0.05}
\definecolor{Red}{rgb}{1, 0 ,0}
\definecolor{RRed}{rgb}{.5, 0 ,0}
\definecolor{White}{rgb}{1, 1, 1}
\definecolor{Grey}{rgb}{.6, .6, .6}
\definecolor{Yellow}{rgb}{.55,.55,0}
\definecolor{mustard}{rgb}{1.0, 0.86, 0.35}
\definecolor{applegreen}{rgb}{0.55, 0.71, 0.0}
\definecolor{fluorescentyellow}{rgb}{0.8, 1.0, 0.0}

\usepackage{tikz}
\usetikzlibrary{backgrounds}
\usetikzlibrary{decorations.pathreplacing}
\tikzset{black node/.style={draw, circle, fill = black, minimum size = 4pt, inner sep = 0pt}}
\tikzset{hblack node/.style={draw, circle, fill = black, minimum size = 5pt, inner sep = 0pt}}

\newcommand{\mynewtheorem}[2]{
\newaliascnt{#1}{dummy}
\newtheorem{#1}[#1]{#2}
\aliascntresetthe{#1}
\expandafter\def\csname #1autorefname\endcsname{#2}
}

\renewcommand{\deg}{{\bf deg}}

\theoremstyle{plain}
\mynewtheorem{theorem}{Theorem}
\mynewtheorem{proposition}{Proposition}
\mynewtheorem{corollary}{Corollary}
\mynewtheorem{lemma}{Lemma}

\theoremstyle{definition}
\mynewtheorem{definition}{Definition}
\theoremstyle{remark}
\mynewtheorem{sublemma}{Sublemma}
\mynewtheorem{claim}{Claim}
\mynewtheorem{remark}{Remark}
\mynewtheorem{conjecture}{Conjecture}
\mynewtheorem{question}{Question}       
\mynewtheorem{observation}{Observation}
\mynewtheorem{problem}{Problem}
\mynewtheorem{note}{Note}

\newcommand{\excl}{{\bf exc}}
\newcommand{\obs}{{\bf obs}}

\newcommand{\Nbb}{\mathbb{N}}

\newcommand{\intv}[1]{\left [ #1 \right ]}
\newcommand{\cupall}{\pmb{\pmb{\bigcup}}}

\usepackage{titling}
\newcommand*\samethanks[1][\value{footnote}]{\footnotemark[#1]}

\title{{Minor Obstructions for Apex-Pseudoforests\thanks{Emails:\! {\scriptsize  \sf \href{mailto:livaditisalex@gmail.com}{livaditisalex@gmail.com},\!
\href{mailto:alexsingh@di.uoa.gr}{alexsingh@di.uoa.gr},\!
\href{mailto:giannosstam@di.uoa.gr}{giannosstam@di.uoa.gr},\!
\href{mailto:sedthilk@thilikos.info}{sedthilk@thilikos.info},\!
\href{mailto:kostistsatsanis@gmail.com}{kostistsatsanis@gmail.com}\,.}}}}

\author{\bigskip 
Alexandros Leivaditis%
\thanks{Mathematical Institute, Universiteit Leiden, Netherlands.}~\thanks{Department of Mathematics, National and Kapodistrian University of Athens, Athens, Greece.}%
\and 
Alexandros Singh%
\thanks{Université Sorbonne Paris Nord, Laboratoire d’Informatique de Paris Nord, CNRS, Villetaneuse, France.}~\thanks{Inter-university Postgraduate Programme ``Algorithms, Logic, and Discrete Mathematics'' (ALMA).}
\and 
Giannos Stamoulis\thanks{Department of Informatics and Telecommunications, National and Kapodistrian University of Athens, Greece}~\samethanks[5]  
\and
Dimitrios  M. Thilikos\thanks{LIRMM, Univ Montpellier, CNRS, Montpellier, France.}~~\,\thanks{Supported by projects DEMOGRAPH (ANR-16-CE40-0028) and ESIGMA (ANR-17-CE23-0010), ELIT (ANR-20-CE48-0008), and the French-German Collaboration ANR/DFG Project UTMA (ANR-20-CE92-0027).} \and Konstantinos Tsatsanis\samethanks[5]~\,\samethanks[3]}

\date{\empty}
\begin{document}
  
%
%

\maketitle

\vspace{-18mm}
\begin{abstract}
	\noindent A graph is called  a  {\em pseudoforest} if none of its connected 
	components contains more than one cycle.
	A graph is  an  {\em apex-pseudoforest} if it can become a pseudoforest by removing one of its vertices.
	We identify 33 graphs that form the minor obstruction set of the class of {apex}-pseudoforests,  i.e., the set of all minor-minimal graphs that are not apex-pseudoforests. 
\end{abstract}
\vspace{-1mm}

\vspace{-3mm}

\section{Introduction}

\label{writings}

%
All graphs in this paper are undirected, finite, and simple, i.e., without loops or multiple edges.

A graph $G$ is a {\em pseudoforest} if every connected component of $G$ contains at most one cycle.
We denote by ${\cal P}$ the set of all pseudoforests. 

We say that a graph $H$ is a {\em minor} of $G$ if a graph isomorphic to $H$ 
can be obtained by some subgraph of $G$ after applying edge contractions. (As in this paper we consider only simple graphs, we always assume that in case multiple edges are created after a contraction, then these edges are automatically suppressed to simple edges.)
We say that a graph class ${\cal G}$ is {\em minor-closed} if every minor of a graph in ${\cal G}$
is also a member of ${\cal G}$. Given a minor-closed graph class ${\cal G}$,  its {\em minor obstruction set} is defined as the set of all minor-minimal graphs that are not in ${\cal G}$ and is denoted by $\obs({\cal G})$. For simplicity, we drop the term ``minor'' when we refer to an obstruction set as, in this paper, we only  consider minor obstruction sets. We also refer to the members of  $\obs({\cal G})$ as {\em obstructions of} ${\cal G}$.
Given a set of graphs ${\cal H}$ we denote by ${\excl}({\cal H})$ as the set containing every graph $G$ that excludes all graphs in ${\cal H}$ as minors.

Notice that if ${\cal G}$ is minor-closed, then a graph $G$ belongs in ${\cal G}$ iff none of the graphs in $\obs({\cal G})$ is a minor of $G$. In this way, $\obs({\cal G})$ can be seen 
as a complete characterization of ${\cal G}$ in terms of forbidden  minors,
i.e., ${\cal G}=\excl(\obs({\cal G}))$.

According to the Roberston and Seymour theorem~\cite{RobertsonSXX}, for every minor-closed graph class ${\cal G}$, the set $\obs({\cal G})$ is finite.
The study of $\obs({\cal G})$ for distinct instantiations of minor-closed graph classes 
is an active topic in graph theory (e.g., see~\cite{BodlaenderThil99,ArnborgPC90forb,JobsonK18allm,CattellDDFL00onco,Chlebikova02thes,Kabanets97reco,GiannopoulouDT12forb,Kuratowski37surl,Archdeacon06akur,DinneenX02mino,KinnersleyL94,DinneenL07prop,DinneenV12obst,CourcelleDF97,Dinneen97,GagarinMC09theb,Lagergren98,RueST12oute,MoharS12obst,BienstockD92ono,Ramachandramurthi97thes,Thilikos00,KoutsonasTY14oute,FellowsJ13fpti,GuptaI97bound,RobertsonST95sach,LiptonMMPRT16sixv,Yu06more,MattmanP16thea, Pierce14PhDThesis}, see also~\cite{Mattman16forb,Adler08open} for  related surveys).

It is easy to observe that pseudoforests form a minor-closed 
graph class. Moreover, it holds that $\obs({\cal P})=\{K_{4}^{-},Z\}$, {where $K_{4}^{-}$ is the {\em diamond graph}, i.e., the complete graph on $4$ vertices minus an edge and $Z$ is the {\em butterfly} graph (also known as {\em bow-tie} or {\em  hourglass} graph), obtained by the disjoint union of two triangles after identifying two of their vertices (see \autoref{ostructionspse}).}To see why $\obs({\cal P})=\{K_{4}^{-},Z\}$, notice that these two obstructions express the existence of two cycles in the same connected component of a graph.\medskip

\begin{figure}[H]
		\centering
				\scalebox{.53}{\begin{tikzpicture}[myNode/.style = hblack node, scale=0.8]
					\node[myNode] (n1) at (1.5,1) {}; 
					\node[myNode] (n1) at (2.5,1) {};
					\node[myNode] (n1) at (2,2) {};
					\node[myNode] (n1) at (2,0.1) {};

					\draw (1.5,1) -- (2.5,1);
					\draw (2.5,1) -- (2,2);
					\draw (2,2) -- (1.5,1);
					\draw (2,0.1) -- (1.5,1);
					\draw (2,0.1) -- (2.5,1);

					\begin{scope}[xshift=3cm]
					
					\node[myNode] (n1) at (1,0) {}; 
					\node[myNode] (n2) [right =of n1] {};
					\coordinate (Middle) at ($(n1)!0.5!(n2)$);
					\node[myNode] (n3) [above =0.8cm of Middle] {};
					\node[myNode] (n4) [above =1.5cm of n1] {};
					\node[myNode] (n5) [above =1.5cm of n2] {};

					\foreach \from/\to in               {n1/n2,n1/n3,n2/n3,n4/n5,n4/n3,n5/n3}
					\draw (\from) -- (\to);
					\end{scope}
				\end{tikzpicture}}	
		\caption{The graphs in $\obs({\cal P})$.}\label{ostructionspse}
\end{figure}

Given a non-negative integer $k$ and a graph class ${\cal G}$, we say that 
a graph  $G$ is a {\em $k$-apex of} ${\cal G}$ if it can be transformed to a member of ${\cal G}$ after removing at most $k$ of its vertices. We also use the term {\em apex of } ${\cal G}$ as a shortcut of $1$-apex of ${\cal G}$.
We denote the set of all $k$-apices of ${\cal G}$ by  ${\cal A}_{k}({\cal G})$.

The problem of characterizing 
$k$-apices of graph classes, has attracted a lot of attention, both 
from the combinatorial and algorithmic point of view. This problem
can be seen as a part  of the wider family of {\sl Graph Modification Problems} (where the modification is the removal of a vertex).

It is easy to see that if ${\cal G}$ is minor-closed, then ${\cal A}_{k}({\cal G})$ is also minor-closed for every $k\geq 0$. Therefore $\obs({\cal A}_{k}({\cal G}))$ 
can be seen as a complete characterization of  $k$-apices of ${\cal G}$.  The study of $\obs({\cal A}_{k}({\cal G}))$
when ${\cal G}$ is some minor-closed graph class has attracted some special attention and can generate several known graph invariants.
For instance, graphs with a vertex cover of size at most $k$ are the graphs in 
${\cal A}_{k}(\excl(\{K_{2}\}))$, graphs 
with a  feedback vertex set at most $k$ are  
the graphs in ${\cal A}_{k}(\excl(\{K_{3}\}))$, and $k$-apex planar graphs 
are the graphs in ${\cal A}_{k}(\excl(\{K_{5},K_{3,3}\}))$.

The general problem that emerges is, given a finite set of graphs ${\cal H}$
and a positive integer $k$,
to identify the set $${\cal H}^{(k)}:=\obs({\cal A}_{k}({\excl}({\cal H}))).$$
A fundamental result in this direction is that the above problem is computable~\cite{AdlerGK08comp}. Moreover, it was shown in~\cite{FominLMS12plan}
that if ${\cal H}$ contains some planar graph, then every graph in ${\cal H}^{(k)}$
has $O(k^{h})$ vertices, where $h$ is some constant depending (non-constructively) on ${\cal H}$. Also, in~\cite{ZorosPhD17}, it was proved that, under the additional assumption that all graphs in ${\cal H}$ are connected,
this bound becomes linear on $k$ for the intersection of ${\cal H}^{(k)}$
with sparse graph classes such as planar graphs or bounded degree graphs.
An other structural result in this direction is the  
characterization of the disconnected obstructions in ${\cal H}^{(k)}$ in the case where ${\cal H}$  consists only of connected graphs~\cite{Dinneen97}.  

An other direction is to study  ${\cal H}^{(k)}$  for particular instantiations of ${\cal H}$ and $k$. In this direction,  $\{K_{2}\}^{(k)}$ has been identified for $k\in\{1,\ldots,5\}$ in~\cite{CattellDDFL00onco}, for $k=6$ in~\cite{DinneenX02mino} and for $k=7$ in~\cite{DinneenV12obst}, while the graphs in $\{K_{3}\}^{(i)}$ have been identified in~\cite{DinneenCF01forb} for $i\in\{1,2\}$.
Recently, in~\cite{Ding2016}, Ding and Dziobiak identified  the 57 graphs in $\{K_{4},K_{2,3}\}^{(1)}$, i.e., the obstruction
set for apex-outerplanar graphs, and the 25 graphs in $\{K_{4}^{-}\}^{(1)}$,   i.e., the obstruction
set for apex-cactus graphs (as announced in~\cite{DziobiakD2013}).
Moreover, the problem of identifying $\{K_{5},K_{3,3}\}^{(1)}$ (i.e., characterizing 1-apex planar graphs -- also simply known as {\sl apex graphs}) has attracted 
particular attention (see e.g.,~\cite{LiptonMMPRT16sixv,Mattman16forb,Yu06more}). In this direction, Mattman and Pierce conjectured that 
$\{K_{5},K_{3,3}\}^{(n)}$ contains the $Y\Delta Y$-families of $K_{n+5}$ and $K_{3,3,1^n}$ (that is the complete $(n+2)$-partite graph where two parts have three vertices and the rest $n$ parts consist of a single vertex each) and provided evidence on this~\cite{MattmanP16thea}. Moreover,  in~\cite{MattmanP16thea}, they showed that $|\{K_{5},K_{3,3}\}^{(1)}|>150$, $|\{K_{5},K_{3,3}\}^{(2)}|>82$, $|\{K_{5},K_{3,3}\}^{(3)}|>601$, $|\{K_{5},K_{3,3}\}^{(4)}|>520$, and $|\{K_{5},K_{3,3}\}^{(5)}|>608$.
Recently,  Jobson and Kézdy \cite{JobsonK18allm} identified all graphs of connectivity two in $\{K_{5},K_{3,3}\}^{(1)}$, where they also reported that $|\{K_{5},K_{3,3}\}^{(1)}|\geq 401$.

\medskip

In this paper we identify $\{K_{4}^{-}, Z\}^{(1)}$, i.e., the obstruction set 
of apex-pseudoforests. 
Let ${\cal O}^{0}, {\cal O}^{1}, {\cal O}^{2}, {\cal O}^{3}$ be the sets of graphs depicted in Figures \ref{unguided}, \ref{dispense}, \ref{animates}, \ref{machines}, respectively. 
Notice that, for every $i\in\{0,1,2,3\},$ the graphs in ${\cal O}^{i}$ are all $i$-connected but not $(i+1)$-connected.

Our main result is  the following.

\begin{theorem}
	\label{ruthless}
	$\obs({\cal A}_{1}({\cal P}))= {{\cal O}^{0}}\cup{\cal O}^{1}\cup {\cal O}^{2}\cup {\cal O}^{3}.$
\end{theorem}

\begin{figure}[H]
	\centering
	\begin{tabular}{ c c c }
		\subcaptionbox{${O}_{1}^{0}$\label{magnetic}}{
			\scalebox{.53}{\begin{tikzpicture}[myNode/.style = hblack node, scale=0.8]
				\node[myNode] (n1) at (1.5,1) {}; 
				\node[myNode] (n1) at (2.5,1) {};
				\node[myNode] (n1) at (2,2) {};
				\node[myNode] (n1) at (2,0.1) {};
				
				\node[myNode] (n1) at (1.5,4) {}; 
				\node[myNode] (n1) at (2.5,4) {};
				\node[myNode] (n1) at (2,5) {};
				\node[myNode] (n1) at (2,3) {};
				
				\draw (1.5,1) -- (2.5,1);
				\draw (2.5,1) -- (2,2);
				\draw (2,2) -- (1.5,1);
				\draw (2,0.1) -- (1.5,1);
				\draw (2,0.1) -- (2.5,1);
				
				\draw (1.5,4) -- (2.5,4);
				\draw (2.5,4) -- (2,5);
				\draw (2,5) -- (1.5,4);
				\draw (2,3) -- (1.5,4);
				\draw (2,3) -- (2.5,4);
				
				\begin{scope}[on background layer]
				\fill[gray!2] (0.75,5.5) -- (3.25,5.5) -- (3.25,-0.5) -- (0.75,-0.5) -- cycle;
				\end{scope}
				
				\end{tikzpicture}}} &
		\subcaptionbox{${O}_{2}^{0}$\label{consiste}}{
			\scalebox{.53}{\begin{tikzpicture}[myNode/.style = hblack node, scale=0.8]
				
				\begin{scope}[xshift=0.25cm]
				
				\node[myNode] (n1) at (1,0) {}; 
				\node[myNode] (n2) [right =of n1] {};
				\coordinate (Middle) at ($(n1)!0.5!(n2)$);
				\node[myNode] (n3) [above =0.8cm of Middle] {};
				\node[myNode] (n4) [above =1.5cm of n1] {};
				\node[myNode] (n5) [above =1.5cm of n2] {};
				
				\node[myNode] (n6) at (1,3) {}; 
				\node[myNode] (n7) [right =of n6] {};
				\coordinate (Middle2) at ($(n6)!0.5!(n7)$);
				\node[myNode] (n8) [above =0.8cm of Middle2] {};
				\node[myNode] (n9) [above =1.5cm of n6] {};
				\node[myNode] (n10) [above =1.5cm of n7] {};

				\foreach \from/\to in               {n1/n2,n1/n3,n2/n3,n4/n5,n4/n3,n5/n3}
				\draw (\from) -- (\to);
				
				\foreach \from/\to in               {n6/n7,n6/n8,n7/n8,n10/n9,n10/n8,n9/n8}
				\draw (\from) -- (\to);
				\end{scope}
				
				\begin{scope}[on background layer]
				\fill[gray!2] (0.75,5.5) -- (3.25,5.5) -- (3.25,-0.5) -- (0.75,-0.5) -- cycle;
				\end{scope}
				
				\end{tikzpicture}}} &
		\subcaptionbox{${O}_{3}^{0}$\label{narcotic}}{
			\scalebox{.53}{\begin{tikzpicture}[myNode/.style = hblack node, scale=0.8]
				
				\begin{scope}[xshift=0.25cm, yshift=0.8cm]
				
				\node[myNode] (n1) at (1,0) {}; 
				\node[myNode] (n2) [right =of n1] {};
				\coordinate (Middle) at ($(n1)!0.5!(n2)$);
				\node[myNode] (n3) [below =0.7cm of Middle] {};
				\node[myNode] (n4) [above =0.8cm of Middle] {};

				\foreach \from/\to in               {n1/n2,n1/n3,n2/n3,n4/n1,n4/n2}
				\draw (\from) -- (\to);
				
				\node[myNode] (n6) at (1,2.2) {}; 
				\node[myNode] (n7) [right =of n6] {};
				\coordinate (Middle2) at ($(n6)!0.5!(n7)$);
				\node[myNode] (n8) [above =0.8cm of Middle2] {};
				\node[myNode] (n9) [above =1.5cm of n6] {};
				\node[myNode] (n10) [above =1.5cm of n7] {};
				
				\foreach \from/\to in               {n6/n7,n6/n8,n7/n8,n10/n9,n10/n8,n9/n8}
				\draw (\from) -- (\to);
				
				\end{scope}
				
				\begin{scope}[on background layer]
				\fill[gray!2] (0.75,5.5) -- (3.25,5.5) -- (3.25,-0.5) -- (0.75,-0.5) -- cycle;
				\end{scope}
				\end{tikzpicture}}} \\
	\end{tabular}
	\vspace{-2mm}\caption{The set ${\cal O}^{0}$ of obstructions for ${\cal A}_{1}({\cal P})$ with vertex connectivity 0.}
	\label{unguided}
\end{figure}
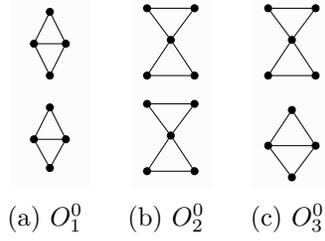
\vspace{-5mm}
\begin{figure}[H]
	\centering
	\begin{tabular}{ c c c c c c c}
		
		\subcaptionbox{${O}_{1}^{1}$\label{consists}}{
			\scalebox{.53}{\begin{tikzpicture}[myNode/.style = hblack node]
				
				\node[myNode] (n1) at (1,1) {};
				\node[myNode] (n1) at (1,2) {};
				\node[myNode] (n1) at (2,1.5) {};
				\node[myNode] (n1) at (3,1) {};
				\node[myNode] (n1) at (3,2) {};
				\node[myNode] (n1) at (4,1.5) {};
				\node[myNode] (n1) at (5,1) {};
				\node[myNode] (n1) at (5,2) {};
				
				\draw (1,1) -- (1,2);
				\draw (1,1) -- (2,1.5);
				\draw (1,2) -- (2,1.5);
				\draw (2,1.5) -- (3,1);
				\draw (2,1.5) -- (3,2);
				\draw (3,1) -- (3,2);
				\draw (5,1) -- (4,1.5);
				\draw (5,2) -- (4,1.5);
				\draw (4,1.5) -- (3,1);
				\draw (4,1.5) -- (3,2);
				\draw (5,1) -- (5,2);
				
				\end{tikzpicture}}} &
		\subcaptionbox{${O}_{2}^{1}$\label{appetite}}{
			\scalebox{.53}{\begin{tikzpicture}[myNode/.style = hblack node]
				
				\node[myNode] (n1) at (1,1) {};
				\node[myNode] (n1) at (1,2) {};
				\node[myNode] (n1) at (2,1.5) {};
				\node[myNode] (n1) at (3,1) {};
				\node[myNode] (n1) at (4,1) {};
				\node[myNode] (n1) at (5,1.5) {};
				\node[myNode] (n1) at (3,2) {};
				\node[myNode] (n1) at (4,2) {};

				\draw (1,1) -- (1,2);
				\draw (1,1) -- (2,1.5);
				\draw (1,2) -- (2,1.5);
				\draw (2,1.5) -- (3,1);
				\draw (2,1.5) -- (3,2);
				\draw (3,1) -- (4,1);
				\draw (3,2) -- (4,2);
				\draw (4,1) -- (5,1.5);
				\draw (4,2) -- (5,1.5);
				\draw (3,1) -- (3,2);
				\draw (4,1) -- (4,2);
				
				\end{tikzpicture}}} &
		\subcaptionbox{${O}_{3}^{1}$\label{adorning}}{
			\scalebox{.53}{\begin{tikzpicture}[myNode/.style = hblack node]
				\node[myNode] (n1) at (1,1) {};
				\node[myNode] (n1) at (1,2) {};
				\node[myNode] (n1) at (2,1.5) {};
				\node[myNode] (n1) at (3,1) {};
				\node[myNode] (n1) at (3,2) {};
				\node[myNode] (n1) at (4,1.5) {};
				\node[myNode] (n1) at (4,0.5) {};
				\node[myNode] (n1) at (4,2.5) {};
				
				\draw (1,1) -- (1,2);
				\draw (1,1) -- (2,1.5);
				\draw (1,2) -- (2,1.5);
				\draw (2,1.5) -- (3,1);
				\draw (2,1.5) -- (3,2);
				\draw (3,1) -- (3,2);
				\draw (4,0.5) -- (4,1.5);
				\draw (4,2.5) -- (4,1.5);
				\draw (4,1.5) -- (3,1);
				\draw (4,1.5) -- (3,2);
				\draw (4,0.5) -- (3,1);
				\draw (4,2.5) -- (3,2);
				
				\end{tikzpicture}}} &
		\subcaptionbox{${O}_{4}^{1}$\label{skeleton}}{
			\scalebox{.53}{\begin{tikzpicture}[myNode/.style = hblack node]
				\node[myNode] (n1) at (3.2,1.2) {};
				\node[myNode] (n1) at (3.2,1.8) {};
				\node[myNode] (n1) at (1,1.5) {};
				\node[myNode] (n1) at (1.85,1.5) {};
				\node[myNode] (n1) at (2.5,0.7) {};
				\node[myNode] (n1) at (2.5,2.3) {};
				\node[myNode] (n1) at (2.5,1.5) {};
				
				\draw (3.2,1.2) -- (3.2,1.8);
				\draw (3.2,1.2) -- (2.5,1.5);
				\draw (3.2,1.8) -- (2.5,1.5);
				\draw (1,1.5) -- (1.85,1.5) ;
				\draw (1,1.5) -- (2.5,0.7);
				\draw (1,1.5) -- (2.5,2.3);
				\draw (1.85,1.5) -- (2.5,0.7);
				\draw (1.85,1.5) -- (2.5,2.3);
				\draw (2.5,0.7) -- (2.5,2.3);
				
				\end{tikzpicture}}} &
		\subcaptionbox{${O}_{5}^{1}$\label{friesian}}{
			\scalebox{.53}{\begin{tikzpicture}[myNode/.style = hblack node]
				\node[myNode] (n1) at (0.3,1.2) {};
				\node[myNode] (n1) at (0.3,1.8) {};
				\node[myNode] (n1) at (1,1.5) {};
				\node[myNode] (n1) at (1.85,1.5) {};
				\node[myNode] (n1) at (2.5,0.7) {};
				\node[myNode] (n1) at (2.5,2.3) {};
				\node[myNode] (n1) at (3.5,1.5) {};
				
				\draw (0.3,1.2) -- (0.3,1.8);
				\draw (0.3,1.2) -- (1,1.5);
				\draw (0.3,1.8) -- (1,1.5);
				\draw (1,1.5) -- (1.85,1.5) ;
				\draw (1,1.5) -- (2.5,0.7);
				\draw (1,1.5) -- (2.5,2.3);
				\draw (1.85,1.5) -- (2.5,0.7);
				\draw (1.85,1.5) -- (2.5,2.3);
				\draw (2.5,0.7) -- (2.5,2.3);
				\draw (2.5,0.7) -- (3.5,1.5);
				\draw (2.5,2.3) -- (3.5,1.5);
				\end{tikzpicture}}} \\
		\subcaptionbox{${O}_{6}^{1}$\label{granting}}{
			\scalebox{.53}{\begin{tikzpicture}[myNode/.style = hblack node]
				
				\node[myNode] (n1) at (1.3,1) {};
				\node[myNode] (n1) at (1.3,2) {};
				\node[myNode] (n1) at (2,1.5) {};
				\node[myNode] (n1) at (3,1) {};
				\node[myNode] (n1) at (3,2) {};
				\node[myNode] (n1) at (4,1.5) {};
				\node[myNode] (n1) at (2,2.3) {};
				\node[myNode] (n1) at (2,0.7) {};
				\node[myNode] (n1) at (5,1.5) {};
				
				\draw (1.3,1) -- (1.3,2);
				\draw (1.3,1) -- (2,1.5);
				\draw (1.3,2) -- (2,1.5);
				\draw (2,1.5) -- (3,1);
				\draw (2,1.5) -- (3,2);
				\draw (3,1) -- (3,2);
				\draw (4,1.5) -- (3,1);
				\draw (4,1.5) -- (3,2);
				\draw (2,0.7) -- (3,1);
				\draw (2,2.3) -- (3,2);
				\draw (2,2.3) -- (2,1.5);
				\draw (2,0.7) -- (2,1.5);
				\draw (3,1) -- (5,1.5);
				\draw (3,2) -- (5,1.5);
				
				\end{tikzpicture}}} &
		
		\subcaptionbox{${O}_{7}^{1}$\label{segments}}{
			\scalebox{.53}{\begin{tikzpicture}[myNode/.style = hblack node]
				
				\node[myNode] (a1) at (240:1) {};
				\node[myNode] (a2) at (300:1) {};
				\node[myNode] (a) at (0,0) {};
				\node[myNode] (ab) at (150:1) {};
				\node[myNode] (b) at (90:1) {};
				\node[myNode] (bc) at (0.5, 1.75) {};
				\node[myNode] (c) at (1,1) {};
				\node[myNode] (cd) at (1.75,0.5) {};
				\node[myNode] (d) at (1,0) {};
				
				\draw[-] (a) -- (a1)-- (a2) -- (a) -- (ab) -- (b) -- (bc) -- (c) -- (cd) -- (d) -- (a) -- (b) -- (c) -- (d);

				\end{tikzpicture}}}&
		
		\subcaptionbox{${O}_{8}^{1}$\label{pensions}}{
			\scalebox{.53}{
				\begin{tikzpicture}[myNode/.style = hblack node]
				
				\node[myNode] (n1) at (0,0) {};
				\node[myNode] (n2) [above of=n1]  {};
				\coordinate (midway1) at ($(n1)!0.5!(n2)$)  {};
				\node[myNode] (n3) [right of=midway1] {};
				
				\node[myNode] (n4) [right=10ex of n3] {};
				\node[myNode] (n5) [above right=2ex and 5ex of n4] {};
				\node[myNode] (n6) [below=5ex of n5] {};
				
				\coordinate (midway2) at ($(n3)!0.5!(n4)$)  {};
				
				\node[myNode] (n7) [above right =15ex and 3.5ex of midway2] {};
				\node[myNode] (n8) [left =6.5ex of n7] {};
				\coordinate (midway3) at ($(n7)!0.5!(n8)$)  {};
				\node[myNode] (n9) [below =5ex of midway3] {};

				\draw (n1) -- (n2);
				\draw (n1) -- (n3);
				\draw (n2) -- (n3);
				
				\draw (n4) -- (n5);
				\draw (n4) -- (n6);
				\draw (n5) -- (n6);
				
				\draw (n9) -- (n8);
				\draw (n9) -- (n7);
				\draw (n7) -- (n8);

				\draw (n4) -- (n3);
				\draw (n9) -- (n2);
				\draw (n9) -- (n5);
				
				\end{tikzpicture}}} &
		
		\subcaptionbox{${O}_{9}^{1}$\label{presents}}{
			\scalebox{.53}{
				\begin{tikzpicture}[myNode/.style = hblack node]
				
				\node[myNode] (n1) at (0,0) {};
				\node[myNode] (n2) [above of=n1]  {};
				\coordinate (midway1) at ($(n1)!0.5!(n2)$)  {};
				\node[myNode] (n3) [right of=midway1] {};
				
				\node[myNode] (n4) [right=10ex of n3] {};
				\node[myNode] (n5) [above right=2ex and 5ex of n4] {};
				\node[myNode] (n6) [below=5ex of n5] {};
				
				\coordinate (midway2) at ($(n3)!0.5!(n4)$)  {};
				
				\node[myNode] (n7) [above right =10ex and 3.5ex of midway2] {};
				\node[myNode] (n8) [left =6.5ex of n7] {};
				\coordinate (midway3) at ($(n7)!0.5!(n8)$)  {};
				\node[myNode] (n9) [above =5ex of midway3] {};

				\draw (n1) -- (n2);
				\draw (n1) -- (n3);
				\draw (n2) -- (n3);
				
				\draw (n4) -- (n5);
				\draw (n4) -- (n6);
				\draw (n5) -- (n6);
				
				\draw (n9) -- (n8);
				\draw (n9) -- (n7);
				\draw (n7) -- (n8);

				\draw (n4) -- (n3);
				\draw (n8) -- (n3);
				\draw (n7) -- (n4);
				
				\end{tikzpicture}}}&
		
		\subcaptionbox{${O}_{10}^{1}$\label{teaching}}{
			\scalebox{.53}{
				\begin{tikzpicture}[myNode/.style = hblack node]
				
				\node[myNode] (n1) at (0,0) {};
				\node[myNode] (n2) [above of=n1]  {};
				\coordinate (midway1) at ($(n1)!0.5!(n2)$)  {};
				\node[myNode] (n3) [right of=midway1] {};
				
				\node[myNode] (n4) [right=10ex of n3] {};
				\node[myNode] (n5) [above right=2ex and 5ex of n4] {};
				\node[myNode] (n6) [below=5ex of n5] {};
				
				\coordinate (midway2) at ($(n3)!0.5!(n4)$)  {};
				
				\node[myNode] (n7) [above right =15ex and 3.5ex of midway2] {};
				\node[myNode] (n8) [left =6.5ex of n7] {};
				\coordinate (midway3) at ($(n7)!0.5!(n8)$)  {};
				\node[myNode] (n9) [below =5ex of midway3] {};

				\draw (n1) -- (n2);
				\draw (n1) -- (n3);
				\draw (n2) -- (n3);
				
				\draw (n4) -- (n5);
				\draw (n4) -- (n6);
				\draw (n5) -- (n6);
				
				\draw (n9) -- (n8);
				\draw (n9) -- (n7);
				\draw (n7) -- (n8);

				\draw (n4) -- (n3);
				\draw (n9) -- (n4);
				\draw (n9) -- (n3);
				
				\end{tikzpicture}}} \\
		\subcaptionbox{${O}_{11}^{1}$\label{immature}}{
			\scalebox{.53}{
				\begin{tikzpicture}[myNode/.style = hblack node]
				
				\node[myNode] (n1) at (0,0) {};
				\node[myNode] (n2) [below of=n1]  {};
				\coordinate (midway1) at ($(n1)!0.5!(n2)$)  {};
				
				\node[myNode] (n4) [right of=midway1] {};
				\node[myNode] (n5) [above right=2.5ex and 5ex of n4] {};
				\node[myNode] (n6) [below of=n5] {};
				\node[myNode] (n7) [right=10ex of n4] {};

				\node[myNode] (n8) [below right of=n2] {};
				\node[myNode] (n9) [below left of=n2] {};
				
				\draw (n1) -- (n2);
				\draw (n2) -- (n4);
				\draw (n1) -- (n4);
				\draw (n2) -- (n6);
				
				\draw (n4) -- (n5);
				\draw (n4) -- (n6);
				\draw (n5) -- (n6);
				\draw (n5) -- (n7);
				\draw (n6) -- (n7);
				
				\draw (n8) -- (n2);
				\draw (n8) -- (n9);
				\draw (n9) -- (n2);
				
				\end{tikzpicture}}} & 
		\subcaptionbox{${O}_{12}^{1}$\label{daylight}}{
			\scalebox{.53}{\begin{tikzpicture}[myNode/.style = hblack node]
				
				\node[myNode] (n1) at (0,0) {};
				\node[myNode] (n2) [above of=n1]  {};
				\coordinate (midway1) at ($(n1)!0.5!(n2)$)  {};
				\node[myNode] (n3) [right of=midway1] {};
				
				\node[myNode] (n4) [above of=n3] {};
				
				\coordinate (midway2) at ($(n3)!0.5!(n4)$)  {};
				\node[myNode] (n5) [right of=midway2] {};
				
				\node[myNode] (n6) [above right=2.5ex and 6ex of n5] {};
				\node[myNode] (n7) [below of=n6] {};
				
				\node[myNode] (n8) [below right=2.5ex and 6ex of n7] {};
				\node[myNode] (n9) [above of= n8] {};
				
				\draw (n1) -- (n2);
				\draw (n1) -- (n3);
				\draw (n2) -- (n3);
				\draw (n3) -- (n5);
				\draw (n4) -- (n5);
				\draw (n4) -- (n3);
				\draw (n5) -- (n6);
				\draw (n5) -- (n7);
				\draw (n6) -- (n7);
				\draw (n7) -- (n8);
				\draw (n7) -- (n9);
				\draw (n9) -- (n8);
				\draw (n7) -- (n3);

				\end{tikzpicture}}} \\
		
	\end{tabular}
	\vspace{-2mm}\caption{The set ${\cal O}^{1}$ of obstructions for ${\cal A}_{1}({\cal P})$ of vertex connectivity 1.}
	\label{dispense}
\end{figure}
%
%
\vspace{-5mm}
\begin{figure}[H]
	\centering
	\begin{tabular}{ c c c c c c c c c }
		
		\subcaptionbox{${O}_{1}^{2}$\label{connects}}{
			\scalebox{.53}{\begin{tikzpicture}[myNode/.style = hblack node]
				
				\node[myNode] (n1) at (2,3) {};
				\node[myNode] (n1) at (1,2) {};
				\node[myNode] (n1) at (3,2) {};
				\node[myNode] (n1) at (2,1) {};
				\node[myNode] (n1) at (1.5,2) {};
				\node[myNode] (n1) at (2.5,2) {};
				
				\draw (2,3) -- (1,2);
				\draw (2,3) -- (3,2);
				\draw (2,3) -- (1.5,2);
				\draw (2,3) -- (2.5,2);
				\draw (2,1) -- (1,2);
				\draw (2,1) -- (3,2);
				\draw (2,1) -- (1.5,2);
				\draw (2,1) -- (2.5,2);
				\draw (1,2) -- (1.5,2);
				\draw (2.5,2) -- (3,2);
				\end{tikzpicture}}} &
		\subcaptionbox{${O}_{2}^{2}$\label{handling}}{
			\scalebox{.53}{\begin{tikzpicture}[myNode/.style = hblack node]
				
				\node[myNode] (n1) at (1.7,3) {};
				\node[myNode] (n1) at (1,2.4) {};
				\node[myNode] (n1) at (2.4,2.4) {};
				\node[myNode] (n1) at (1.7,1.7) {};
				\node[myNode] (n1) at (1,1) {};
				\node[myNode] (n1) at (2.4,1) {};
				\node () at (2.6,1) {};
					
				\draw (1.7,3) -- (1,2.4);
				\draw (1.7,3) -- (2.4,2.4);
				\draw (1,2.4) -- (2.4,2.4);
				\draw (1,2.4) -- (1,1);
				\draw (1,2.4) -- (2.4,1);
				\draw (2.4,2.4) -- (1,1);
				\draw (2.4,2.4) -- (2.4,1);
				\draw (1,1) -- (2.4,1);
				\end{tikzpicture}}} &
		\subcaptionbox{${O}_{3}^{2}$\label{softened}}{
			\scalebox{.53}{\begin{tikzpicture}[myNode/.style = hblack node]
				\node[myNode] (n1) at (1.5,2.7) {};
				\node[myNode] (n1) at (1,2) {};
				\node[myNode] (n1) at (2,2) {};
				\node[myNode] (n1) at (1.5,0.3) {};
				\node[myNode] (n1) at (1,1) {};
				\node[myNode] (n1) at (2,1) {};
				\node[myNode] (n1) at (0.3,1.5) {};
				\node[myNode] (n1) at (2.7,1.5) {};
				
				\draw (1.5,2.7) -- (1,2);
				\draw (1.5,2.7) -- (2,2);
				\draw (1.5,0.3) -- (1,1);
				\draw (1.5,0.3) -- (2,1);
				\draw (0.3,1.5) -- (1,1);
				\draw (0.3,1.5) -- (1,2);
				\draw (2.7,1.5) -- (2,1);
				\draw (2.7,1.5) -- (2,2);
				\draw (1,2) -- (2,2);
				\draw (1,2) -- (1,1);
				\draw (2,2) -- (2,1);
				\draw (1,1) -- (2,1);
				\end{tikzpicture}}} &
		\subcaptionbox{${O}_{4}^{2}$\label{begrudge}}{
			\scalebox{.53}{\begin{tikzpicture}[myNode/.style = hblack node]
				
				\node[myNode] (n1) at (1,5) {}; 
				\node[myNode] (n1) at (3,5) {};
				\node[myNode] (n1) at (2,5.6) {};
				\node[myNode] (n1) at (2,6.6) {};
				\node[myNode] (n1) at (2,5) {}; 
				\node[myNode] (n1) at (1.5,4.5) {};
				\node[myNode] (n1) at (2.5,4.5) {};
				
				\draw (1,5) -- (3,5);
				\draw (1,5) -- (2,5.6);
				\draw (1,5) -- (2,6.6);
				\draw (3,5) -- (2,5.6);
				\draw (3,5) -- (2,6.6);
				\draw (2,5.6) -- (2,6.6);
				\draw (1.5,4.5) -- (1,5); 
				\draw (1.5,4.5) -- (2,5);
				\draw (2.5,4.5) -- (2,5);
				\draw (2.5,4.5) -- (3,5);
				
				\end{tikzpicture}}} &
		
		\subcaptionbox{${O}_{5}^{2}$\label{bringing}}{
			\scalebox{.53}{\begin{tikzpicture}[myNode/.style = hblack node]
				
%
				
				\node[myNode] (n1) at (0.8,5) {}; 
				\node[myNode] (n1) at (3.2,5) {};
				\node[myNode] (n1) at (2,5.5) {};
				\node[myNode] (n1) at (2,6.7) {};
				\node[myNode] (n1) at (2,4.5) {}; 
				\node[myNode] (n1) at (2.3,5.8) {};
				
				\draw (0.8,5) -- (3.2,5);
				\draw (0.8,5) -- (2,5.5);
				\draw (0.8,5) -- (2,6.7);
				\draw (3.2,5) -- (2,5.5);
				\draw (3.2,5) -- (2,6.7);
				\draw (2,5.5) -- (2,6.7);
				\draw (0.8,5) -- (2,4.5); 
				\draw (3.2,5) -- (2,4.5);
				\draw (2,6.7) -- (2.3,5.8);
				\draw (2,5.5) -- (2.3,5.8);
				\end{tikzpicture}}} &
		\subcaptionbox{${O}_{6}^{2}$\label{operetta}}{
			\scalebox{.53}{\begin{tikzpicture}[myNode/.style = hblack node]
				
%
\node[myNode] (n1) at (1,5) {}; 
\node[myNode] (n1) at (3,5) {};
\node[myNode] (n1) at (2,5.6) {};
\node[myNode] (n1) at (2,6.6) {};
\node[myNode] (n1) at (2,4.5) {};
\node[myNode] (n1) at (2.8,6.4) {};
\node[myNode] (n1) at (2.5,5.8) {};

\draw (1,5) -- (3,5);
\draw (1,5) -- (2,5.6);
\draw (1,5) -- (2,6.6);
\draw (3,5) -- (2,5.6);
\draw (3,5) -- (2,6.6);
\draw (2,5.6) -- (2,6.6);
\draw (2,4.5) -- (1,5);
\draw (2,4.5) -- (3,5);
\draw (2.8,6.4) -- (2.5,5.8);
\draw (2.8,6.4) -- (2,6.6);
				
				\end{tikzpicture}}} \\
		\subcaptionbox{${O}_{7}^{2}$\label{uttering}}{
			\scalebox{.53}{\begin{tikzpicture}[myNode/.style = hblack node]
				
				\node[myNode] (n1) at (2,0.5) {};
				\node[myNode] (n1) at (2,2.5) {};
				\node[myNode] (n1) at (2,1.5) {};
				\node[myNode] (n1) at (3,1) {};
				\node[myNode] (n1) at (4,1) {};
				\node[myNode] (n1) at (5,1.5) {};
				\node[myNode] (n1) at (3,2) {};
				\node[myNode] (n1) at (4,2) {};
				
				\draw (2,0.5) -- (3,1);
				\draw (2,2.5) -- (3,2);
				\draw (2,2.5) -- (2,1.5);
				\draw (2,0.5) -- (2,1.5);
				\draw (2,1.5) -- (3,1);
				\draw (2,1.5) -- (3,2);
				\draw (3,1) -- (4,1);
				\draw (3,2) -- (4,2);
				\draw (4,1) -- (5,1.5);
				\draw (4,2) -- (5,1.5);
				\draw (3,1) -- (3,2);
				\draw (4,1) -- (4,2);
				\end{tikzpicture}}} &
		\subcaptionbox{${O}_{8}^{2}$\label{furrowed}}{
			\scalebox{.53}{\begin{tikzpicture}[myNode/.style = hblack node]
				
				\node[myNode] (n1) at (1,7) {};
				\node[myNode] (n1) at (3,7) {};
				\node[myNode] (n1) at (2,7) {};
				\node[myNode] (n1) at (1.5,6) {};
				\node[myNode] (n1) at (2.5,6) {};
				\node[myNode] (n1) at (2,5) {};
				\node[myNode] (n1) at (2,6.3) {};
				
				\draw (1,7) -- (2,7);
				\draw (3,7) -- (2,7);
				\draw (1,7) -- (1.5,6);
				\draw (2.5,6) -- (3,7);
				\draw (2.5,6) -- (2,7);
				\draw (1.5,6) -- (2,7);
				\draw (2.5,6) -- (1.5,6);
				\draw (2.5,6) -- (2,5);
				\draw (1.5,6) -- (2,5);
				\draw (2,6.3) -- (1.5,6);
				\draw (2,6.3) -- (2.5,6);
				\draw (2,6.3) -- (2,7);
				\end{tikzpicture}}} &
		
		\subcaptionbox{${O}_{9}^{2}$\label{holiness}}{
			\scalebox{.53}{\begin{tikzpicture}[myNode/.style = hblack node]
				
				\node[myNode] (n1) at (0.2,2) {};
				\node[myNode] (n1) at (1,2.7) {};
				\node[myNode] (n1) at (1,1.3) {};
				\node[myNode] (n1) at (2,3.1) {};
				\node[myNode] (n1) at (2,0.9) {};
				\node[myNode] (n1) at (3,2) {};
				\node[myNode] (n1) at (4,2) {};
				
				\draw (0.2,2) -- (1,2.7);
				\draw (0.2,2) -- (1,1.3);
				\draw (1,2.7) -- (1,1.3);
				\draw (1,2.7) -- (2,3.1);
				\draw (1,1.3) -- (2,0.9);
				\draw (2,3.1) -- (3,2);
				\draw (2,3.1) -- (4,2);
				\draw (2,0.9) -- (3,2);
				\draw (2,0.9) -- (4,2);
				\draw (3,2) -- (4,2);
				\end{tikzpicture}}} &
		\subcaptionbox{${O}_{10}^{2}$\label{supports}}{
			\scalebox{.53}{\begin{tikzpicture}[myNode/.style = hblack node]
				
				\node[myNode] (n1) at (0.3,0.3) {};
				\node[myNode] (n1) at (1,0.3) {};
				\node[myNode] (n1) at (2.1,0.3) {};
				\node[myNode] (n1) at (0.7,1) {};
				\node[myNode] (n1) at (1.1,1.3) {};
				\node[myNode] (n1) at (1.5,1) {};
				\node[myNode] (n1) at (1.1,1.8) {};
				\node[myNode] (n1) at (2,-0.7) {};
				
				\draw (0.3,0.3) -- (1,0.3);
				\draw (0.3,0.3) -- (0.7,1);
				\draw (1,0.3) -- (0.7,1);
				\draw (1,0.3) -- (1.5,1);
				\draw (1,0.3) -- (2.1,0.3);
				\draw (1,0.3) -- (2,-0.7);
				\draw (2.1,0.3) -- (1.5,1);
				\draw (2.1,0.3) -- (2,-0.7);
				\draw (0.7,1) -- (1.1,1.8);
				\draw (0.7,1) -- (1.5,1);
				\draw (0.7,1) -- (1.1,1.3);
				\draw (1.1,1.3) -- (1.5,1);
				\draw (1.5,1) -- (1.1,1.8);
				\end{tikzpicture}}} &
		\subcaptionbox{${O}_{11}^{2}$\label{conceive}}{
			\scalebox{.53}{\begin{tikzpicture}[myNode/.style = hblack node]
				
				\node[myNode] (n1) at (1,7) {};
				\node[myNode] (n1) at (3,7) {};
				\node[myNode] (n1) at (2,7) {};
				\node[myNode] (n1) at (1.5,6) {};
				\node[myNode] (n1) at (2.5,6) {};
				\node[myNode] (n1) at (2,5) {};
				\node[myNode] (n1) at (2.5,6.65) {};
				\node[myNode] (n1) at (1.5,6.65) {};
				\node[myNode] (n1) at (2,5.65) {};

				\draw (1,7) -- (2,7);
				\draw (3,7) -- (2,7);
				\draw (1,7) -- (1.5,6);
				\draw (2.5,6) -- (3,7);
				\draw (2.5,6) -- (2,7);
				\draw (1.5,6) -- (2,7);
				\draw (2.5,6) -- (1.5,6);
				\draw (2.5,6) -- (2,5);
				\draw (1.5,6) -- (2,5);
				\draw (2,5.65) -- (1.5,6);
				\draw (2,5.65) -- (2.5,6);
				\draw (1.5,6.65) -- (1.5,6);
				\draw (1.5,6.65) -- (2,7);
				\draw (2.5,6.65) -- (2.5,6);
				\draw (2.5,6.65) -- (2,7);
				\end{tikzpicture}}} &
		
		\subcaptionbox{${O}_{12}^{2}$\label{typifies}}{
			\scalebox{.53}{\begin{tikzpicture}[myNode/.style = hblack node]
				
				\node[myNode] (n1) at (2,0.5) {};
				\node[myNode] (n2) at (2,2.5) {};
				\node[myNode] (n3) at (2,1.5) {};
				\node[myNode] (n4) at (3,1) {};
				\node[myNode] (n5) at (4,1) {};
				\node[myNode] (n6) at (5,1.5) {};
				\node[myNode] (n7) at (3,2) {};
				\node[myNode] (n8) at (4,2) {};
				
				\draw (n1) -- (n4);
				\draw (n2) -- (n7);
				\draw (n2) -- (n3);
				\draw (n1) -- (n3);
				\draw (n3) -- (n4);
				\draw (n3) -- (n7);
				\draw (n4) -- (n5);
				\draw (n7) -- (n8);
				\draw (n5) -- (n6);
				\draw (n8) -- (n6);
				\draw (n4) -- (n8);
				\draw (n5) -- (n8);
				\end{tikzpicture}}} \\
		
		\subcaptionbox{${O}_{13}^{2}$\label{solution}}{
			\scalebox{.53}{\begin{tikzpicture}[myNode/.style = hblack node]
				
				\node[myNode] (n1) at (0,0) {};
				\node[myNode] (n2) [right of=n1]  {};
				\coordinate (peak) at ($(n1)!0.5!(n2)$)  {};
				\node[myNode] (n3) [above of=peak] {};
				
				\node[myNode] (n4) [right of=n2]  {};
				\node[myNode] (n5) [right of=n4]  {};
				\coordinate (peak2) at ($(n2)!0.5!(n4)$)  {};
				\node[myNode] (n6) [above of=peak2] {};
				\coordinate (peak3) at ($(n4)!0.5!(n5)$)  {};
				\node[myNode] (n7) [above of=peak3] {};
				
				\draw (n1) -- (n2);
				\draw (n1) -- (n3);
				\draw (n2) -- (n3);
				\draw (n2) -- (n4);
				\draw (n2) -- (n6);
				\draw (n4) -- (n5);
				\draw (n4) -- (n7);
				\draw (n4) -- (n6);
				\draw (n3) -- (n6);
				\draw (n5) -- (n7);
				\draw (n6) -- (n7);
				
				\end{tikzpicture}}} &

		\subcaptionbox{${O}_{14}^{2}$\label{barracks}}{
			\scalebox{.53}{
				\begin{tikzpicture}[myNode/.style = hblack node]
				
				\node[myNode] (n1) at (0,0) {};
				\node[myNode] (n2) [above of=n1]  {};
				\coordinate (midway1) at ($(n1)!0.5!(n2)$)  {};
				\node[myNode] (n3) [right of=midway1] {};
				
				\node[myNode] (n4) [right=10ex of n3] {};
				\node[myNode] (n5) [above right=2ex and 5ex of n4] {};
				\node[myNode] (n6) [below=5ex of n5] {};
				
				\coordinate (midway2) at ($(n3)!0.5!(n4)$)  {};
				
				\node[myNode] (n7) [above right =10ex and 3.5ex of midway2] {};
				\node[myNode] (n8) [left =6.5ex of n7] {};
				\coordinate (midway3) at ($(n7)!0.5!(n8)$)  {};
				\node[myNode] (n9) [above =5ex of midway3] {};

				\draw (n1) -- (n2);
				\draw (n1) -- (n3);
				\draw (n2) -- (n3);
				
				\draw (n4) -- (n5);
				\draw (n4) -- (n6);
				\draw (n5) -- (n6);
				
				\draw (n9) -- (n8);
				\draw (n9) -- (n7);
				\draw (n7) -- (n8);

				\draw (n4) -- (n3);
				\draw (n8) -- (n2);
				\draw (n7) -- (n5);
				
				\end{tikzpicture}}}&
		
		\subcaptionbox{${O}_{15}^{2}$\label{straight}}{
			\scalebox{.53}{
				\begin{tikzpicture}[myNode/.style = hblack node]
				
				\node[myNode] (n1) at (0,0) {};
				\node[myNode] (n2) [above of=n1]  {};
				\coordinate (midway1) at ($(n1)!0.5!(n2)$)  {};
				\node[myNode] (n3) [left of=midway1] {};
				
				\node[myNode] (n4) [right=20ex of midway1] {};
				\node[myNode] (n5) [above left=2ex and 5ex of n4] {};
				\node[myNode] (n6) [below=5ex of n5] {};
				
				\node[myNode] (n7) [right=6.4ex of n2] {};
				\node[myNode] (n8) [below of=n7] {};

				\draw (n1) -- (n2);
				\draw (n1) -- (n3);
				\draw (n2) -- (n3);
				
				\draw (n4) -- (n5);
				\draw (n4) -- (n6);
				\draw (n5) -- (n6);
				
				\draw (n7) -- (n8);
				\draw (n2) -- (n7);
				\draw (n7) -- (n5);
				
				\draw (n8) -- (n6);
				\draw (n8) -- (n1);
				
				\end{tikzpicture}}}

	\end{tabular}
	\vspace{-2mm}\caption{The set ${\cal O}^{2}$ of  obstructions for ${\cal A}_{1}({\cal P})$ with vertex connectivity 2.}
	\label{animates}
\end{figure}
\vspace{-5mm}
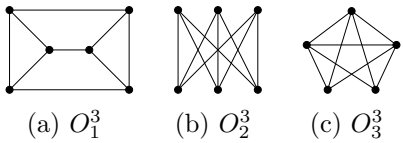
\begin{figure}[H]
	\centering
	\begin{tabular}{  c c c  }
		
		\subcaptionbox{${O}_{1}^{3}$\label{japanese}}{
			\scalebox{.53}{\begin{tikzpicture}[myNode/.style = hblack node]
				
				\node[myNode] (n1) at (1,3) {};
				\node[myNode] (n1) at (2,2) {};
				\node[myNode] (n1) at (1,1) {};
				\node[myNode] (n1) at (4,3) {};
				\node[myNode] (n1) at (3,2) {};
				\node[myNode] (n1) at (4,1) {};
				
				\draw (1,3) -- (2,2);
				\draw (1,1) -- (2,2);
				\draw (1,3) -- (4,3);
				\draw (1,3) -- (1,1);
				\draw (1,1) -- (4,1);
				\draw (2,2) -- (3,2);
				\draw (3,2) -- (4,1);
				\draw (3,2) -- (4,3);
				\draw (4,1) -- (4,3);
				
				\end{tikzpicture}}} &
		\subcaptionbox{${O}_{2}^{3}$\label{persuade}}{
			\scalebox{.53}{\begin{tikzpicture}[myNode/.style = hblack node]
				
				\node[myNode] (n1) at (1,3) {};
				\node[myNode] (n1) at (2,3) {};
				\node[myNode] (n1) at (3,3) {};
				\node[myNode] (n1) at (1,1) {};
				\node[myNode] (n1) at (2,1) {};
				\node[myNode] (n1) at (3,1) {};
				
				\draw (1,3) -- (1,1);
				\draw (1,3) -- (2,1);
				\draw (1,3) -- (3,1);
				\draw (2,3) -- (1,1);
				\draw (2,3) -- (2,1);
				\draw (2,3) -- (3,1);
				\draw (3,3) -- (1,1);
				\draw (3,3) -- (2,1);
				\draw (3,3) -- (3,1);
				
				\end{tikzpicture}}} &
		\subcaptionbox{${O}_{3}^{3}$\label{rigorous}}{
			\scalebox{.53}{\begin{tikzpicture}[myNode/.style = hblack node]
				
				\node[myNode] (n1) at (2,2.6) {};
				\node[myNode] (n1) at (0.875,1.75) {};
				\node[myNode] (n1) at (3.125,1.75) {};
				\node[myNode] (n1) at (1.4,0.625) {};
				\node[myNode] (n1) at (2.6,0.625) {};
				
				\draw (2,2.6) -- (0.875,1.75);
				\draw (2,2.6) -- (3.125,1.75);
				\draw (2,2.6) -- (1.4,0.625);
				\draw (2,2.6) -- (2.6,0.625);
				\draw (0.875,1.75) -- (3.125,1.75);
				\draw (0.875,1.75) -- (1.4,0.625);
				\draw (0.875,1.75) -- (2.6,0.625);
				\draw (3.125,1.75) -- (1.4,0.625);
				\draw (3.125,1.75) -- (2.6,0.625);
				
				\end{tikzpicture}}} \\
		
	\end{tabular}
	\vspace{-2mm}\caption{The set ${\cal O}^{3}$ obstructions for ${\cal A}_{1}({\cal P})$ with vertex connectivity $3.$}
	\label{machines}
\end{figure}

We set  ${\cal O} = {\cal O}^{0}\cup {\cal O}^{1} \cup {\cal O}^{2}\cup{\cal O}^{3}.$
Notice that since $\hyperref[persuade]{{O}_{2}^{3}}=K_{3,3}$ and $\hyperref[rigorous]{{O}_{3}^{3}}$ is a subgraph of $K_{5}$, all graphs in ${\cal O}\setminus\{\hyperref[persuade]{{O}_{2}^{3}},\hyperref[rigorous]{{O}_{3}^{3}}\}$ are  planar.
For the proof of~\autoref{ruthless}, we first note, by inspection, 
that $\obs({\cal A}_{1}({\cal P}))\supseteq  {\cal O}.$ As such an inspection might be 
 quite tedious to do manually for all the 33 graphs in ${\cal O}$, 
one may use a computer program that can do this in an automated way (see~\href{http://www.cs.upc.edu/~sedthilk/oapf/}{www.cs.upc.edu/$\sim$sedthilk/oapf/} for code in {\sf SageMath} that can do this). The main contribution of the paper is that ${\cal O}$ is a complete list, i.e., that $\obs({\cal A}_{1}({\cal P}))\subseteq  {\cal O}.$

Our proof strategy is to assume that there exists a graph $G\in  \obs({\cal A}_{1}({\cal P}))\setminus {\cal O}$  and gradually restrict the structure of $G$ by deriving contradictions to some of the conditions of the following observation.

\begin{observation}
	\label{selected}
	If $G\in \obs({\cal A}_{1}({\cal P}))\setminus  {\cal O}$ then $G$ satisfies the following conditions:
	\begin{enumerate}
		\item $G\not\in {\cal A}_{1}({\cal P}),$
		\item if $G'$ is a  minor of $G$ that is different than $G$,  then $G'\in{\cal A}_{1}({\cal P}),$ and 
		\item none of the graphs in ${\cal O}$ is a minor of $G.$
	\end{enumerate}
\end{observation}

The rest of the paper is dedicated to the proof of \autoref{ruthless} and is organized as follows.
In~\autoref{supplied} we give the basic definitions and some preliminary results. In \autoref{fissures} 
we prove some auxiliary results that restrict the structure of the graphs in $\obs({\cal A}_{1}({\cal P}))\setminus  {\cal O}$.
In \autoref{refusing} we use the results of \autoref{fissures} in order to, first, prove that 
graphs in $\obs({\cal A}_{1}({\cal P}))\setminus  {\cal O}$ are biconnected  (\autoref{destined})  and, 
next, prove that the 
graphs in $\obs({\cal A}_{1}({\cal P}))\setminus  {\cal O}$ are triconnected (\autoref{personal}). The proof of \autoref{ruthless}
follows from the fact that 
every triconnected  graph either contains a graph in $ {{\cal O}^{3}}$ 
or it is a graph in ${\cal A}_{1}({\cal P})$ (\autoref{reliably}, proved in~\autoref{supplied}).

\section{Definitions and preliminary results}
\label{supplied}

\textbf{Sets and integers.}
We denote by $\Nbb$ the set of all non-negative integers and we
set $\Nbb^+=\Nbb\setminus\{0\}.$
Given two integers $p$ and $q,$  we set $\intv{p,q}=\{p,\ldots,q\}$
and given a $k\in\Nbb^+$ we denote $[k]=[1,k].$ Given a set $A,$ we denote by $2^{A}$ the set of all its subsets and we define 
$\binom{A}{2}:=\{e\mid e\in 2^{A}\wedge |e|=2\}.$
If ${\cal S}$ is a collection of objects where the operation $\cup$ is defined, then we  denote $\cupall{\cal S}=\bigcup_{X\in{\cal S}}X.$

\bigskip
\noindent
\textbf{Graphs}.   Given a graph $G,$ we denote by $V(G)$ the set of vertices of  $G$ and by $E(G)$ the set of the edges of $G.$ For an edge $e=\{x,y\}\in E(G),$ we use instead the notation $e=xy,$ that is equivalent to  $e=yx.$
{Given a vertex $v \in V(G),$ we define the \emph{neighborhood} of
	$v$ as $N_G(v) = \{u \mid u \in V(G), \{u,v\} \in E(G)\}$ and the \emph{closed neighborhood}
	of $v$ as $N_G[v] = N_G(v) \cup \{v\}.$} If $X\subseteq V(G),$ then we write $N_{G}(X)=(\bigcup_{v\in X}N_{G}(v))\setminus X.$
The {\em  degree} of a vertex $v$ in $G$ is defined as $\deg_{G}(v)=|N_{G}(v)|.$
We define $\delta(G)=\min\{\deg_{G}(x)\mid x\in V(G)\}.$
Given two graphs $G_{1},G_{2}$, we define the \emph{union} of $G_{1},G_{2}$ as
the graph $G_{1}\cup G_{2}=(V(G_{1})\cup V(G_{2}), E(G_{1})\cup E(G_{2}))$ and the \emph{intersection} of
$G_{1},G_{2}$ as the graph $G_{1}\cap G_{2}=(V(G_{1})\cap V(G_{2}), E(G_{1})\cap E(G_{2}))$.
A \emph{subgraph} of a graph $G$ is every graph $H$
where $V(H)\subseteq V(G)$ and $E(H)\subseteq E(G).$
If $S \subseteq V(G),$ the subgraph of $G$ \emph{induced by} $S,$ denoted by $G[S],$ is the graph $(S, E(G) \cap \binom{S}{2}).$
We also define $G \setminus S$ to be the subgraph of $G$ induced by $V(G) \setminus S.$
If $S \subseteq E(G),$ we denote by $G \setminus S$ the graph $(V(G), E(G) \setminus S).$
Given a vertex $x\in V(G)$ we define $G\setminus x=G\setminus \{x\}$ and given an edge $e\in E(G)$ we define $G\setminus e=G\setminus \{e\}.$

\paragraph{Paths and separators.}
If $s,t\in V(G)$ are two distinct vertices, an {\em  $(s,t)$-path} of $G$
is any connected subgraph $P$ of $G$ with maximum degree at most $2,$ where $\deg_{P}(s) = 1$ and $\deg_{P}(t) = 1.$
If $s\in V(G)$, an {\em  $(s,s)$-path} of $G$
is the subgraph of $G$ consisting of the single vertex $s$.
The {\em distance} between $s$ and $t$ in $G$ is the minimum number of edges 
of an $(s,t)$-path in $G.$ 
Given a path $P,$ we say that $v \in V(P)$ is an \emph{internal vertex} of $P$ if $\deg_{P}{(v)} = 2,$ while if $\deg_{P}{(v)} = 1$ we say that $v$ is a {\em terminal vertex} of $P.$
We say that two paths $P_{1}$ and $P_{2}$ in $G$ are {\em internally vertex disjoint}
if none of the internal vertices of the one is an internal vertex of the other.
Given an integer $k$ and a graph $G,$ we say that $G$ is {\em $k$-connected} if for each $\{u,v\} \in \binom{V(G)}{2},$
there exist  $k$ pairwise internally vertex disjoint $(u,v)$-paths of $G,$ say $P_{1},\ldots,P_{k},$ such that
for each $\{i,j\}\in\binom{[k]}{2},$  $P_{i} \not = P_{j},$ $V(P_{i}) \cap V(P_{j}) = \{u,v\}.$
We call 2-connected graphs {\em  biconnected} and 3-connected graphs {\em  triconnected}.
Given a set $S\subseteq V(G),$ we say that $S$ is a {\em  separator} of $G$ if $G$ has fewer connected components than $G\setminus S.$ We call a separator of size $k$ a {\em k-separator}. Notice that, by Menger's theorem,
a graph is $k$-connected iff it does not contain a separator of size less than $k.$

We say that $e\in E(G)$ is a {\em  bridge} if $G$ has fewer connected components than $G\setminus e.$
A graph that does not contain any bridge is called {\em  bridgeless}.
 A {\em  block} of a graph $G$ is either a bridge or a maximal biconnected subgraph of $G$.
 A block of a graph $G$ is called {\em non-trivial} if it is not a bridge.

A vertex $v\in V(G)$ is a {\em  cut-vertex} of $G$ if $\{v\}$ is a separator of $G.$ 
We also say that $S$ is a {\em  rich separator}
if $G\setminus S$ has at least 2 more connected components than $G.$

\paragraph{Special graphs.}
By $K_{r}$ we denote the complete graph on $r$ vertices. Similarly,  by $K_{r_{1},r_{2}}$ we denote the complete bipartite graph of which one part has $r_{1}$ vertices and the other $r_{2}.$ We denote by $K_{r}^{-}$ the graph obtained by $K_{r}$ after removing any edge. 

For an $r\geq 3,$ we denote by $C_{r}$ the connected graph on $r$ vertices of degree 2 (i.e., the cycle on $r$ vertices).
If $G$ is a graph and $C$ is a subgraph of $G$ isomorphic to $C_{r}$ for some $r\geq 3$,
then an edge $e=\{u,v\}\in E(G)\setminus E(C)$ where $u,v\in V(C)$ is called \emph{chord} of $C$. 

For $r\geq 3,$ the {\em $r$-wheel}, denoted by $W_{r},$ is the graph obtained by adding a new vertex $v_{\rm new},$ called the {\em  central vertex} of $W_{r},$ to $C_{r}$ 
along with edges, called {\em spokes}, connecting each vertex of $C_{r}$ with $v_{\rm new}.$ The subgraph $W_{r}\setminus v_{\rm new}$ is called the {\em circumference} of $W_{r}.$

{We denote by $Z$ the {\em butterfly} graph, that is the graph obtained by the disjoint union of two $K_{3}$ after identifying two of their vertices (see rightmost figure of \autoref{ostructionspse}).}

A graph $G$ is {\em  outerplanar} if it can be embedded in the plane so that there's no crossing edges and all its vertices lie on the same face.
It is known that the obstruction set of the class of outerplanar 
graphs is $\{K_{2,3},K_{4}\}$. The outer face of such an embedding contains every vertex of $G.$ Thus, we can observe the following:

\begin{observation}\label{decoding}
	If $G$ is biconnected and outerplanar then $G$ contains a {\em Hamiltonian cycle}, i.e. a cycle that contains every vertex of $G.$
\end{observation}


\paragraph{Minors and topological minors.} We define $G / e,$ the graph obtained from the graph $G$ by {\em contracting} an edge $e = xy \in E(G),$ to be the graph obtained by replacing the edge $e$ by a new vertex $v_{e}$ which becomes adjacent to all neighbors of $x,y$ (apart from $y$ and $x$) and deleting vertices $x,y$.  
Given two graphs $H$ and $G$ we say that $H$ is a {\em  minor} of $G,$ denoted by $H\leq G,$ if $H$ can be obtained by some subgraph of $G$ after contracting edges.

Given a set ${\cal H}$ of graphs, we write ${\cal H}\leq G$
to denote that there exists an $H\in{\cal H}$ such that $H\leq G$ and we defined $\excl({\cal H})=\{G\mid {\cal H}\nleq G\}$. 
If $H\nleq G$, then we say that {\em $G$ is $H$-minor free}, or, in short, {\em $H$-free}. Also, given a graph $G$ and a set of graphs ${\cal H}$
we say that $G$ is ${\cal H}$-free if it is $H$-free, for each $H\in {\cal H}$.
Given a graph class ${\cal G}$ we say that ${\cal G}$ is {\em  minor-closed} if for every graph $H$ such that $H\leq G$ and $G\in {\cal G}$, it holds that $H\in{\cal G}.$
We also define $\obs({\cal G})$ as the set of all minor-minimal graphs that do not belong in ${\cal G}$
and we call $\obs({\cal G})$ the {\em obstruction set} of the class ${\cal G}.$

If $e=xy$ is an edge of a graph $G$ then the operation of replacing $e$ by a path of length $2,$ i.e two edges
$\{x,v_{e}\},\{v_{e},y\},$ where $v_{e}$ is a new vertex,
is called {\em subdivision} of $e.$ A graph $G$ is called a {\em subdivision} of a graph
$H$ if $G$ can be obtained from $H$ by repeatedly subdividing edges, i.e. by replacing 
some edges of $H$ with new paths between its endpoints, so that the intersection of any 
two such paths is either empty or a vertex of $H.$
The original vertices of $H$ are called {\em branch vertices},
while the new vertices are called {\em subdividing vertices}.
If a graph $G$ contains a subdivision of $H$ as a subgraph, then $H$ is a {\em topological minor} of $G.$
It is easy to see that if $H$ is a topological minor of $G$ then it is also a minor of $G.$

Let $G$ be a subdivision of some wheel $W_r$. In keeping with the notation previously introduced for wheels, we define the {\em spokes} of $G$ to be the paths of $G$ produced by the subdivision of the spokes of $W_r$ and similarly we define the {\em circumference} of $G$ to be the cycle of $G$ produced by the subdivision of the circumference of $W_r.$

The following is an easy consequence of Dirac's Theorem~\cite{Dirac52apro}, stating 
that if $\delta(G)\geq 3,$ then $G$ contains $K_{4}$ as a minor.
\begin{proposition}\label{hegelian}
	Let $G$ be a biconnected outerplanar graph with at least 3 vertices. Then there exist at least two vertices of $G$ of degree 2.
\end{proposition}

\paragraph{\bf Triconnected components.}
Let $q\in\Bbb{N}^{+}$. Let $G$ be a graph and $S\subseteq V(G)$ and let $V_{1},\ldots,V_{q}$ be  the vertex sets of the connected components of $G\setminus S.$ We define ${\cal C}(G,S)=\{G_{1},\ldots,G_{q}\}$ where, for $i\in [q],$
$G_{i}$ is the graph obtained from $G[V_{i}\cup S]$ if we add all edges  between vertices in $S.$ We call the members of the set ${\cal C}(G,S)$ {\em augmented connected components}. Given a vertex $x\in V(G)$ we define ${\cal C}(G,x) = {\cal C}(G,\{x\}).$

It is easy to observe the following:

\begin{observation}\label{bicoobs}
Let $G$ be a biconnected graph and $S$ be a 2-separator of $G$. Then every $H\in {\cal C}(G,S)$ is biconnected.
\end{observation}

To get some intuition why the above observation holds, we set $S=\{x,y\}$ and notice that, for
every pair $u,v$ of vertices of $H\setminus S$,
among two vertex-disjoint $(u,v)$-paths in $G$, at most one ``exits'' $H$ and this path can be ``represented'' by the edge $xy$ of $H$.

Given a graph $G,$ the set ${\cal Q}(G)$ of its {\em  triconnected components} is recursively defined as follows:
\begin{itemize}
	\item If $G$ is triconnected or a complete graph on at most $3$ vertices, then ${\cal Q}(G)=\{G\}.$
	\item If $G$ contains a separator $S$ where $|S|\leq 2,$ then ${\cal Q}(G)=\bigcup_{H\in{\cal C}(G,S)}{\cal Q}(H).$
\end{itemize}

Notice that all graphs in ${\cal Q}(G)$ are either complete graphs on at most 3 vertices 
or triconnected graphs (graphs without any separator of size less than 3).
The study of triconnected components of  plane graphs
dates back to the work of Saunders Mac Lane in~\cite{maclane1937} (see also~\cite{Tutte66}).
{He also proved that the set ${\cal Q}(G)$ is uniquely determined (up to isomorphism), although different choices of the 2-separators may have been chosen in its recursive definition.}
\begin{observation}
	\label{supplier}
	Let $G$ be a graph. All graphs in ${\cal Q}(G)$ are topological minors of $G.$
\end{observation}

Let $G$ be a graph and  $v\in V(G)$ where $\deg_{G}(v)\geq 4.$ Let also ${\cal P}_{v}=\{A,B\}$ be a partition of $N_{G}(v)$
such that $|A|,|B|\geq 2.$  We define the {\em  ${\cal P}_{v}$-split
	of $G$} to be the graph $G'$  obtained by adding,
in the graph $G\setminus v,$ two new adjacent vertices $v_{A}$ and $v_{B}$ and making $v_{A}$ adjacent to the vertices of $A$ and $v_{B}$ adjacent to the vertices of $B.$ If $G'$ can be obtained by some ${\cal P}_{v}$-split
of $G,$ we say that $G'$ is a {\em  splitting} of $G.$
\begin{observation}
	\label{required}
	If $G'$ is a splitting of $G$ then $G$ is a minor of $G'.$
\end{observation} 

\begin{proposition}[Tutte~\cite{Tutte61athe}]
	\label{marching}
	A graph $G$ is triconnected if and only if there is a sequence of triconnected graphs $G_{0},\ldots,G_{q}$ such that 
	$G_{0}$ is isomorphic to $W_{r}$ for some $r\geq 3,$ $G_{q}=G,$ and for $i\in[q],$ $G_{i}$ is a splitting of $G_{i-1}$ or there exists an $e\in E(G_{i})$ such that $G_{i-1}=G_{i}\setminus e.$
\end{proposition}

The next proposition is a direct consequence of~\autoref{supplier} and \autoref{marching}.

\begin{proposition}
	\label{shimmers}
	Let $G$ be a graph. $K_{4}\not\leq G$ if and only if none of the graphs in ${\cal Q}(G)$ is triconnected.
\end{proposition}

%
%

\begin{lemma}
	\label{reliably}
	If $G$ is a triconnected graph that is not isomorphic to $W_{r},$ for every $r\geq 3,$ then $ {{\cal O}^{3}}\leq G.$
\end{lemma}

\begin{proof}
	Let $G$ be a triconnected graph not isomorphic to a wheel. By \autoref{marching}, there exists a sequence of graphs
	\begin{equation*}
	W_{r}=G_{0},G_{1},\ldots ,G_{q}=G
	\end{equation*}
	for some $r\geq 3,$ such that for every $i\in\intv{q},$ $G_{i}$ is a splitting of $G_{i-1}$ or there
	exists an edge $e\in E(G_{i})$ such that $G_{i-1}=G_{i}\setminus e.$
	Observe that, since $G\not\cong W_{r}$, we have that $q\geq 1.$ 
	Also observe that $r\geq 4,$
	since if $r=3$ then $q=0$ due to the fact that none of the vertices of $W_{3}$ can be split and all of them are adjacent to one another.
	
	Let now $z$ be the central vertex of $W_r$ and $C_r=W_r\setminus z.$ We examine how the graph $G_{1}$ may occur from $W_{r}.$ 
	For that, we distinguish the following two cases and our strategy, in both cases, is to prove that ${\cal O}^{3}\leq G$.
	\smallskip
	
	\noindent{\em  Case 1:} There exists an edge $e\in E(G_{1})$ such that $W_{r}= G_{1}\setminus e.$
	Let $e=uv$ for some $u,v\in V(G_{1})=V(W_{r}).$
	Since every vertex in $V(C_{r})$ is in the neighborhood of $z$ in $W_r$ then $u\neq z$ and $v\neq z.$
	Now, since $u,v$ are not adjacent vertices in $W_{r}$, there exists an internal vertex in each of the two $(u,v)$-paths of the graph $C_r,$ say $x,y,$ respectively (see \autoref{ulterior}).
	Therefore, by contracting each of the $(x,u),(x,v),(y,v),(y,u)$-paths of $C_r$ to an edge we get $\hyperref[rigorous]{{O}_{3}^{3}}$ as a minor of $G_{1}.$ 
	Now, by \autoref{required}, $G_{1}$ is a minor of $G$ and therefore $\hyperref[rigorous]{{O}_{3}^{3}}\leq G$.
	
	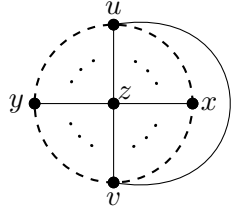
\begin{figure}[H]
		\centering
		\begin{tikzpicture}[every node/.style = black node, circle dotted/.style={dash pattern=on .05mm off 2mm,line cap=round}, scale = 0.7]

		\draw[dashed, thick] (0,0) circle (1.5);
		\foreach \t in {1, ..., 4} {
			\draw (\t * 90:1.5) node {} -- (0,0);
			\draw[line width = 0.4mm,circle dotted] (\t * 90 + 25:0.9) arc (\t * 90 + 25: \t * 90 + 65:0.9);
		}
		\draw (0,0) node {};
		\draw (0,0) node[label=45:$z$] {};
		\draw (90:1.5) node[label=90:$u$] {};
		\draw (180:1.5) node[label=180:$y$] {};
		\draw (270:1.5) node[label=270:$v$] {};
		\draw (360:1.5) node[label=360:$x$] {};
		\draw (90:1.5) .. controls +(10:3) and +(350:3) .. (270:1.5) {};
		
		\end{tikzpicture}
		\vspace{-2mm}\caption{The structure of the graph $G$ in Case 1.}\label{ulterior}
	\end{figure}
	
	\smallskip
	
	\noindent{\em  Case 2:} $G_{1}$ is a splitting of $W_{r}.$
	Observe that $G_{1}$ is a splitting of $W_{r}$ obtained by a ${\cal P}_{z}$-split. 
	So, let ${\cal P}_{z}=\{A,B\}$ and $v_{A},v_{B}$ the new adjacent vertices of $G_{1},$ where $N_{G_{1}}(v_{A}) = A$ and $N_{G_{1}}(v_{B}) = B.$
	We have that $|A|,|B|\geq 2$ and so there exist $x_{1},y_{1}\in A$ and $x_{2},y_{2}\in B.$
	We now distinguish the following subcases:
	\smallskip
	
	\noindent{\em Subcase 2.1:} One of the two $(x_{1},y_{1})$-paths in $C_r$ contains both of $x_{2},y_{2}$ (see leftmost figure of \autoref{wrapping}).
		This implies that $\hyperref[japanese]{{O}_{1}^{3}}\leq G_{1}$ and, as in case 1, it follows that $\hyperref[japanese]{{O}_{1}^{3}} \leq G$.
	\smallskip
		
	\noindent{\em Subcase 2.2:} Each one of the two $(x_{1},y_{1})$-paths in $C_r$ contains exactly one of $x_{2},y_{2}$ (see rightmost figure of \autoref{wrapping}).
		This implies that $\hyperref[persuade]{{O}_{2}^{3}}\leq G_{1}$ and, as in case 1, it follows that $\hyperref[persuade]{{O}_{2}^{3}}\leq G$.
	
	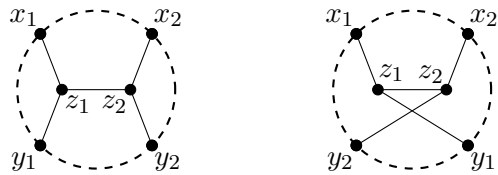
\begin{figure}[H]
		\centering
		\begin{tikzpicture}[every node/.style = black node, scale = 0.7,circle dotted/.style={dash pattern=on .05mm off 2mm,line cap=round}]
		
		\begin{scope}
		\draw[dashed, thick] (0,0) circle (1.5);
		\foreach \t in {1, 2} {
			\draw (\t * 90+45:1.5) node {}-- (-0.65,0);
			\draw (\t* 90+225:1.5) node {}-- (0.65,0);
		}
		\draw (135:1.5) node[label=135:$x_{1}$] {};
		\draw (225:1.5) node[label=225:$y_{1}$] {};
		\draw (315:1.5) node[label=315:$y_{2}$] {};
		\draw (45:1.5) node[label=45:$x_{2}$] {};
		\draw (180:0.65) node[label=324:$z_{1}$] {} -- (360:0.65) node[label=216:$z_{2}$] {};
		\end{scope}
		
		\begin{scope}[xshift = 6cm]
		\draw[dashed, thick] (0,0) circle (1.5);
		
		\foreach \t in {1, 2} {
			\draw (\t * 180+315:1.5) node {}-- (-0.65,0);
			\draw (\t* 180+45:1.5) node {}-- (0.65,0);
		}
		\draw (135:1.5) node[label=135:$x_{1}$] {};
		\draw (225:1.5) node[label=225:$y_{2}$] {};
		\draw (315:1.5) node[label=315:$y_{1}$] {};
		\draw (45:1.5) node[label=45:$x_{2}$] {};
		\draw (180:0.65) node[label=67.5:$z_{1}$] {} -- (360:0.65) node[label=135:$z_{2}$] {};
		\end{scope}
		
		\end{tikzpicture}
		\vspace{-2mm}\caption{The structure of the graph $G$ in the two Subcases of Case 2.}\label{wrapping}
	\end{figure}

	Since we have exhausted all possible cases for $G_{1}$ we conclude that ${{\cal O}^{3}}\leq G$.
\end{proof}

\paragraph{Disconnected obstructions.} 


We now prove that every disconnected graph in $\obs({\cal A}_{1}({\cal P}))$ is in ${\cal O}^{0}$.

\begin{lemma}
	\label{hundreds}
	If $G\in \obs({\cal A}_{1}({\cal P}))\setminus  {\cal O}$, then $G$ is connected.
\end{lemma}


 \begin{proof}
 Suppose, to the contrary, that $G$ is not connected.
 Notice that, since $G\not\in {\cal A}_{1}({\cal P})$, there exists a connected component $H$ of $G$ that contains at least two cycles. Also notice that, due to
 $\{\hyperref[magnetic]{{O}_{1}^{0}},\hyperref[consiste]{{O}_{2}^{0}},\hyperref[narcotic]{{O}_{3}^{0}}\}$-freeness of $G$, $H$ is the unique connected component of $G$ that contains at least two cycles.
 Now, since $G$ is not connected, $H$ is a minor of $G$ that is not isomorphic to $G$. This together with the fact that $G\in \obs({\cal A}_{1}({\cal P}))$  implies that $H\in {\cal A}_{1}({\cal P})$ and therefore there exists a vertex $v\in V(H)$ such that $H\setminus v\in {\cal P}$.
 But then, since every connected component of $G$ different than $H$ contains at most one cycle, we have that $G\setminus v\in {\cal P}$ which implies that $G\in {\cal A}_{1}({\cal P})$, a cotnradiction.
 \end{proof}


\section{Auxiliary lemmata}
\label{fissures}




%

By \autoref{hundreds}, 
we know that a graph $G\in \obs({\cal A}_{1}({\cal P}))\setminus  {\cal O}$ 
should be connected. In this section we prove a series of lemmata
that  further restrict the structure of  the graphs in $\obs({\cal A}_{1}({\cal P}))\setminus  {\cal O}$.

\subsection{General properties of the obstructions}

\label{obtained}

Given a graph $G$ and a vertex $v\in V(G)$ we say that $v$ is {\em  simplicial} if $G[N_{G}(v)]$ is isomorphic to $K_{r}$ for $r = \deg_{G}(v).$

Given a graph class ${\cal G}$, a graph $G$, and a vertex $x$, where $G\setminus x\in{\cal G},$ then we say that $x$ is a {\em ${\cal G}$-apex} of $G.$

\begin{lemma}
	\label{generals}
	If $G\in \obs({\cal A}_{1}({\cal P}))\setminus  {\cal O}$ then
	\begin{enumerate}
		\item $\delta(G)\geq 2,$
		\item $G$ is bridgeless, and
		\item all its vertices of degree $2$ are simplicial.
	\end{enumerate}
\end{lemma}

\begin{proof}
	(1) Consider a vertex $u\in V(G)$ with $\deg_{G}(u) < 2$. If $G\setminus u\in {\cal A}_{1}({\cal P})$, then also $G\in {\cal A}_{1}({\cal P})$, since $u$ does not participate in a cycle, a contradiction.
	
	(2)  Consider an edge $e=xy$ that is a bridge of $G$. By~\autoref{hundreds}, $G$ is connected.
	Since $e$ is a bridge, then $G\setminus e$ contains two connected
	components $H_{1},H_{2}$, such that $x\in V(H_{1})$ and $y\in V(H_{2})$.
	Observe that by $ {{\cal O}^{0}}$-freeness of $G$, one of $H_{1},
	H_{2}$, say $H_{1}$, contains at most one cycle.
	
	Consider the graph $G'= G / e$ and let $v_{e}$ be the vertex formed
	by contracting $e$.
	We denote $H_{1}^{\prime}, H_{2}^{\prime}$ the	graphs obtained from $H_{1}, H_{2}$ by replacing the vertices $x,y$ with $v_{e}$, respectively.
	Observe that $H_{1}^{\prime}$ also contains at most one cycle.
	By minor-minimality of $G$, it follows that $G'\in {\cal A}_{1}({\cal P})$ and
	therefore there exists some $u\in V(G')$ that is a ${\cal P}$-apex of $G'$.
	So, if $u\in V(H_{1}^{\prime})$ then $v_{e}$ is also a ${\cal P}$-apex of $G'$.
	Therefore we consider the case that $u\in V(H_{2}^{\prime})$. If
	$u=v_{e}$, then every connected component of $H_{2}^{\prime}\setminus v_{e}$ contains at most one cycle.
	Since $H_{2}^{\prime}\setminus v_{e} = H_{2}\setminus y$, it follows that
	every connected component of $G\setminus y$ contains at most one cycle, a contradiction.
%
%
In the case that $u\not=v_{e}$ we consider the augmented connected component $Q'\in {\cal C}(G', u)$ that contains $v_{e}$ and observe that, since  $u$  is a ${\cal P}$-apex of $G'$, $Q'$ contains at most one cycle.
%
%
%
This implies that if $Q$ is the augmented connected component of ${\cal C}(G, u)$ that contains $e$, then $Q$ also contains at most one cycle.
%
%
The latter, together with the fact that ${\cal C}(G', u)\setminus\{Q'\}={\cal C}(G, u)\setminus \{Q\}$ implies that $G\setminus u\in {\cal P}$, a contradiction.

	(3) Suppose, to the contrary, that there exists a non-simplicial vertex  $v  \in V(G)$ of degree 2,
	and let $e \in E(G)$ be an edge incident to $v,$ i.e. $e= uv$ for some $u\in V(G).$ By minor-minimality of $G,$ we have that $G' := G / e \in {\cal A}_{1}({\cal P}).$ 
	Let $x$ be an {${\cal P}$-apex} vertex of $G'$ and $v_{e}$ the vertex formed by contracting $e.$ Observe that, every cycle in $G$ that contains $v$ also contains $u$
	and so if $x = v_{e}$ then $u$ is an {${\cal P}$-apex} vertex of $G,$ a contradiction. Therefore, $x \not= v_{e}$ and so $x \in V(G).$
	Since $v$ is a non-simplicial vertex, the contraction of $e$ can only shorten cycles and not destroy them.
	Hence, $x$ is an {${\cal P}$-apex} vertex of $G,$ a contradiction.
\end{proof}

\medskip

For a graph $G\in \obs({\cal A}_{1}({\cal P}))\setminus  {\cal O}$, observe that, due to \autoref{generals}, all of its connected components and blocks contain a cycle.
Therefore, for the rest of the paper, we always assume that blocks are non-trivial.
Moreover, for such $G$, all graphs in ${\cal Q}(G)$ are either triconnected or isomorphic to $K_{3}$.

\subsection{Properties of obstructions containing a \texorpdfstring{$K_{4}$}{K4}}\label{daughter}

We now prove some results which will be useful in the main section of the proof.

\begin{lemma}
	\label{fanfares}
	If $G$ is a biconnected graph such that ${\cal O}\nleq G$, then there exists at most one triconnected graph in ${\cal Q}(G).$
\end{lemma}

\begin{proof}
	Suppose, to the contrary, that there are at least two triconnected graphs in ${\cal Q}(G)$
	and let $H_{1}, H_{2}$ be two of them. 
	Due to the recursive definition of ${\cal Q}(G)$ and by \autoref{supplier}, 
	there exists a separator $S$ such that $H_{1}, H_{2}$ are topological minors
	of some $G_{1},G_{2} \in {\cal C}(G,S),$ respectively. By the
	biconnectivity of $G$ we have that $S$ is a $2$-separator of $G$. Let $S=\{x,y\}$.
	Since $H_{1},H_{2}$ are triconnected graphs and topological minors of $G$,
	\autoref{reliably} implies that each $H_{i}, \ i\in [2]$ is isomorphic to a wheel.
	Thus, $K_{4}$ is a topological minor of both $H_{1}$ and $H_{2}$.
	
	
	Let $R_{1}, R_{2}$ be the subdivisions of $K_{4}$ in $G_{1},G_{2},$ respectively.
	\medskip
	

	\noindent{\em Claim:} For each $G_{i}, i\in\{1,2\}$, there is a subgraph $Q_{i}$  of $G_{i}$ such that the following hold:
	\begin{enumerate}
		\item $Q_{i}$ is a subdivision of $K_{4}$ and
		
		\item there exists an edge $e$ of $K_{4}$ such that $x$ and $y$ are vertices of the path of $Q_{i}$ corresponding to $e$.
	\end{enumerate}

\noindent{\em Proof of Claim:}
	Let $i\in \{1,2\}$.
	By Menger's theorem, there exist two disjoint paths from the separator $S$ to $R_{i}$.
		Let $P, P'$ be the shortest such paths. Now, let $R_{i}^{+}$ be the graph $R_{i} \cup P\cup P'\cup \{\{x,y\},\{xy\}\}$, that is the graph $R_{i}$ together with the paths $P,P'$ and the edge $xy$. 
Clearly, $R_{i}^{+}$ is a subgraph of $G_{i}$.
We now describe how to obtain a graph $Q_{i}$ that is a subgraph of $R_{i}^{+}$ and is a subdivision of $K_{4}$, where the path $P\cup P'\cup \{\{x,y\},\{xy\}\}$ is a subdivided edge of $K_{4}$.

Let $z$, $z'$ be the vertices of $R_{i}$ that are endpoints of $P$, $P'$ respectively. For every edge $e$ of $K_{4}$, let $P_{e}$  be the path of $R_{i}$ corresponding to $e$.

We distinguish the following two cases  based on whether every $(z,z')$-path in $R_{i}$ contains a branch vertex as an internal vertex or not.
\smallskip

\noindent{\em Case 1:} There is a $(z,z')$-path in $R_{i}$ such that none of its internal vertices is a branch vertex of $K_{4}$.\smallskip

In this case, there exists a subdivided edge $e$ of $K_{4}$ such that  $z,z'$ are vertices of $P_{e}$. We set $Q_{i}$ to be the graph obtained from $R_{i}^{+}$ by removing all internal vertices of the $(z,z')$-subpath of $P_{e}$ (see \autoref{bananafig1}).
\begin{figure}[H]
	\centering
	\includegraphics[width=5.5cm]{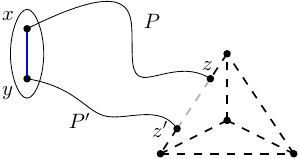}
	\caption{The graph $R_{i}^{+}$ in the case where there is a $(z,z')$-path in $R_{i}$ such that none of its internal vertices is a branch vertex of $K_{4}$. The graph $Q_{i}$ is obtained from $R_{i}^{+}$ after removing all internal vertices of the grey dashed path.}\label{bananafig1}
\end{figure}

\noindent{\em Case 2:}
Every $(z,z')$-path in $R_{i}$ contains a branch vertex as an internal vertex.\smallskip

In this case, one of $z,z'$, say $z$, is a subdividing vertex of $K_{4}$. Let $e$ be the subdivided edge of $K_{4}$ that contains $z$. Notice that $z'$ is either a vertex of $P_{e'}$, where $e'$ is the edge of $K_{4}$ such that $e\cap e' =\emptyset$, or 
an internal vertex of $P_{e''}$, where $e''$ is an edge of $K_{4}$ such that $|e\cap e' |=1$.

If $z'$ is a vertex of $P_{e'}$, where $e'$ is the edge of $K_{4}$ such that $e\cap e' =\emptyset$, then let $w$ be an endpoint of $P_{e'}$ such that $w\neq z'$ and let $Q_{i}$ be the graph obtained from $R_{i}^{+}$ by removing all internal vertices of the $(z',w)$-subpath of $P_{e'}$ (see leftmost figure of \autoref{bananafig2}). 

If $z'$ is an internal vertex of $P_{e''}$, where $e''$ is an edge of $K_{4}$ such that $|e\cap e' |=1$, let $w\in V(R_{i})$ be the common endpoint of $P_{e}$ and $P_{e'}$ and let $Q_{i}$ be the graph obtained from $R_{i}^{+}$ by removing all internal vertices of the $(z,w)$-subpath of $P_{e}$ (see rightmost figure of \autoref{bananafig2}).

 \begin{figure}[H]
 	\centering
 	\includegraphics[width=5.5cm]{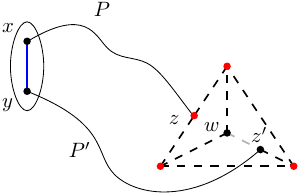}\includegraphics[width=5.5cm]{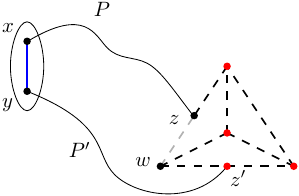}
 	\caption{The graph $R_{i}^{+}$ in Case 2. In both figures, the graph $Q_{i}$ is obtained from $R_{i}^{+}$ by removing all internal vertices of the grey dashed path. The branch vertices of $Q_{i}$ are depicted in red.}\label{bananafig2}
 \end{figure}

Notice that, in all cases above, we obtain a graph $Q_{i}$ that is a subgraph of $G_{i}$ and is a subdivision of $K_{4}$ where the path $P\cup P'\cup \{\{x,y\},\{xy\}\}$ is a subdivided edge of $K_{4}$. Claim follows.

	\smallskip
	
	Therefore, by applying the above Claim for both $G_{1},G_{2}$, and after contracting all edges of  $Q_{i}, i\in\{1,2\}$ that are incident to vertices of degree two in $Q_{i}, i\in\{1,2\}$, we get $\hyperref[connects]{{O}_{1}^{2}}$ as a minor of $G$, a contradiction.
\end{proof}

\medskip

The results of \autoref{reliably} and \autoref{fanfares}, together with \autoref{supplier} and \autoref{shimmers} imply the following corollary:

\begin{corollary}\label{etiology}
	Let $G$ be a biconnected graph such that ${\cal O}\nleq G$ and $K_{4}\leq G.$
	Then there exists a unique triconnected graph $H$ in ${\cal Q}(G)$ such that:
	\begin{itemize}
		\item $H$ is isomorphic to an $r$-wheel for some $r\geq 3$ and
		\item $H$ is a topological minor of $G.$
	\end{itemize}
\end{corollary}

Let $G$ be a biconnected graph such that ${\cal O}\nleq G$ and  $K_{4}\leq G$
and let $K$ be a subdivision of the (unique) $r$-wheel $H\in {\cal Q}(G),$ as in \autoref{etiology}.
We call the pair $(H,K)$ an {\em ${r}$-wheel-subdivision pair of $G$}.
Notice that there may be many  ${r}$-wheel-subdivision pairs in  $G,$ as there might be  many possible 
choices for $K$, but there is only one choice for $H$.

\bigskip

\begin{lemma}
	\label{analysts}
	Let $G$ be a biconnected graph such that $K_4 \leq G$ and ${\cal O} \not\leq G$. Let $(H,K)$ be an ${r}$-wheel-subdivision pair of $G$.
	Then for every $x,y \in V(K)$ and every $(x,y)$-path that intersects $K$ only in its endpoints, there exists an edge $e\in E(H)$ such that $x,y$ are both vertices of the subdivision of $e$ in $K$.
	
	
\end{lemma}

\begin{proof} Recall that  $H$ is isomorphic to an $r$-wheel for some $r\geq 3,$
	and $K$ is a subdivision of $H.$
	Consider an $(x,y)$-path that intersects $K$ only in its endpoints. Suppose, to the contrary, that $x,y$ belong to subdivisions of different edges of $H.$
	We distinguish the following cases: \smallskip
	
	\noindent{{\em Case 1:}  One of $x,y,$ say $x,$ is a branch vertex on the circumference of $K.$} 
	
	\smallskip
	
	First, we observe the following:
	
	\noindent{\em Observation 1:}
	$r \not= 3.$ Indeed, if $H\cong W_{3},$ then since $y$ belongs to the subdivision of an edge of $H$ not incident to $x,$
	a subdivision of a bigger wheel would be formed with $x$ as its central vertex (see \autoref{chrysler}),
	a contradiction to the definition of the triconnected components.
	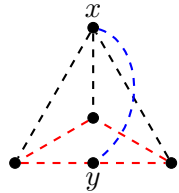
\begin{figure}[H]
		\centering
		\begin{tikzpicture}[every node/.style = black node, scale=.8]
		
		\node[label= above: $x$] (X) at (90:1.5) {};
		\node (Z) at (0,0) {};
		\node (A) at (210:1.5) {};
		\node (B) at (330:1.5) {};
		
		\begin{scope}[on background layer]
		\draw[dashed, thick] (X) -- (A);
		\draw[dashed, thick] (B) -- (X) -- (Z);
		\draw[dashed, thick, red] (B) -- (Z) -- (A) --(B);
		\end{scope}
		
		\node[label = below: $y$] (Y) at ($(A)!0.50!(B)$) {};
		\draw[dashed, thick, blue] (X) .. controls (60:1.5) and (1:1) .. (Y);
		
		\end{tikzpicture}
		\vspace{-2mm}\caption{The $(x,y)$-path (depicted in blue) where $y$ belongs to the subdivision of some edge of $H$ not incident to $x$ (depicted in red) in the proof of Observation 1.}\label{chrysler}
	\end{figure}
	
	Suppose then that $r\geq 4.$
	Let $x_{1}, x_{2}$ be the vertices adjacent to $x$ on the circumference of $H.$ We distinguish the following subcases:
	
	\smallskip
	
	\noindent{\em Subcase 1.1:} $y$ belongs to the subdivision of some spoke $e$ of $H.$ Then, $e$ is not incident to $x.$
	If $y$ is an internal vertex of the subdivision of a spoke incident to either $x_{1}$ or
	$x_{2}$ then $\hyperref[japanese]{{O}_{1}^{3}}\leq G$ (see leftmost figure of \autoref{stressed}), while if the spoke is not incident to $x_{1}$ or $x_{2}$
	then $\hyperref[rigorous]{{O}_{3}^{3}}\leq G$ (see central figure of \autoref{stressed}), a contradiction in both cases.
	\smallskip
		
	\noindent{\em Subcase 1.2:} $y$ belongs to some subdivided edge $e$ of the circumference of $K.$
	Then, $e$ is different from the subdivided edges corresponding to $x x_{1}, x x_{2}.$
	Hence, $\hyperref[rigorous]{{O}_{3}^{3}}\leq G$ (see rightmost figure of \autoref{stressed}), a contradiction.	\vspace{-2mm}

	\begin{figure}[H]
		\centering
		\begin{tikzpicture}[every node/.style = black node, scale=.8]
		\begin{scope}
		
		\draw[dashed, thick] (0,0) circle (1.5);
		\foreach \t in {1, ..., 4} {
			\draw[dashed, thick] (\t * 90:1.5) node {} -- (0,0);
		}
		
		\node () at (0,0) {};
		\node[label=90:$x$] (X) at (90:1.5) {};
		\node[label=180:$x_{1}$] (X1) at (180:1.5) {};
		\node () at (270:1.5) {};
		\node[label=360:$x_{2}$] (X2) at (360:1.5) {};
		
		\node[label = below: $y$] (Y) at (360:0.75) {};
		\draw[dashed, thick, blue] (X)  to [bend left = 30] (Y);
		
		\end{scope}
		
		\begin{scope}[xshift=6cm]
		
		\draw[dashed, thick] (0,0) circle (1.5);
		\foreach \t in {1, ..., 4} {
			\draw[dashed, thick] (\t * 90:1.5) node {} -- (0,0);
		}
		
		\node () at (0,0) {};
		\node[label=90:$x$] (X) at (90:1.5) {};
		\node[label=180:$x_{1}$] (X1) at (180:1.5) {};
		\node () at (270:1.5) {};
		\node[label=360:$x_{2}$] (X2) at (360:1.5) {};
		
		\node (Y) at (270:0.75) [label = left: $y$] {};
		\draw[dashed, thick, blue] (X) to [bend left = 60] (Y);

		\end{scope}

		\begin{scope}[xshift=12cm]
		
		\draw[dashed, thick] (0,0) circle (1.5);
		\foreach \t in {1, ..., 4} {
			\draw[dashed, thick] (\t * 90:1.5) node {} -- (0,0);
		}
		
		\node () at (0,0) {};
		\node[label=90:$x$] (X) at (90:1.5) {};
		\node[label=180:$x_{1}$] (X1) at (180:1.5) {};
		\node () at (270:1.5) {};
		\node[label=360:$x_{2}$] (X2) at (360:1.5) {};
		
		\node[label = below: $y$] (Y) at (315:1.5) {};
		\draw[dashed, thick, blue] (X)  to [bend left = 30] (Y);
		
		\end{scope}
		
		\end{tikzpicture}
		\vspace{-2mm}\caption{Possible configurations of the $(x,y)$-path (depicted in blue) in the proof of Subcases of Case 1.}\label{stressed}
	\end{figure}

	\noindent{{\em Case 2:} One of $x, y,$ say $x,$ is a subdividing vertex on the circumference of $K.$} \smallskip
	
	Since we have examined the case that one of $x,y$ is a branch vertex on the circumference of $K,$ suppose that $y$ is not such.
	
	Let $e= uv$ be the edge of $H$ whose the corresponding subdivision in $K$ contains $x.$
	
	\noindent{\em Observation 2:} $y$ is not the central vertex of $H.$ This is because, if otherwise, a subdivision of a bigger wheel would be formed in $G$ (see \autoref{adopting}), which is a contradiction to the definition of the triconnected components.	\vspace{-3mm}
		
	\begin{figure}[H]
		\centering
		\begin{tikzpicture}[every node/.style = black node, scale=.8]
		
		\draw[dashed, thick] (0,0) circle (1.5);
		\foreach \t in {1,5,9} {
			\draw[dashed, thick] (\t * 30:1.5) node {} -- (0,0);
		}
		
		\node[label=below left:$y$] (Y) at (0,0) {};
		\node[label=30:$v$] (V) at (30:1.5) {};
		\node[label=180:$u$] (U) at (150:1.5) {};
		\node () at (270:1.5) {};
		\node[label=above:$x$] (X) at (90:1.5) {};
		
		\draw[dashed, thick, blue] (X)  to [bend left = 30] (Y);
		
		\end{tikzpicture}
		\vspace{-2mm}\caption{The $(x,y)$-path (depicted in blue) in the proof of Observation 2, where $y$ is the central vertex of $H.$}\label{adopting}
	\end{figure}
	
	\smallskip
	
	\noindent{\em Subcase 2.1:} $y$ is an internal vertex of some subdivided edge $e'$ of a spoke of $K.$
	If $e'$ is incident to either $u$ or $v,$ then $\hyperref[japanese]{{O}_{1}^{3}}\leq G$ (see leftmost figure of \autoref{suitable}), while if $e'$ is not incident to either $u$ or $v$ then $\hyperref[persuade]{{O}_{2}^{3}}\leq G$ (see central figure of \autoref{suitable})),
	a contradiction in both cases.
	
	\smallskip
	
	\noindent{\em Subcase 2.2:} $y$ is an internal vertex of some subdivided edge of the circumference of $K$ different from $e.$
	Then $\hyperref[japanese]{{O}_{1}^{3}}\leq G$ (see rightmost figure of \autoref{suitable}), a contradiction.
	\begin{figure}[H]
		\centering
		\begin{tikzpicture}[every node/.style = black node, scale=.8]
		
		\begin{scope}
		
		\draw[dashed, thick] (0,0) circle (1.5);
		\foreach \t in {1,5,9} {
			\draw[dashed, thick] (\t * 30:1.5) node {} -- (0,0);
		}
		
		\node () at (0,0) {};
		\node[label=30:$v$] (V) at (30:1.5) {};
		\node[label=180:$u$] (U) at (150:1.5) {};
		\node () at (270:1.5) {};
		\node[label=above:$x$] (X) at (90:1.5) {};
		\node[label = below right: $y$] (Y) at (30:0.75) {};
		
		\draw[dashed, thick, blue] (X) to [bend left = 30] (Y);
		\end{scope}
		
		\begin{scope}[xshift = 6cm]
		
		\draw[dashed, thick] (0,0) circle (1.5);
		\foreach \t in {1,5,9} {
			\draw[dashed, thick] (\t * 30:1.5) node {} -- (0,0);
		}
		
		\node () at (0,0) {};
		\node[label=30:$v$] (V) at (30:1.5) {};
		\node[label=180:$u$] (U) at (150:1.5) {};
		\node () at (270:1.5) {};
		\node[label=above:$x$] (X) at (90:1.5) {};
		\node[label = below right: $y$] (Y) at (270:0.75) {};
		
		\draw[dashed, thick, blue] (X) to [bend left = 60] (Y);
		
		\end{scope}
		
		\begin{scope}[xshift=12cm]
		
		\draw[dashed, thick] (0,0) circle (1.5);
		\foreach \t in {1,5,9} {
			\draw[dashed, thick] (\t * 30:1.5) node {} -- (0,0);
		}
		
		\node () at (0,0) {};
		\node[label=30:$v$] (V) at (30:1.5) {};
		\node[label=180:$u$] (U) at (150:1.5) {};
		\node () at (270:1.5) {};
		\node[label=above:$x$] (X) at (90:1.5) {};
		\node[label = below right: $y$] (Y) at (315:1.5) {};
		
		\draw[dashed, thick, blue] (X) to [bend left = 30] (Y);
		
		\end{scope}
		
		\end{tikzpicture}
		\vspace{-2mm}\caption{Possible configurations of the $(x,y)$-path (depicted in blue) in the proof of Subcases of Case 2.}\label{suitable}
	\end{figure}

	\noindent {\em Case 3:} Both of $x, y$ are internal vertices of the subdivisions of some spokes $e,e'$ of $W_{r},$ respectively.

	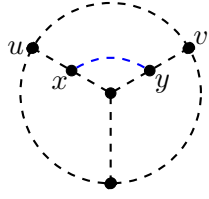
\begin{figure}[H]
		\centering
		\begin{tikzpicture}[every node/.style = black node, scale=.8]
		\begin{scope}
		
		\draw[dashed, thick] (0,0) circle (1.5);
		\foreach \t in {1,5,9} {
			\draw[dashed, thick] (\t * 30:1.5) node {} -- (0,0);
		}
		
		\node () at (0,0) {};
		\node[label=30:$v$] (V) at (30:1.5) {};
		\node[label=180:$u$] (U) at (150:1.5) {};
		\node () at (270:1.5) {};
		\node[label= below left:$x$] (X) at (150:0.75)  {};
		\node[label = below right: $y$] (Y) at (30:0.75) {};
		
		\draw[dashed, thick, blue] (X) to [bend left = 30] (Y);
		
		\end{scope}
		
		\end{tikzpicture}
		\vspace{-2mm}\caption{The $(x,y)$-path (depicted in blue) in the proof of Case 3.}\label{reckoned}
	\end{figure}
	\noindent In this case, $e, e'$ are distinct and so $\hyperref[japanese]{{O}_{1}^{3}}\leq G$ (see \autoref{reckoned}), a contradiction.
\end{proof}

\bigskip

We can now define the notion of a {\em flap}.
Let $G$ be a biconnected graph such that ${\cal O}\nleq G$ and $K_{4}\leq G$.
Let also  $(H,K)$ be an  ${r}$-wheel-subdivision pair of  $G$.
For every 2-separator $S\subseteq V(K)$ of $G$, the {\em flap of $(H,K)$ of base $S$} is the subgraph of $G$ defined as $F=\cupall \{C\in{\cal C}(G,S): K_{4}\not\leq C \mbox{~and~} C \text{ is biconnected}\}$ if $V(F)\neq \emptyset$. If $V(F) =\emptyset$, then the flap of base $S$ is not defined.
Observe that every 2-separator $S\subseteq V(K)$ of $G$ defines
at most one flap of $(H,K)$ of base $S$.
The {\em $(x,y)$-flap} of $(H,K)$, for some 2-separator $\{x,y\}\subseteq V(K)$, is the flap of $(H,K)$ of base $\{x,y\}$.

Given an $(x,y)$-flap $F$ of $(H,K)$ and a vertex $v\in V(F)$, we say that
$F$ is {\em $v$-oriented} if every cycle of $F$ contains $v$.

\bigskip

Regarding the arguments in the remaining part of \autoref{daughter}, consider a graph
$G\in \obs({\cal A}_{1}({\cal P}))\setminus  {\cal O}$ such that $K_{4}\leq G$.
Observe that $K_{4}$ is a minor of a block $B$ of $G$ and
therefore we can consider an ${r}$-wheel-subdivision pair of $B$.

\begin{lemma}\label{campaign}
	Let $G\in \obs ({\cal A}_{1}({\cal P}))\setminus {\cal O} $ such that $K_{4}\leq G$ and $(H,K)$ be an ${r}$-wheel-subdivision pair of a block $B$ of $G$. If $F$ is an $(x,y)$-flap of $(H,K),$ then it holds that:
	\begin{enumerate}
		\item $F$ is biconnected,
		\item $G[V(F)]$ contains a cycle, and 
		\item there exists an edge $e\in E(H)$ such that $x,y$ are both vertices of the subdivision of $e$ in $K.$
	\end{enumerate}
Moreover, if $F_{1},F_{2}$ are flaps of $(H,K)$ whose bases are $S_{1}, S_{2}$ respectively, where $S_{1}\neq S_{2}$, then $V(F_{1})\cap V(F_{2})\subseteq S_{1}\cap S_{2}$.
\end{lemma}

\begin{proof}
	
	(1) and (2) are direct consequences of the definition of the $(x,y)$-flap and \autoref{generals}.
	
   To prove (3), consider an $(x,y)$-flap $F$ of $(H,K)$. Then, from the definition
   of $F$, there exists a biconnected graph $C\in{\cal C}(G,\{x,y\})$ such that 
   $K_{4}\not\leq C$ and hence it contains an $(x,y)$-path that intersects $K$ only in
   its endpoints. Therefore, \autoref{analysts} implies (3).

It remains to prove that if $F_{1},F_{2}$ are flaps of $(H,K)$ whose bases are $S_{1}, S_{2}$ respectively, where $S_{1}\neq S_{2}$, then $V(F_{1})\cap V(F_{2})\subseteq S_{1}\cap S_{2}$. First, we observe that for every $i\in[2]$, $V(F_{i})\cap V(K)\subseteq S_{i}$.
Now, suppose, towards a contradiction, that there exists a $C\in {\cal C}(G,S_{1})$ such that $K_{4}\nleq C$, $C$ is biconnected, and $C\setminus S_{1}$ contains a vertex $x$ of $F_{2}$. 
Since $S_{1}\neq S_{2}$, we have that $S_{1}\cap S_{2}\subsetneq S_{2}$. This, together with the fact that $F$ is biconnected implies that there is a
path in $F_{2}\setminus (S_{1}\cap S_{2})$ that connects $x$ with a vertex $v\in S_{2}\setminus S_{1}\subseteq V(K)$. Therefore, $v$ is a vertex of $(C\setminus S_{1})\cap V(K)$, a contradiction to the fact that $V(F_{1})\cap V(K)\subseteq S_{1}$.
\end{proof}

We conclude this subsection by proving the next results concerning flaps:

\begin{lemma} \label{mobility}
	Let $G$ be a biconnected graph such that ${\cal O}\nleq G$ and $K_{4}\leq G$ and let $(H,K)$ be an ${r}$-wheel-subdivision pair of  $G$. Then  every $(x,y)$-flap $F$ of $(H,K)$  is either $x$-oriented or $y$-oriented.
\end{lemma}

\begin{proof}
	Consider an $(x,y)$-flap $F$ of $(H,K)$ for which the contrary holds.
	We distinguish the following cases:
	
	\smallskip
	
	\noindent{\em Case 1:} There exists a cycle $C$ in $F$ disjoint to both
	$x$ and $y$. Then, {since by \autoref{campaign}(1)},
	 $F$ is biconnected, there exist two disjoint paths
	$P_{1},P_{2}$ connecting the cycle $C$ with $x,y$, respectively.
	Hence, by contracting all the edges of $P_{1},P_{2}$ we form
	$\hyperref[holiness]{O_{9}^{2}}$ as a minor of $G$, a contradiction
	(see \autoref{loveless}).	\vspace{-3mm}
	
	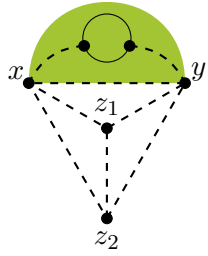
\begin{figure}[H]
		\centering
		\begin{tikzpicture}[every node/.style = black node, scale=.8]
		\fill[applegreen!80] (150:1.5) .. controls +(0,1.8) and +(0,1.8) .. (30:1.5) {};
		
		\draw[dashed, thick] (270:1.5) -- (30:1.5) -- (0,0) -- (150:1.5) -- (270:1.5);
		\draw[dashed, thick] (30:1.5) -- (150:1.5) ;
		\draw[dashed, thick] (270:1.5) -- (0,0);
		\coordinate (C) at (0,1.5) {};
		\draw[-, name path=path1] (C) circle (0.4);
		
		\path[name path=path2] (150:1.5) to [bend left=60] (30:1.5);
		
		\draw [name intersections={of= path1 and path2}]
		\foreach \s in {1,2}{
			(intersection-\s) node {}
		};
		\coordinate (I1)  at (intersection-1);
		\coordinate (I2)  at (intersection-2);
		
		\draw[dashed, thick] (I1) to [bend right=30] (150:1.5) (I2) to [bend left=30] (30:1.5);
		
		\draw (150:1.5) node[label=150:$x$] {};
		\draw (0,0) node[label=90:$z_{1}$] {};
		\draw (270:1.5) node[label=270:$z_{2}$] {};
		\draw (30:1.5) node[label=30:$y$] {};
		
		\end{tikzpicture}
		\vspace{-2mm}\caption{An example of an $(x,y)$-flap that contains a cycle disjoint to $x,y$ in the proof of Case 1.}\label{loveless}
	\end{figure}
	
	\smallskip 
	
	\noindent{\em Case 2:} There exists a cycle $C$ of $F$ that contains $x$
	but not $y$ and a cycle $C'$ that contains $y$ but not $x$. Then,
	if $C,C'$ are disjoint, $\hyperref[consists]{O_{1}^{1}}\leq G$, if
	the share only one vertex, $\hyperref[begrudge]{O_{4}^{2}}\leq G$, and if
	they share more than one vertex, $\hyperref[connects]{O_{1}^{2}}\leq G$,
	a contradiction in all cases (see figure \autoref{feelings}).

	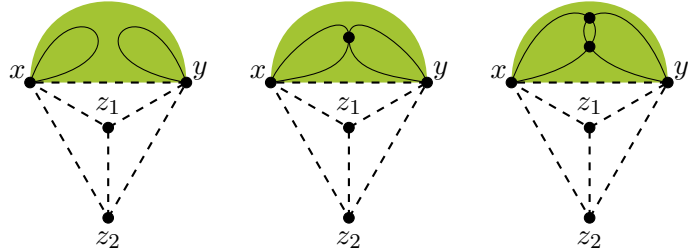
\begin{figure}[H]
		\centering
		\vspace{-1cm}
		\begin{tikzpicture}[every node/.style = black node, scale=.8]
		\begin{scope}
		\fill[applegreen!80] (150:1.5) .. controls +(0,1.8) and +(0,1.8) .. (30:1.5) {};
		
		\draw[dashed, thick] (270:1.5) -- (30:1.5) -- (0,0) -- (150:1.5) -- (270:1.5);
		\draw[dashed, thick] (30:1.5) -- (150:1.5) ;
		\draw[dashed, thick] (270:1.5) -- (0,0);
		
		\draw[-] (150:1.5).. controls ++(10:2.2) and ++(70:2).. (150:1.5);
		\draw[-] (30:1.5).. controls ++(170:2.2) and ++(110:2).. (30:1.5);
		
		\draw (150:1.5) node[label=150:$x$] {};
		\draw (0,0) node[label=90:$z_{1}$] {};
		\draw (270:1.5) node[label=270:$z_{2}$] {};
		\draw (30:1.5) node[label=30:$y$] {};
		
		\end{scope}
		
		\begin{scope}[xshift=4cm]
		
		\fill[applegreen!80] (150:1.5) .. controls +(0,1.8) and +(0,1.8) .. (30:1.5) {};
		
		\draw[dashed, thick] (270:1.5) -- (30:1.5) -- (0,0) -- (150:1.5) -- (270:1.5);
		\draw[dashed, thick] (30:1.5) -- (150:1.5) ;
		\draw[dashed, thick] (270:1.5) -- (0,0);

		\draw[-] (150:1.5) to [out=70, in=90] (90:1.5) to [out=-90, in=10] (150:1.5);
		\draw[-] (30:1.5) to [out=110, in=90] (90:1.5) to [out=-90, in=170] (30:1.5);
		
		\node () at (90:1.5) {};
		
		\draw (150:1.5) node[label=150:$x$] {};
		\draw (0,0) node[label=90:$z_{1}$] {};
		\draw (270:1.5) node[label=270:$z_{2}$] {};
		\draw (30:1.5) node[label=30:$y$] {};
		
		\end{scope}
		
		\begin{scope}[xshift=8cm]
		
		\fill[applegreen!80] (150:1.5) .. controls +(0,1.8) and +(0,1.8) .. (30:1.5) {};
		
		\draw[dashed, thick] (270:1.5) -- (30:1.5) -- (0,0) -- (150:1.5) -- (270:1.5);
		\draw[dashed, thick] (30:1.5) -- (150:1.5) ;
		\draw[dashed, thick] (270:1.5) -- (0,0);
		
		\draw[-, name path=path1] (150:1.5).. controls ++(10:2.7) and ++(70:2.5).. (150:1.5);
		\draw[-, name path=path2] (30:1.5).. controls ++(170:2.7) and ++(110:2.5).. (30:1.5);
		
		\draw [name intersections={of= path1 and path2}]
		\foreach \s in {1,2}{
			(intersection-\s) node {}
		};
		
		\draw (150:1.5) node[label=150:$x$] {};
		\draw (0,0) node[label=90:$z_{1}$] {};
		\draw (270:1.5) node[label=270:$z_{2}$] {};
		\draw (30:1.5) node[label=30:$y$] {};
		
		\end{scope}
		\end{tikzpicture}
		\vspace{-2mm}\caption{The ways $C, C^{\prime}$ may intersect in the proof of Case 2.}\label{feelings}
	\end{figure}

	We arrived at a contradiction in both cases. Hence, Lemma follows.
\end{proof}

\begin{lemma}
	\label{patience}
	Let $G\in \obs({\cal A}_{1}({\cal P}))\setminus  {\cal O}$ such that
	$K_{4}\leq G$ and let $(H,K)$ be an ${r}$-wheel-subdivision pair of a block $B$ of $G$.
	Then if $r\geq 4$, the center of $K$ is a ${\cal P}$-apex vertex of $B$.
\end{lemma}

\begin{proof} 
	
	We start with the following claim:

	\medskip

    \noindent{\em Claim: For every $(x,y)$-flap of $(H,K)$, one of $x,y$ is the centre of $K$.}
    
    \smallskip
    
    \noindent{\em Proof of Claim:} Let $F$ be an $(x,y)$-flap.
    \autoref{campaign}(3)
    implies that there exists an edge $e$ of $H$ such that $x,y$ are vertices of the
    subdivision of $e$ in $K$. Suppose, towards a contradiction, that neither of
    $x,y$ is the center of $K$. If $e$ is in the circumference of $H$ then,
    since by \autoref{campaign}(1),(2),
     $F$ is biconnected and $G[V(F)]$ contains a cycle,
    $\hyperref[handling]{{O}_{2}^{2}}\leq G$ (see left figure of
    \autoref{linearly}), while if $e$ is a spoke of $H$, then, similarly, $\hyperref[skeleton]{{O}_{4}^{1}}\leq G$ (see right figure of \autoref{linearly}),
    a contradiction in both cases.

	\begin{figure}[H]
		\centering
		\begin{tikzpicture}[every node/.style = black node, circle dotted/.style={dash pattern=on .05mm off 2mm,line cap=round}, scale=.8]
		\begin{scope}
		\draw[dashed, thick] (0,0) circle (1.5);
		
		\foreach \t in {1, ..., 4} {
			\draw[dashed, thick] (\t * 90:1.5) -- (0,0);
			\draw[line width = 0.4mm,circle dotted] (\t * 90 + 25:0.9) arc (\t * 90 + 25: \t * 90 + 65:0.9);
		}
		\foreach \t in {1, ..., 4} {
			\draw (\t * 90:1.5) node {};
		}
		
		\begin{scope}[on background layer]
		\fill[applegreen!90] (115:1.5) to [bend right=50] (140:2.2) to  [bend right=50] (165:1.5) to [bend left=23] (115:1.5);
		\end{scope}

		\draw (0,0) node {};
		\draw (115:1.5) node[label=80:$y$] {};
		\draw (165:1.5) node[label=200:$x$] {};
		\end{scope}
		
		\begin{scope}[xshift=6cm]
		
		\draw[dashed, thick] (0,0) circle (1.5);
		
		\foreach \t in {1, ..., 4} {
			\draw[dashed, thick] (\t * 90:1.5) -- (0,0);
			\draw[line width = 0.4mm,circle dotted] (\t * 90 + 25:0.9) arc (\t * 90 + 25: \t * 90 + 65:0.9);
		}
		\foreach \t in {1, ..., 4} {
			\draw (\t * 90:1.5) node {};
		}
		
		\begin{scope}
		\fill[applegreen!90] (0:1.2) to [bend right=90, min distance=.7cm] (0:0.4);
		\end{scope}

		\draw (0,0) node {};
		\draw (0:1.2) node[label=270:$y$] {};
		\draw (0:0.4) node[label=270:$x$] {};
		
		\end{scope}
		\end{tikzpicture}
		\vspace{-2mm}\caption{Possible configurations of an $(x,y)$-flap (depicted in green) in the proof of Observation.}\label{linearly}
	\end{figure}
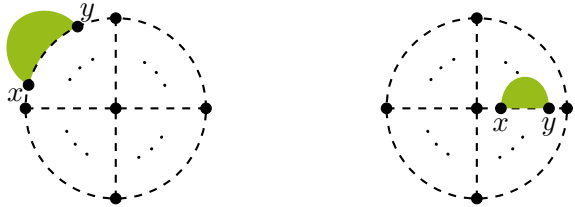

	Let now $z$ be the centre of $K$. According to the above Claim,
	every flap of $(H,K)$ is a $(z,y)$-flap. It also holds
	that every $(z,y)$-flap $F$ is $z$-oriented. Indeed, if
	otherwise, then \autoref{mobility} implies that there would exist a
	cycle in $F$ containing $y$ but not $z$ and hence
	$\hyperref[skeleton]{{O}_{4}^{1}}\leq G,$ a contradiction.
	
	Therefore, taking into account that $B$ is a block of $G$,
	and, due to \autoref{fanfares}, $H$ is the unique triconnected graph in
	${\cal Q}(G)$, we derive that every
	cycle in $B$, except for the circumference of $K,$
	contains $z$ and so $z$ is a ${\cal P}$-apex vertex of $B$.
\end{proof}

\subsection{Properties of obstructions containing a \texorpdfstring{$K_{2,3}$}{K2,3}}

The purpose of this section is to prove \autoref{caillois} that gives us some information on the structure of a connected graph $G$ such that $ {\cal O}\nleq G$, $K_{4}\nleq G$, and  $K_{2,3}\leq G$.\medskip

Let $S$ be a 2-separator of $G$ and $B$ be a (non-trivial) block of some $H\in {\cal C}(G,S)$. We say that $B$ is an {\em S-block} of $G$ if $S\subseteq V(B)$.
Also, if $S$ is a rich 2-separator and at least three graphs in ${\cal C}(G,S)$ contain $S$-blocks, we call $S$ a {\em b-rich separator}.
{We say that a b-rich separator $S$ of $G$ is {\em nice} if for every $H\in{\cal C}(G,S)$ that contains an $S$-block, it holds that the graph $G[V(H)]$ contains a cycle.}

We start with an easy observation.

\begin{observation}
	\label{negative}
	Let $G$ be a $K_{4}$-free graph such that $K_{2,3} \leq G.$ Then the connectivity of $G$ is at most 2 and $G$ contains a b-rich separator. 
\end{observation}

We now prove the following:

\begin{lemma}
\label{princess}
Let $G\in\obs({\cal A}_{1}({\cal P}))\setminus  {\cal O}$ be a graph.
Every b-rich separator of $G$ is nice.
\end{lemma}

\begin{proof}
Let $S=\{x,y\}$ be a b-rich separator of $G$.
Suppose that $S$ is not nice.
Then, there is an $H\in {\cal C}(G,S)$ that contains an $S$-block and the graph $G[V(H)]$ does not contain a cycle.
Given that $H$ contains an $S$-block, there is a biconnected graph that is a subgraph of $H$ and its vertex set contains $x$ and $y$.
Therefore, $H$ contains an $(x,y)$-path $P$ that does not contain the edge $xy$ (which, by definition of ${\cal C}(G,S)$, is an edge of $H$).
Notice that $P$ is also a subgraph of $G[V(H)]$.

We claim that for every internal vertex $v$ of $P$, it holds that $\deg_{G}(v) = 2$.
To see why this holds, first observe that $\deg_{G}(v)\geq\deg_{P}(v)= 2$ and,
suppose towards a contradiction, that there is a third neighbor $u\in V(G)$ of
$v$. Observe that $u\in V(H)\setminus V(P)$, since if $u\in V(P)$ then
$G[V(H)]$ contains a cycle, contradicting our initial assumption, and if $v\in V(G)\setminus V(H)$, then $S$ is not a separator of $G$, a contradiction.
Also, every $(u,v)$-path in $G$ contains the edge $uv$, since, otherwise, the existence of a $(u,v)$-path in $G$ that avoids $uv$ implies the existence of a cycle in $G[V(H)]$, a contradiction to our initial assumption.
But then the edge $uv$ then $uv$ is a bridge of $G$, a contradiction to the fact that $G$ is bridgeless (due to  \autoref{generals}).

Thus, for every internal vertex $v$ of $P$, it holds that $\deg_{G}(v) = 2$ and, by  \autoref{generals}, every such vertex has to be simplicial.
This implies the existence of a cycle in $G[V(H)]$, a contradiction to our initial assumption.
\end{proof}
%

We now argue why in a graph $G$ that is connected and $K_4$-free and where ${\cal O}\nleq G$ and $K_{2,3} \leq G$, the existence of more than two nice b-rich separators implies the existence of a ${\cal P}$-apex vertex of $G$ that is the intersection of all such separators.
We stress that the following lemma is not stated for graphs in $\obs({\cal A}_{1}({\cal P}))\setminus  {\cal O}$ and, therefore, we can not use \autoref{princess} and assume that every b-rich separator is nice.

\begin{lemma}
	\label{symphony}
%
%
	{Let $G$ be a connected $K_{4}$-free graph such that ${\cal O}\nleq G$ and $K_{2,3} \leq G$ and let $S_{1}, \ldots, S_{k}$ be the b-rich separators of $G$. If $k\geq 2$ and every $S_{i}, i\in [k]$, is nice, then there is a vertex $x\in V(G)$ such that $\bigcap\limits_{i=1}^{k} S_{i} = \{x\}$ and $x$ is a ${\cal P}$-apex vertex of $G$.}
\end{lemma}

\begin{proof}
	Suppose that $k\geq 2$ and every b-rich separator $S_{i}, i\in[k]$ of $G$ is nice.
	We first prove the following claim:
	\medskip
	
	\noindent {\em Claim 1:} There exists some $x\in V(G)$ such that $\bigcap\limits_{i=1}^{k} S_{i} = \{x\}$.
	
	\smallskip
	
	\noindent {\em Proof of Claim 1:} We first prove that every two b-rich separators of $G$ have a non-empty intersection. 
	Suppose, towards a contradiction, that there exist $i,j\in[k]$, where $i\neq j$ and
	$S_{i}\cap S_{j}= \emptyset$.
	Due to $K_{4}$-freeness of $G$, there exists an augmented component $H\in {\cal C}(G,S_{i})$ such that $S_{j} \subseteq V(H)$.
	Observe that, since $S_{i}$ is nice, there exist at least two augmented connected components that contain $S_{i}$-blocks in ${\cal C}(G,S_{i}) \setminus \{H\}$ that contain a cycle and together form $K_{4}^{-}$ as a minor.
	By applying the same arguments symmetrically, there exist at least two augmented components in $ {\cal C}(G,S_{j})$ that do not contain $S_{i}$, which together also form $K_{4}^{-}$ as a minor.
	Then, notice that $\hyperref[magnetic]{{O}_{1}^{0}}\leq G$ (see leftmost figure of \autoref{marchers}), a contradiction.

	Now, suppose that there exist three distinct $i,j,l\in [k]$ such that
    $S_{i} \cap S_{j} \cap S_{l} = \emptyset$. Notice that since $S_{j} \cap S_{l}\not = \emptyset$, there exists a unique augmented connected component of ${\cal C}(G,S_{i})$ that contains both $S_{j}, S_{l}$,
	while there also exist at least two other augmented connected components that
	contain $S_{i}$-blocks and,	since $S_{i}$ is nice, form $K_{4}^{-}$ as a minor.
	By applying the same argument to ${\cal C}(G,S_{j}), {\cal C}(G,S_{l})$, we have that for each said separator we can form $K_{4}^{-}$ as a minor and therefore, $\hyperref[conceive]{{O}_{11}^{2}}\leq G$ (see rightmost figure of \autoref{marchers}), a contradiction.
	
	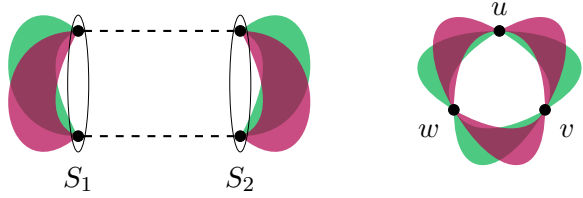
\begin{figure}[H]
		\centering
		\vspace{-.7cm}
		\begin{tikzpicture}[myNode/.style = black node, scale=0.7]
				
				\begin{scope}
				
				\node[myNode] (A) at (0,2) {};
				\node[myNode] (B) at (0,0) {};
				\node[myNode] (C) [right=2cm of A] {};
				\node[myNode] (D) [right=2cm of B] {};
				
				\begin{scope}[on background layer]
				\fill[green!70!blue, opacity=0.7]
				(A.center) to [out=200, in=120] (B.center) to [out=180, in=140, min distance =2cm] (A.center); 
				\fill[red!70!blue, opacity=0.7]
				(B.center) to [out=160, in=240] (A.center) to [out=180, in=220, min distance=2cm] (B.center); 
				\end{scope}
				
				\begin{scope}[on background layer,xscale=-1]
				\fill[green!70!blue, opacity=0.7]
				(C.center) to [out=200, in=120] (D.center) to [out=180, in=140, min distance =2cm] (C.center); 
				\fill[red!70!blue, opacity=0.7]
				(D.center) to [out=160, in=240] (C.center) to [out=180, in=220, min distance=2cm] (D.center); 
				\end{scope}
				
				\draw[very thin] ($(A)!0.5!(B)$) ellipse (0.2 and 1.3);
				\node [below=.2cm of B] {$S_{i}$};
				\draw[very thin] ($(C)!0.5!(D)$) ellipse (0.2 and 1.3);
				\node [below=.2cm of D] {$S_{j}$};
				
				\draw[dashed,thick] (A)      to  (C);
				\draw[dashed,thick] (B)      to  (D);
				
				\end{scope}

				\begin{scope}[xshift=8cm, yshift =1cm]

				\node[myNode, label=above:$u$] (A) at (90:1) {};
				\node[myNode, label=below left:$w$] (B) at (210:1) {};
				\node[myNode, label=below right:$v$] (C) at (330:1) {};

				\begin{scope}[on background layer]
				\fill[green!70!blue, opacity=0.7]
				(B.center) to [out=120, in=210] (A.center) to [out=180, in=150, min distance =1.5cm] (B.center); 
				\fill[red!70!blue, opacity=0.7]
				(B.center) to [out=90, in=180] (A.center) to [out=150, in=120, min distance=1.5cm] (B.center); 
				\end{scope}
				
				\begin{scope}[on background layer, xscale=-1]
				\fill[green!70!blue, opacity=0.7]
				(C.center) to [out=120, in=210] (A.center) to [out=180, in=150, min distance =1.5cm] (C.center); 
				\fill[red!70!blue, opacity=0.7]
				(C.center) to [out=90, in=180] (A.center) to [out=150, in=120, min distance=1.5cm] (C.center); 
				\end{scope}
				
				\begin{scope}[on background layer, yscale=-1]
				\fill[green!70!blue, opacity=0.7]
				(B.center) to [out=60, in=150] (C.center) to [out=120, in=90, min distance =1.5cm] (B.center); 
				\fill[red!70!blue, opacity=0.7]
				(B.center) to [out=30, in=120] (C.center) to [out=90, in=60, min distance=1.5cm] (B.center); 
				\end{scope}
				
				\end{scope}
				\end{tikzpicture}
				\vspace{-2mm}\caption{The proof of Claim 1.}\label{marchers}
	\end{figure}
	
	Therefore, $\bigcap\limits_{i=1}^{k} S_{i} = \{x\}$ for some $x\in V(G)$.
	Claim 1 follows.
	
	\medskip

	Consider the b-rich separators $S_{1},S_{2}$ of $G$. According to Claim 1, we
	have that $S_{1} = \{x,u_{1} \},$ $S_{2} = \{x,u_{2} \}$, for some $u_{1},u_{2}\in V(B)$, where $u_{1}\neq u_{2}$.
	Let $$H_{1} = \cupall \{ H \in {\cal C}(G,S_{1}) : u_{2} \notin V(H) \}\mbox{~~and~~}H_{2} = \cupall \{ H \in {\cal C}(G,S_{2}) : u_{1} \notin V(H) \}.$$ 
	
	{Also, let $\bar{H} = G \setminus ((V(H_{1}) \cup V(H_{2}))\setminus\{u_{1},u_{2}\})$.} See \autoref{2020figure} for an illustration of the above graphs. For each $i\in[2]$, observe that, since $S_{i}$ is nice, we have that $K_{4}^{-}\leq H_{i}$.
	\begin{figure}[H]
\centering
\begin{tikzpicture}[myNode/.style = black node, scale=.8]
\begin{scope}
\node[myNode, label=above:$x$] (A) at (90:2) {};
\node[myNode, label=below left:$u_{1}$] (B) at (210:2) {};
\node[myNode, label=below right:$u_{2}$] (C) at (330:2) {};

\begin{scope}[on background layer]
\fill[red!70!blue, opacity=0.7]
(B.center) to [out=90, in=190] (A.center) to [out=150, in=120, min distance=2.5cm] (B.center);
\fill[green!70!blue, opacity=0.7]
(B.center) to [out=110, in=210] (A.center) to [out=180, in=150, min distance =2.5cm] (B.center);





\end{scope}

\begin{scope}[on background layer, xscale=-1]
\fill[red!70!blue, opacity=0.7]
(C.center) to [out=90, in=190] (A.center) to [out=150, in=120, min distance=2.5cm] (C.center); 
\fill[green!70!blue, opacity=0.7]
(C.center) to [out=110, in=210] (A.center) to [out=180, in=150, min distance =2.5cm] (C.center);

%
%

\end{scope}

\begin{scope}[on background layer]
\fill [mustard!80] (B.center) to [bend left =60] (C.center) to [bend left = 60] (B.center);

\coordinate (Mid) at ($(C)!0.5!(B)$) {};
\node[label=center:$\bar{H}$] () at ($(Mid)$) {};
\end{scope}

\end{scope}
\end{tikzpicture}
\caption{An illustration of the graphs $H_{1}$, $H_{2}$ (the red and green parts on the left and on the right respectively) and the graph $\bar{H}$, depicted in yellow. }\label{2020figure}
	\end{figure}

%
	
	{Notice that for each $i\in [2]$, the graph $G\setminus H_{i}$ is connected and the graph $\bar{H}\setminus u_{i}$ is a connected subgraph of $G\setminus H_{i}$. Therefore, we can derive the following:}
	
	\smallskip
	
	\noindent {\em Observation 1:} If $y\in V(\bar{H})$, then there exists a $(y, u_{1})$-path in $\bar{H}\setminus u_{2}$ and a $(y,u_{2})$-path in $\bar{H}\setminus u_{1}$. In particular, $u_{1}, u_{2}$ are not cut-vertices of $\bar{H}$.
	
	{It remains to prove that $x$ is a ${\cal P}$-apex vertex of $G$. For that,
	suppose, towards a contradiction, that there exist} two cycles $C_{1}, C_{2}$, that are connected in $G\setminus x$. We will now argue that the following holds:

	\medskip
	
	\noindent {\em Claim 2:} $C_{1}, C_{2}$ are in $\bar{H}$. \smallskip
	
	\noindent {\em Proof of Claim 2:}
	{If $C_{1}, C_{2}$ are both in some $H_{i}, \ i\in [2]$, say $H_{1}$, then $H_{1}\setminus x \not\in{\cal P}$ and so $\obs({\cal P})\leq H_{1}\setminus x$.}
	{Hence, since $K_{4}^{-}\leq H_{2}$, $\{\hyperref[magnetic]{{O}_{1}^{0}},\hyperref[narcotic]{{O}_{3}^{0}}\}\leq G$, a contradiction.}
	Also, if each $C_{i}$ belongs to different $H_{j}, \ i,j\in [2]$, say
    $C_{1}\subseteq H_{1}$ and $C_{2}\subseteq H_{2}$, then since $C_{1}, C_{2}$ are connected in $G\setminus x$, we have $\hyperref[daylight]{{O}_{12}^{1}}\leq G$ (see left figure of \autoref{mystical}), a contradiction.
	
	Therefore, at least one of $C_{1}, C_{2}$ is in  $\bar{H}$. Suppose then,
	towards a contradiction, that one of $C_{1},C_{2}$, say $C_{1}$, is in
	$H_{1}\setminus x$. Then $C_{2}\subseteq \bar{H}$. Observation 1 implies that
	$C_{2}$ is connected with $u_{1}$ through a path disjoint from $u_{2}$ and thus, if $u_{2}\notin V(C_{2}), \ \hyperref[narcotic]{{O}_{3}^{0}}\leq G$ (see central figure of \autoref{mystical}), while if $u_{2}\in V(C_{2})$, then by contracting all edges in the said path, we get $\hyperref[granting]{{O}_{6}^{1}}\leq G$ (see right figure of \autoref{mystical}),  a contradiction in both cases. Claim 2 follows.

	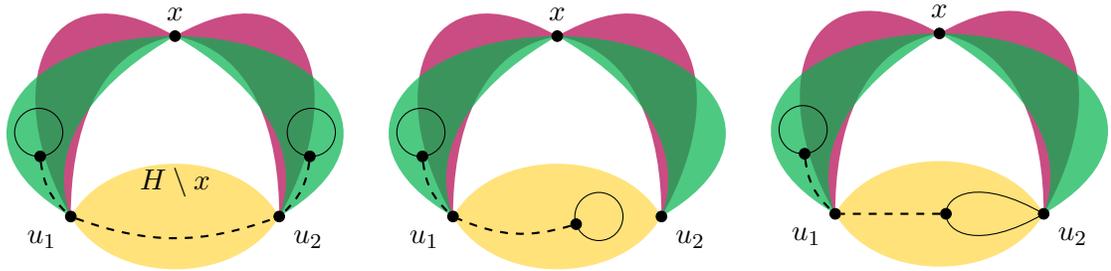
\begin{figure}[H]
		\centering
		\begin{subfigure}{.3\linewidth}
			\begin{tikzpicture}[myNode/.style = black node, scale=.8]
			\begin{scope}
			\node[myNode, label=above:$x$] (A) at (90:2) {};
			\node[myNode, label=below left:$u_{1}$] (B) at (210:2) {};
			\node[myNode, label=below right:$u_{2}$] (C) at (330:2) {};

			\begin{scope}[on background layer]
			\fill[red!70!blue, opacity=0.7]
			(B.center) to [out=90, in=190] (A.center) to [out=150, in=120, min distance=2.5cm] (B.center);
			\fill[green!70!blue, opacity=0.7]
			(B.center) to [out=110, in=210] (A.center) to [out=180, in=150, min distance =2.5cm] (B.center);
			
			\coordinate (C1) at (170:2.3){};
			\draw[-, name path = path1] (C1) circle (0.4);
			
			\path[name path=path2] (C1) to [bend right =20] (B);
			
			\draw [name intersections={of= path1 and path2}]
			\foreach \s in {1}{
				(intersection-\s) node {}
			};
			\node[myNode] (I1)  at (intersection-1) {};
			
			\draw[dashed,thick] (I1) to [bend right =20] (B);
			
			\end{scope}
			
			\begin{scope}[on background layer, xscale=-1]
			\fill[red!70!blue, opacity=0.7]
			(C.center) to [out=90, in=190] (A.center) to [out=150, in=120, min distance=2.5cm] (C.center); 
			\fill[green!70!blue, opacity=0.7]
			(C.center) to [out=110, in=210] (A.center) to [out=180, in=150, min distance =2.5cm] (C.center);

			\coordinate (C2) at (170:2.3){};
			\draw[-, name path = path1] (C2) circle (0.4);
			
			\path[name path=path2] (C2) to [bend right =20] (C);
			
			\draw [name intersections={of= path1 and path2}]
			\foreach \s in {1}{
				(intersection-\s) node {}
			};
			\node[myNode] (I2)  at (intersection-1) {};
			
			\draw[dashed,thick] (I2) to [bend right =20] (C);
			\end{scope}
			
			\begin{scope}[on background layer]
			\fill [mustard!80] (B.center) to [bend left =60] (C.center) to [bend left = 60] (B.center);
			
			\coordinate (Mid) at ($(C)!0.5!(B)$) {};
			\node[label=above:$\bar{H}$] () at ($(Mid)$) {};
			\draw[dashed, thick] (B) to [bend right =20] (C);
			\end{scope}
			
			\end{scope}
			
			\end{tikzpicture}
		\end{subfigure}
		\begin{subfigure}{.3\linewidth}
			\begin{tikzpicture}[myNode/.style = black node, scale=.8]
			\begin{scope}
			\node[myNode, label=above:$x$] (A) at (90:2) {};
			\node[myNode, label=below left:$u_{1}$] (B) at (210:2) {};
			\node[myNode, label=below right:$u_{2}$] (C) at (330:2) {};

			\begin{scope}[on background layer]
			\fill[red!70!blue, opacity=0.7]
			(B.center) to [out=90, in=190] (A.center) to [out=150, in=120, min distance=2.5cm] (B.center);
			\fill[green!70!blue, opacity=0.7]
			(B.center) to [out=110, in=210] (A.center) to [out=180, in=150, min distance =2.5cm] (B.center);
			
			\coordinate (C1) at (170:2.3) {};
			\draw[-, name path = path1] (C1) circle (0.4);
			
			\path[name path=path2] (C1) to [bend right =20] (B);
			
			\draw [name intersections={of= path1 and path2}]
			\foreach \s in {1}{
				(intersection-\s) node {}
			};
			\node[myNode] (I1)  at (intersection-1) {};
			
			\draw[dashed,thick] (I1) to [bend right =20] (B);
			
			\end{scope}
			
			\begin{scope}[on background layer, xscale=-1]
			\fill[red!70!blue, opacity=0.7]
			(C.center) to [out=90, in=190] (A.center) to [out=150, in=120, min distance=2.5cm] (C.center); 
			\fill[green!70!blue, opacity=0.7]
			(C.center) to [out=110, in=210] (A.center) to [out=180, in=150, min distance =2.5cm] (C.center);

			\end{scope}
			
			\begin{scope}[on background layer]
			\fill [mustard!80] (B.center) to [bend left =60] (C.center) to [bend left = 60] (B.center);
			
			\coordinate (C2) at ($(B)!0.7!(C)$) {};
			\draw[-, name path = path1] (C2) circle (0.4);
			
			\path[name path=path2] (C2) to [bend left =20] (B);
			
			\draw [name intersections={of= path1 and path2}]
			\foreach \s in {1}{
				(intersection-\s) node {}
			};
			\node[myNode] (I2)  at (intersection-1) {};
			
			\draw[dashed,thick] (I2) to [bend left =20] (B);
			\end{scope}
			
			\end{scope}
			
			\end{tikzpicture}
		\end{subfigure}
		\begin{subfigure}{.3\linewidth}
			\begin{tikzpicture}[myNode/.style = black node, scale=.8]
			\begin{scope}
			\node[myNode, label=above:$x$] (A) at (90:2) {};
			\node[myNode, label=below left:$u_{1}$] (B) at (210:2) {};
			\node[myNode, label=below right:$u_{2}$] (C) at (330:2) {};

			\begin{scope}[on background layer]
			\fill[red!70!blue, opacity=0.7]
			(B.center) to [out=90, in=190] (A.center) to [out=150, in=120, min distance=2.5cm] (B.center);
			\fill[green!70!blue, opacity=0.7]
			(B.center) to [out=110, in=210] (A.center) to [out=180, in=150, min distance =2.5cm] (B.center);
			
			\coordinate (C1) at (170:2.3) {};
			\draw[-, name path = path1] (C1) circle (0.4);
			
			\path[name path=path2] (C1) to [bend right =20] (B);
			
			\draw [name intersections={of= path1 and path2}]
			\foreach \s in {1}{
				(intersection-\s) node {}
			};
			\node[myNode] (I1)  at (intersection-1) {};
			
			\draw[dashed,thick] (I1) to [bend right =20] (B);
			
			\end{scope}
			
			\begin{scope}[on background layer, xscale=-1]
			\fill[red!70!blue, opacity=0.7]
			(C.center) to [out=90, in=190] (A.center) to [out=150, in=120, min distance=2.5cm] (C.center); 
			\fill[green!70!blue, opacity=0.7]
			(C.center) to [out=110, in=210] (A.center) to [out=180, in=150, min distance =2.5cm] (C.center); 
			
			\end{scope}
			
			\begin{scope}[on background layer]
			\fill [mustard!80] (B.center) to [bend left =60] (C.center) to [bend left = 60] (B.center);
			
			\draw[-, name path = path1] (C.center).. controls ++(150:2.5) and ++(210:2.5).. (C.center);
			
			\path[name path=path2] (C) to (B);
			
			\draw [name intersections={of= path1 and path2}]
			\foreach \s in {1}{
				(intersection-\s) node {}
			};
			\node[myNode](I3)  at (intersection-1) {};
			
			\draw[dashed,thick] (I3) to (B);
			\end{scope}
			
			\end{scope}
			
			\end{tikzpicture}
		\end{subfigure}
		
		\vspace{-2mm}\caption{The cycles $C_{1}$ and $C_{2}$ in  the proof of Claim 2.}
		\label{mystical}
	\end{figure}

	\noindent {\em Observation 2:} $\{u_{1}, u_{2}\}\subseteq V(C_{1}\cup C_{2})$. Indeed, suppose that there exists an $i\in [2]$ such that $u_{i}\notin V(C_{1}\cup C_{2})$. By Claim 2, $C_{1}, C_{2}$ are in $\bar{H}$ and Observation 1 implies that there exists a path connecting $C_{1}, C_{2}$ avoiding $u_{i}$. Hence,
	{$\bar{H}\setminus u_{i}\not\in{\cal P}$ and so $\obs({\cal P})\leq \bar{H}\setminus u_{i}$ which, together with the fact that $K_{4}^{-}\leq H_{i}$, 
    implies that} $\{\hyperref[magnetic]{{O}_{1}^{0}},\hyperref[narcotic]{{O}_{3}^{0}}\}\leq G$,  a contradiction. \medskip
	
	\noindent {\em Claim 3:} There exists a block $H'$ of $\bar{H}$ that contains both $C_{1}, C_{2}$ and is outerplanar. \smallskip
	
	\noindent {\em Proof of Claim 3:} We start with the following observation:
	
	\smallskip
	
	\noindent {\em Observation 3:} There exists some block $H'$ of $\bar{H}$ that contains both $u_{1}, u_{2}$.
	Indeed, suppose towards a contradiction, that there exists some cut-vertex $v\in V(\bar{H})$ separating $u_{1}, u_{2}$ in $\bar{H}$. Then, there exist $D_{1}, D_{2}\in {\cal C}(\bar{H}, v)$ such that $D_{1}\neq D_{2}, u_{1}\in V(D_{1})$ and $u_{2}\in V(D_{2})$. By Observation 2, we can assume that $C_{1}\subseteq D_{1}$ and $C_{2}\subseteq D_{2}$, which in turn implies that
	$\hyperref[softened]{O_{3}^{2}}\leq G$ (see \autoref{inchoate}), a contradiction.
	
	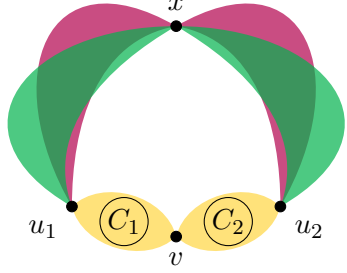
\begin{figure}[H]
		\centering
		\vspace{-1cm}
		\begin{tikzpicture}[myNode/.style = black node, scale=.8]
		\begin{scope}
		\node[myNode, label=above:$x$] (A) at (90:2) {};
		\node[myNode, label=below left:$u_{1}$] (B) at (210:2) {};
		\node[myNode, label=below right:$u_{2}$] (C) at (330:2) {};
		\node[myNode, label=below:$v$] (v) at (270:1.5) {};

		\begin{scope}[on background layer]
		\fill[red!70!blue, opacity=0.7]
		(B.center) to [out=90, in=190] (A.center) to [out=150, in=120, min distance=2.5cm] (B.center);
		\fill[green!70!blue, opacity=0.7]
		(B.center) to [out=110, in=210] (A.center) to [out=180, in=150, min distance =2.5cm] (B.center);
		\end{scope}
		
		\begin{scope}[on background layer, xscale=-1]
		\fill[red!70!blue, opacity=0.7]
		(C.center) to [out=90, in=190] (A.center) to [out=150, in=120, min distance=2.5cm] (C.center); 
		\fill[green!70!blue, opacity=0.7]
		(C.center) to [out=110, in=210] (A.center) to [out=180, in=150, min distance =2.5cm] (C.center); 
		\end{scope}
		
		\begin{scope}[on background layer]
		\fill [mustard!80] (B.center) to [bend left =60] (v.center) to [bend left = 60] (B.center);
		
		\coordinate (d1) at ($(v)!0.5!(B)$) {};
		\draw[-] (d1) circle (0.4);
		\node () at ($(d1)$) {$C_{1}$};
		
		\fill [mustard!80] (C.center) to [bend left =60] (v.center) to [bend left = 60] (C.center);
		
		\coordinate (d2) at ($(v)!0.5!(C)$) {};
		\draw[-] (d2) circle (0.4);
		\node () at ($(d2)$) {$C_{2}$};

		\end{scope}
		
		\end{scope}
		
		\end{tikzpicture}
		\vspace{-2mm}\caption{The cycles $C_{1}, C_{2}$ in the proof of Observation 3.}
		\label{inchoate}
	\end{figure}
		
	\smallskip
	
	According to Observation 3, let $H'$ be a block of $\bar{H}$ that contains both $u_{1}, u_{2}$. Suppose that some $C_{i}, \ i\in[2]$, say $C_{1}$, is not in $H'$.
	Then, there exists some cut-vertex $u$ of $\bar{H}$ such that $C_{1}$ is in some $H''\in {\cal C}(\bar{H}, u)$ that does not contain $H'$ as a subgraph.
	Observation 1 implies that $u\notin\{u_{1}, u_{2}\}$.
	Therefore, $\hyperref[adorning]{O_{3}^{1}}\leq G$, a contradiction.
	Thus, $H'$ contains $C_{1}, C_{2}$.

    {We now prove that $H'$ is outerplanar. We have that $K_{4}\not\leq G$ and
    so $K_{4}\not\leq H'$. Hence, it suffices to prove that $K_{2,3}\not\leq H'$.
    For that, suppose, towards a contradiction, that $K_{2,3}\leq H'$. Then, 
    \autoref{negative} implies that $H'$ contains a b-rich separator $S$. But then
    $S$ is also a b-rich separator of $G$. Indeed, this holds since 
    $K_{4}$-freeness of $G$ implies that every two $A,B\in{\cal C}(H',S)$ that 
    contain $S$-blocks are not connected in $G\setminus S$. Now, Claim 1 implies that
    $x\in S$, a contradiction to the fact that
     $S\subseteq H'\subseteq \bar{H}$. Claim 3 follows.
     }
	
	\medskip
	
	We now return to the proof of the lemma. According to  Claim 3, let $H'$ be a block of $\bar{H}$ that contains both $C_{1}, C_{2}$ and is outerplanar. {Since $H'$ is biconnected and outerplanar, then by
	\autoref{decoding}, $H'$ contains a Hamiltonian cycle, namely $C$. }
	Since $C_{1}, C_{2}\subseteq H'$, there exists some chord $e$ of $C$. To complete the proof of the Lemma, we distinguish the following cases:
	
	\smallskip
	
	\noindent {\em Case 1:} $u_{1}u_{2}\not\in E(C)$.
	Let $P_{1}, P_{2}$ be the connected components of $(C\setminus u_{1})\setminus u_{2}$. Notice that $K_{4}$-freeness of $G$ implies that $e$ is incident to vertices of some $P_{i}, \ i\in [2]$.  But then, $\hyperref[conceive]{O_{11}^{2}}\leq G$, a contradiction (see leftmost figure in \autoref{fakfbask}).
	
	\smallskip
	
	\noindent {\em Case 2:} $u_{1}u_{2}\in E(C)$. Then, the existence of $e$ implies
	that $\hyperref[supports]{O_{10}^{2}}\leq G$, a contradiction (see rightmost figure of \autoref{fakfbask}).

	\begin{figure}[h]
		\centering
		\begin{subfigure}{0.3\linewidth}
			\begin{tikzpicture}[myNode/.style = black node, scale=.8]
			\begin{scope}
			\node[myNode, label=above:$x$] (A) at (90:2) {};
			\node[myNode, label=below left:$u_{1}$] (B) at (210:2) {};
			\node[myNode, label=below right:$u_{2}$] (C) at (330:2) {};

			\begin{scope}[on background layer]
			\fill[red!70!blue, opacity=0.7]
			(B.center) to [out=90, in=190] (A.center) to [out=150, in=120, min distance=2.5cm] (B.center);
			\fill[green!70!blue, opacity=0.7]
			(B.center) to [out=110, in=210] (A.center) to [out=180, in=150, min distance =2.5cm] (B.center);
			\end{scope}
			
			\begin{scope}[on background layer, xscale=-1]
			\fill[red!70!blue, opacity=0.7]
			(C.center) to [out=90, in=190] (A.center) to [out=150, in=120, min distance=2.5cm] (C.center); 
			\fill[green!70!blue, opacity=0.7]
			(C.center) to [out=110, in=210] (A.center) to [out=180, in=150, min distance =2.5cm] (C.center); 
			\end{scope}
			
			\begin{scope}[on background layer]
			\fill [mustard!80] (B.center) to [bend left =60] (C.center) to [bend left = 60] (B.center);
			\draw[dashed, thick] (B.center) to [bend left =60] (C.center) to [bend left = 60] (B.center);

			\path[name path=path1] (215:2.5) to (325:2.5);
			\path[name path=path2] (B.center) to [bend right =60] (C.center);
			
			\draw [name intersections={of= path1 and path2}]
			\foreach \s in {1}{
				(intersection-\s) node {}
			};
			\node[myNode](I1)  at (intersection-1) {};
			\node[myNode](I2)  at (intersection-2) {};
			
			\draw[thick] (I1.center) to (I2.center);
			
			\coordinate (Mid) at ($(C)!0.5!(B)$) {};
			\node[label=center:$e$] () at ($(Mid)$) {};
			\end{scope}
			
			\end{scope}
			\end{tikzpicture}
		\end{subfigure}
		\begin{subfigure}{0.3\linewidth}
			\begin{tikzpicture}[myNode/.style = black node, scale=.8]
			\begin{scope}
			\node[myNode, label=above:$x$] (A) at (90:2) {};
			\node[myNode, label=below left:$u_{1}$] (B) at (210:2) {};
			\node[myNode, label=below right:$u_{2}$] (C) at (330:2) {};

			\begin{scope}[on background layer]
			\fill[red!70!blue, opacity=0.7]
			(B.center) to [out=90, in=190] (A.center) to [out=150, in=120, min distance=2.5cm] (B.center);
			\fill[green!70!blue, opacity=0.7]
			(B.center) to [out=110, in=210] (A.center) to [out=180, in=150, min distance =2.5cm] (B.center);
			\end{scope}
			
			\begin{scope}[on background layer, xscale=-1]
			\fill[red!70!blue, opacity=0.7]
			(C.center) to [out=90, in=190] (A.center) to [out=150, in=120, min distance=2.5cm] (C.center); 
			\fill[green!70!blue, opacity=0.7]
			(C.center) to [out=110, in=210] (A.center) to [out=180, in=150, min distance =2.5cm] (C.center); 
			\end{scope}
			
			\begin{scope}[on background layer]
			\fill [mustard!80] (B.center) to (C.center) to [bend left = 80] (B.center);
			\draw[dashed, thick] (B.center) to [bend right =80] (C.center);
			\draw (B.center) to (C.center);
			
			\path[name path=path1] (230:2.5) to (325:2.5);
			\path[name path=path2] (B.center) to [bend right =80] (C.center);
			
			\draw [name intersections={of= path1 and path2}]
			\foreach \s in {1}{
				(intersection-\s) node {}
			};
			\node[myNode](I1)  at (intersection-1) {};
			\node[myNode](I2)  at (intersection-2) {};
			
			\draw[thick] (I1.center) to (I2.center);
			\coordinate (Mid) at ($(C)!0.5!(B)$) {};
			\node[label=below:$e$] () at ($(Mid)$) {};
			\end{scope}
			
			\end{scope}
			\end{tikzpicture}
		\end{subfigure}
		\caption{The chord $e$ of the Hamiltonian cycle $C$ of $H'$ in the case analysis in the end of the proof of \autoref{symphony}.}\label{fakfbask}
	\end{figure}
	
	This completes the proof of the Lemma.
\end{proof}





\medskip

We now prove the following result:

\begin{lemma}\label{outerplanarfl}
Let $G$ be a $K_{4}$-free graph that contains a b-rich separator $S$. If $S$ is the unique b-rich separator of $G$, then every $H\in {\cal C}(G,S)$ is an outerplanar graph.
\end{lemma}

\begin{proof}
	Let $S$ be the unique b-rich separator of $G$ and let $H\in {\cal C}(G,S)$. To prove that $H$ is an outerplanar graph, it we prove that $\{K_4, K_{2,3}\}\not\leq H$. Since $K_{4}\nleq G$ implies that $K_{4}\nleq H$, we now aim to argue that $K_{2,3}\nleq H$.
	Suppose to the contrary that $K_{2,3}\leq H$, which by 
	\autoref{negative} implies that $H$ contains a b-rich separator $S'$, where $S'\neq S$. Since $K_{4}\nleq G$, we have that every two $A,B\in{\cal C}(H,S')$ that contain $S'$-blocks are not connected in $G\setminus S'$ and therefore $S'$ is a b-rich separator of $G$, a contradiction to the uniqueness of $S$.
\end{proof}

Let $G$ be a $K_4$-free biconnected graph with $K_{2,3} \leq G$ and $S$ a b-rich-separator of $G$, such that every $H\in {\cal C}(G,S)$ is outerplanar. For an $H\in {\cal C}(G,S)$, we denote $C_{H}$ the Hamiltonian cycle of $H$, which exists due to \autoref{bicoobs} and
\autoref{decoding}.
Given an $x\in V(G)$, we say that an edge $e\in E(G)$ is an {\em x-chord} of $G$, if there exists an $ H\in {\cal C}(G,S)$ such that $e$ is a chord of $C_{H}$ incident to $x$.
Also, given two vertices $x,y\in V(G)$, we say that an edge $e$ is an {\em (x,y)-disjoint chord} of $G$ if there exists an $ H\in {\cal C}(G,S)$ such that $e$ is a chord of $C_{H}$ disjoint to $x,y$.

We conclude this subsection by proving the following result:

\begin{lemma}
	\label{caillois}
{Let $G$ be a $K_{4}$-free biconnected graph such that ${\cal O}\nleq G$ and $K_{2,3}\leq G$. Also, let $S = \{x,y\}$ be a b-rich separator of $G$.
	If every $H\in {\cal C}(G,S)$ is an outerplanar graph and  if there exist at least two augmented connected components in ${\cal C}(G,S)$ not isomorphic to a cycle, then one of the following holds:
	\begin{itemize}
		
		\item There exists a unique $(x,y)$-disjoint chord of $G$ and there do not exist both $x$-chords and $y$-chords of $G$, or
		
		\item There do not exist $(x,y)$-disjoint chords of $G$ and there exists at most one $x$-chord or at most one $y$-chord of $G$.
		
\end{itemize}}
\end{lemma}

\begin{proof}
	Suppose that every $H\in {\cal C}(G,S)$ is outerplanar and there exist at least
	two augmented connected components in ${\cal C}(G,S)$ not isomorphic to a cycle.
	Note that since $G$ is biconnected, it holds that every $H\in {\cal C}(G,S)$ is also biconnected and so by \autoref{decoding}, for every $H$ we can consider its 
	Hamiltonian cycle, which we denote by $C_{H}$.
	Keep in mind that if some $H\in {\cal C}(G,S)$ is not isomorphic to a cycle, then $C_{H}$ contains some chord.
	Also, observe that $K_{4}$-freeness of $G$ implies that $xy\in E(C_{H})$.
	
	\medskip
	
	\noindent {\em Claim 1:} There exists at most one $(x,y)$-disjoint chord in $G$.
	
	\smallskip
	
	\noindent {\em Proof of Claim 1:}
	Suppose to the contrary that there exist two $(x,y)$-disjoint chords, namely $e,e'$.
	{If there exists some $H\in{\cal C}(G,S)$ such that $e,e'$ are chords of
	$C_{H}$, then we have that $\obs({\cal P})\leq H\setminus S$. Therefore,
	by our assumption that there exists some $H'\in {\cal C}(G,S)$ different than $H$ that contains a chord implies that
	$K_{4}^{-}\leq G\setminus( H\setminus S)$ and hence
	$\{\hyperref[magnetic]{O_{1}^{0}},\hyperref[narcotic]{O_{3}^{0}}\}\leq G$, a contradiction.
	Therefore, there exist different $H,H'\in {\cal C}(G,S)$ such that
	$e,e'$ are chords of $C_{H}, C_{H'}$, respectively.} Hence,
	$\hyperref[straight]{O_{15}^{2}}\leq G$ (see \autoref{shelters}),
	a contradiction. Claim 1 follows.
	
	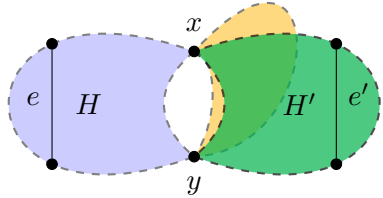
\begin{figure}[H]
		\centering
		\vspace{-1.5cm}
		\begin{tikzpicture}[x  = {(0:1cm)}, y  = {(90:1cm)},z= {(30:0.8cm)}, myNode/.style = black node,scale=0.7]
		
		\begin{scope}
		\node[label=above:$x$] (A) at (0,2,0) {};
		\node[label=below:$y$] (B) at (0,0,0) {};
		
		\begin{scope}[canvas is yz plane at x=0]
		
		\filldraw[fill=red!30!yellow,  opacity=0.5,dashed, thick]
		(A.center) to [bend right=60, min distance=0.5cm] (B.center) to [bend left =110, min distance =4cm] (A.center);
		
		\end{scope}

		\begin{scope}[canvas is xy plane at z=0]
		
		\begin{scope}
		\filldraw[fill=blue!40,  opacity=0.5,dashed, thick, name path=path1]
		(A.center) to [bend right=50, min distance=1cm] (B.center) to [bend left =110, min distance =5cm] (A.center);
		
		\node () at (-2,1) {$H$};
		
		\path[name path=path2] (-2.7,3) -- (-2.7,-1);
		
		\draw [name intersections={of= path1 and path2}]
		\foreach \s in {1,2}{
			(intersection-\s) node {}
		};
		\node[myNode] (I1)  at (intersection-1) {};
		\node[myNode] (I2)  at (intersection-2) {};
		
		\draw[-] (I1) to coordinate[pos=.7] (e1) (I2);
		\end{scope}
		
		\begin{scope}
		\filldraw[fill=green!70!blue,  opacity=0.7, dashed, thick, name path=path3]
		(A.center) to [bend left=50, min distance =1cm] (B.center) to [bend right =110, min distance =5cm] (A.center); 
		
		\node () at (2,1) {$H'$};

		\path[name path=path4] (2.7,3) -- (2.7,-1);
		
		\draw [name intersections={of= path3 and path4}]
		\foreach \s in {1,2}{
			(intersection-\s) node {}
		};
		\node[myNode] (I3)  at (intersection-1) {};
		\node[myNode] (I4)  at (intersection-2) {};
		
		\draw[-] (I3) to coordinate[pos=.7] (e2) (I4);
		
		\end{scope}
		
		\end{scope}
		
		\node[myNode] () at (A) {};
		\node[myNode] () at (B) {};
		\node[label=100:$e$] () at ($(e1)$) {};
		\node[label=80:$e'$] () at ($(e2)$) {};
		
		\end{scope}
		\end{tikzpicture}
		\vspace{-0.5cm}
		\vspace{-2mm}\caption{The chords $e, e'$ in the second part of the proof of Claim 1.}
		\label{shelters}
	\end{figure}
	
	We now distinguish the following cases depending on whether there exists an $(x,y)$-disjoint chord:
	
	\medskip
	
	\noindent{\em Case 1:} There exists an $(x,y)$-disjoint chord of $G$.
	
	\noindent Let $e$ be an $(x,y)$-disjoint chord of $G$ and let $H\in {\cal C}(G,S)$ be the
	augmented component, where $e$ is a chord of $C_{H}$. Claim 1 implies
	that $e$ is the unique $(x,y)$-disjoint chord of $G$ and thus every
	other chord of each $C_{H'}, \ H'\in {\cal C}(G,S)$, is either an $x$-chord or a $y$-chord of $G$.

	Recall that there exists some $H'\in{\cal C}(G,S)$ different than $H$ that is not isomorphic to a cycle.
	Therefore, $C_{H'}$ contains some chord $e'$ that is either an $x$-chord or a
	$y$-chord of $G$. Assume, without loss of generality, that $e'$ is an $x$-chord of $G$.
	We prove the following claim:
	
	\medskip
	
	\noindent {\em Claim 2:} Every edge of $G$ that is a chord of some
	$C_{H''}, \ H''\in {\cal C}(G,S)$, different from $e$ is an $x$-chord of $G$.
	
	\smallskip
	
	\noindent {\em Proof of Claim 2:} Suppose, towards a contradiction, that there exists
	some $H''\in {\cal C}(G,S)$ such that $C_{H''}$ contains a chord $e''$ different
	from $e$ which is not an $x$-chord of $G$. Then, Claim 1 implies that $e''$ is a $y$-chord of $G$.
	Observe that $H''\in\{H, H'\}$, because otherwise
	$\hyperref[pensions]{O_{8}^{1}}\leq G$, a contradiction.
	Therefore, if $H''=H$, then 
	$\{\hyperref[typifies]{O_{12}^{2}},\hyperref[solution]{O_{13}^{2}}\}\leq G$
	(see left and central figure of \autoref{primeval}), while if $H''=H'$,
	$\hyperref[uttering]{O_{7}^{2}}\leq G$ (see right figure of \autoref{primeval}),
	a contradiction in both cases. Claim 2 follows.	\vspace{-3mm}
	
	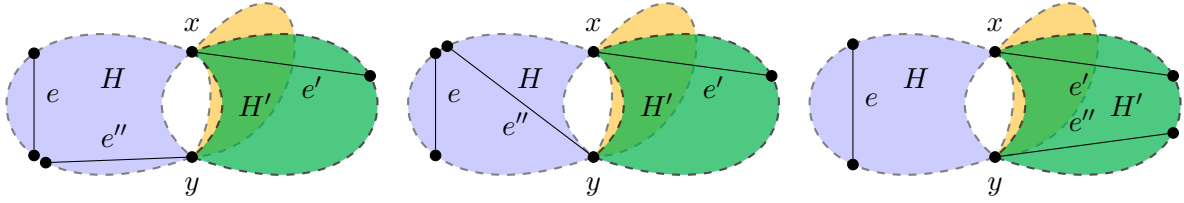
\begin{figure}[H]
	\vspace{-1cm}
	\hspace{-0.6cm}
	\begin{subfigure}{0.3\textwidth}
			\begin{tikzpicture}[x  = {(0:1cm)}, y  = {(90:1cm)},z= {(30:0.8cm)}, myNode/.style = black node, scale=.7]
			\begin{scope}
			
			\node[label=above:$x$] (A) at (0,2,0) {};
			\node[label=below:$y$] (B) at (0,0,0) {};
			
			\begin{scope}[canvas is yz plane at x=0]
			
			\filldraw[fill=red!30!yellow,  opacity=0.5,dashed, thick]
			(A.center) to [bend right=60, min distance=0.5cm] (B.center) to [bend left =110, min distance =4cm] (A.center);
			
			\end{scope}

			\begin{scope}[canvas is xy plane at z=0]
			
			\begin{scope}
			\filldraw[fill=blue!40,  opacity=0.5,dashed, thick, name path=path1]
			(A.center) to [bend right=50, min distance=1cm] (B.center) to [bend left =110, min distance =5cm]   coordinate[pos=0.27] (yy) (A.center);
			
			\node () at (-1.5,1.5) {$H$};
			
			\path[name path=path2] (-3,3) -- (-3,-1);
			
			\draw [name intersections={of= path1 and path2}]
			\foreach \s in {1,2}{
				(intersection-\s) node {}
			};
			\node[myNode] (I1)  at (intersection-1) {};
			\node[myNode] (I2)  at (intersection-2) {};
			
			\draw[-] (I1) to coordinate[pos=.7] (e1) (I2);
			\end{scope}
			
			\begin{scope}
			\filldraw[fill=green!70!blue,  opacity=0.7, dashed, thick]
			(A.center) to [bend left=50, min distance =1cm] (B.center) to [bend right =110, min distance =5cm]coordinate[pos=0.6] (xx) (A.center); 
			
			\node () at (1.2,1) {$H'$};

			\draw[-] (A.center) to coordinate[pos=.5] (e2) (xx) node[myNode] () {};
			\draw[-] (B.center) to coordinate[pos=.3] (e3) (yy) node[myNode] () {};
			
			\end{scope}
			
			\end{scope}
			
			\node[myNode] () at (A) {};
			\node[myNode] () at (B) {};
			\node[label=80:$e$] () at ($(e1.center)$) {};
			\node[label=-10:$e'$] () at ($(e2.center)$) {};
			\node[label=170:$e''$] () at ($(e3.center)$) {};
			
			\end{scope}
			\end{tikzpicture}
	\end{subfigure}
	~
	\begin{subfigure}{0.3\textwidth}
			\begin{tikzpicture}[x  = {(0:1cm)}, y  = {(90:1cm)},z= {(30:0.8cm)}, myNode/.style = black node, scale=.7]
			\begin{scope}
			
			\node[label=above:$x$] (A) at (0,2,0) {};
			\node[label=below:$y$] (B) at (0,0,0) {};
			
			\begin{scope}[canvas is yz plane at x=0]
			
			\filldraw[fill=red!30!yellow,  opacity=0.5,dashed, thick]
			(A.center) to [bend right=60, min distance=0.5cm] (B.center) to [bend left =110, min distance =4cm] (A.center);
			
			\end{scope}

			\begin{scope}[canvas is xy plane at z=0]
			
			\begin{scope}
			\filldraw[fill=blue!40,  opacity=0.5,dashed, thick, name path=path1]
			(A.center) to [bend right=50, min distance=1cm] (B.center) to [bend left =110, min distance =5cm]   coordinate[pos=0.73] (yy) (A.center);
			
			\node () at (-1.2,1.5) {$H$};
			
			\path[name path=path2] (-3,3) -- (-3,-1);
			
			\draw [name intersections={of= path1 and path2}]
			\foreach \s in {1,2}{
				(intersection-\s) node {}
			};
			\node[myNode] (I1)  at (intersection-1) {};
			\node[myNode] (I2)  at (intersection-2) {};
			
			\draw[-] (I1) to coordinate[pos=.7] (e1) (I2);
			\end{scope}
			
			\begin{scope}
			\filldraw[fill=green!70!blue,  opacity=0.7, dashed, thick]
			(A.center) to [bend left=50, min distance =1cm] (B.center) to [bend right =110, min distance =5cm]  coordinate[pos=0.6] (xx) (A.center); 
			
			\node () at (1.2,1) {$H'$};

			\draw[-] (A.center) to coordinate[pos=.5] (e2) (xx) node[myNode] () {};
			\draw[-] (B.center) to coordinate[pos=.3] (e3) (yy) node[myNode] () {};
			
			\end{scope}
			
			\end{scope}
			
			\node[myNode] () at (A) {};
			\node[myNode] () at (B) {};
			\node[label=80:$e$] () at ($(e1)$) {};
			\node[label=-10:$e'$] () at ($(e2)$) {};
			\node[label=180:$e''$] () at ($(e3)$) {};
			
			\end{scope}
			\end{tikzpicture}
	\end{subfigure}
	~
	\begin{subfigure}{0.3\textwidth}		
			\begin{tikzpicture}[x  = {(0:1cm)}, y  = {(90:1cm)},z= {(30:0.8cm)}, myNode/.style = black node, scale=.7]
			\begin{scope}
			
			\node[label=above:$x$] (A) at (0,2,0) {};
			\node[label=below:$y$] (B) at (0,0,0) {};
			
			\begin{scope}[canvas is yz plane at x=0]
			
			\filldraw[fill=red!30!yellow,  opacity=0.5,dashed, thick]
			(A.center) to [bend right=60, min distance=0.5cm] (B.center) to [bend left =110, min distance =4cm] (A.center);
			
			\end{scope}

			\begin{scope}[canvas is xy plane at z=0]
			
			\begin{scope}
			\filldraw[fill=blue!40,  opacity=0.5,dashed, thick, name path=path1]
			(A.center) to [bend right=50, min distance=1cm] (B.center) to [bend left =110, min distance =5cm] (A.center);
			
			\node () at (-1.5,1.5) {$H$};
			
			\path[name path=path2] (-2.7,3) -- (-2.7,-1);
			
			\draw [name intersections={of= path1 and path2}]
			\foreach \s in {1,2}{
				(intersection-\s) node {}
			};
			\node[myNode] (I1)  at (intersection-1) {};
			\node[myNode] (I2)  at (intersection-2) {};
			
			\draw[-] (I1) to coordinate[pos=.7] (e1) (I2);
			\end{scope}
			
			\begin{scope}
			\filldraw[fill=green!70!blue,  opacity=0.7, dashed, thick]
			(A.center) to [bend left=50, min distance =1cm] (B.center) to [bend right =110, min distance =5cm]  coordinate[pos=0.4] (yy) coordinate[pos=0.6] (xx) (A.center); 
			
			\node () at (2.5,1) {$H'$};

			\draw[-] (A.center) to coordinate[pos=.3] (e2) (xx) node[myNode] () {};
			\draw[-] (B.center) to coordinate[pos=.3] (e3) (yy) node[myNode] () {};
			
			\end{scope}
			
			\end{scope}
			
			\node[myNode] () at (A) {};
			\node[myNode] () at (B) {};
			\node[label=80:$e$] () at ($(e1)$) {};
			\node[label=-20:$e'$] () at ($(e2)$) {};
			\node[label=40:$e''$] () at ($(e3)$) {};
			
			\end{scope}
			\end{tikzpicture}
	\end{subfigure}		
	\vspace{-0.5cm}
		\vspace{-2mm}\caption{The possible configurations of the chords $e,e', e''$ in the proof of Claim 2}
		\label{primeval}
	\end{figure}
	
	Therefore, in this case, $e$ is the unique $(x,y)$-disjoint chord of
	$G$ and there do not exist $y$-chords.
	
	\medskip 
	
	\noindent{\em Case 2:} There do not exist $(x,y)$-disjoint chords of $G$.
	
	In this case we prove that there exists at most one $x$-chord or at most one 
	$y$-chord of $G$, which will conclude the proof of the Lemma.
    For that, suppose, towards a contradiction, that there exist two
	$x$-chords of $G$, namely  $e_{1}^{x}$ and $e_{2}^{x}$, and two $y$-chords of $G$,
	namely $e_{1}^{y}$ and $e_{2}^{y}$.
	We say that a pair of edges $(e,e')$,
	where $e,e'\in\{e_{1}^{x},e_{2}^{x},e_{1}^{y},e_{2}^{y}\}$ is {\em homologous}
	if there exists some $H\in {\cal C}(G,S)$ such that $e,e'\in E(H)$.

	We now distinguish the following subcases:\medskip

	\noindent {\em Subcase 2.1:} $(e_{1}^{x},e_{2}^{x})$ is not homologous and $(e_{1}^{y},e_{2}^{y})$ is not homologous.
	Then, $\{\hyperref[consiste]{O_{2}^{0}},\hyperref[daylight]{O_{12}^{1}},\hyperref[softened]{O_{3}^{2}}\}\leq G$ depending whether there exist 0, 1 or 2 homologous pairs $(e,e')$ , where $e\in\{e_{1}^{x},e_{2}^{x}\}, e'\in\{e_{1}^{y},e_{2}^{y}\}$
	(see \autoref{striving}). In any case we have a contradiction.	\vspace{-3mm}
	
	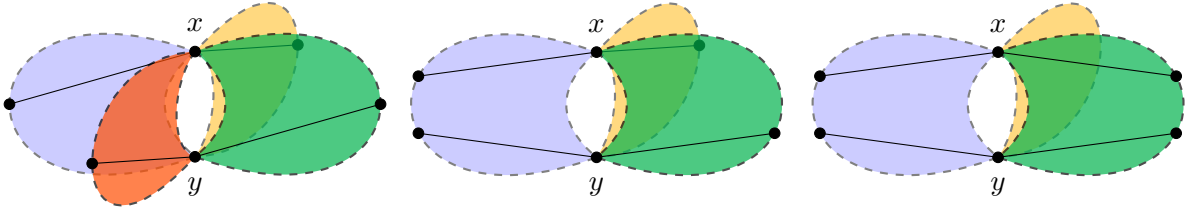
\begin{figure}[H]
	\vspace{-1cm}
	\hspace{-0.6cm}
		\begin{subfigure}{0.3\textwidth}
		\begin{tikzpicture}[x  = {(0:1cm)}, y  = {(90:1cm)},z= {(30:0.8cm)}, myNode/.style = black node, scale=.7]
		\begin{scope}
		
		\node[label=above:$x$] (A) at (0,2,0) {};
		\node[label=below:$y$] (B) at (0,0,0) {};
		
		\begin{scope}[canvas is yz plane at x=0]
		
		\filldraw[fill=red!30!yellow,  opacity=0.5,dashed, thick]
		(A.center) to [bend right=60, min distance=0.5cm] (B.center) to  [bend left =110, min distance =4cm] coordinate[pos=0.5] (e1) (A.center);
		
		\draw[-] (A.center) to (e1) node[myNode] () {};
		
		\end{scope}

		\begin{scope}[canvas is xy plane at z=0]
		
		\begin{scope}
		\filldraw[fill=blue!40,  opacity=0.5,dashed, thick]
		(A.center) to [bend right=50, min distance=1cm] (B.center) to [bend left =110, min distance =5cm] coordinate[pos=0.5] (e2) (A.center);

		\draw[-] (A.center) to (e2) node[myNode] () {};
		
		\end{scope}
		
		\begin{scope}
		\filldraw[fill=green!70!blue,  opacity=0.7, dashed, thick]
		(A.center) to [bend left=50, min distance =1cm] (B.center) to [bend right =110, min distance =5cm] coordinate[pos=0.5] (e3) (A.center); 
		
		\draw[-] (B.center) to (e3) node[myNode] () {};
		\end{scope}
		
		\end{scope}
		
		\begin{scope}[canvas is zy plane at x=0]
		
		\filldraw[fill=red!70!yellow,  opacity=0.7,dashed, thick]
		(A.center) to [bend right=60, min distance=0.5cm] (B.center) to [bend left =110, min distance =4cm] coordinate[pos=0.5] (e4) (A.center);
		
		\draw[-] (B.center) to (e4) node[myNode] () {};
		
		\end{scope}
		\node[myNode] () at (A) {};
		\node[myNode] () at (B) {};
		
		\end{scope}
		\end{tikzpicture}
	\end{subfigure}
	~
	\begin{subfigure}{0.3\textwidth}
	\begin{tikzpicture}[x  = {(0:1cm)}, y  = {(90:1cm)},z= {(30:0.8cm)}, myNode/.style = black node, scale=.7]
			\begin{scope}
			
			\node[label=above:$x$] (A) at (0,2,0) {};
			\node[label=below:$y$] (B) at (0,0,0) {};
			
			\begin{scope}[canvas is yz plane at x=0]
			
			\filldraw[fill=red!30!yellow,  opacity=0.5,dashed, thick]
			(A.center) to [bend right=60, min distance=0.5cm] (B.center) to [bend left =110, min distance =4cm] coordinate[pos=0.5] (e11) (A.center);
			
			\draw[-] (A.center) to (e11) node[myNode] () {};
			
			\end{scope}

			\begin{scope}[canvas is xy plane at z=0]
			
			\begin{scope}
			\filldraw[fill=blue!40,  opacity=0.5,dashed, thick]
			(A.center) to [bend right=50, min distance=1cm] (B.center) to [bend left =110, min distance =5cm] coordinate[pos=0.4] (e22) coordinate[pos=0.6] (e33) (A.center);
			
			\draw[-] (A.center) to (e33) node[myNode] () {};
			\draw[-] (B.center) to (e22) node[myNode] () {};
			
			\end{scope}
			
			\begin{scope}
			\filldraw[fill=green!70!blue,  opacity=0.7, dashed, thick]
			(A.center) to [bend left=50, min distance =1cm] (B.center) to [bend right =110, min distance =5cm] coordinate[pos=0.4] (e44) (A.center);

			\draw[-] (B.center) to (e44) node[myNode] () {};
			\end{scope}
			
			\end{scope}
			
			\begin{scope}[canvas is zy plane at x=0]
			
			\path (A.center) to [bend right=60, min distance=0.5cm] (B.center) to [bend left =110, min distance =4cm] coordinate[pos=0.5] (e4) (A.center);
			\end{scope}
			
			\node[myNode] () at (A) {};
			\node[myNode] () at (B) {};
			
			\end{scope}
	\end{tikzpicture}
	\end{subfigure}
	~
	\begin{subfigure}{0.3\textwidth}			
	\begin{tikzpicture}[x  = {(0:1cm)}, y  = {(90:1cm)},z= {(30:0.8cm)}, myNode/.style = black node, scale=.7]	
			\begin{scope}
			
			\node[label=above:$x$] (A) at (0,2,0) {};
			\node[label=below:$y$] (B) at (0,0,0) {};
			
			\begin{scope}[canvas is yz plane at x=0]
			
			\filldraw[fill=red!30!yellow,  opacity=0.5,dashed, thick]
			(A.center) to [bend right=60, min distance=0.5cm] (B.center) to [bend left =110, min distance =4cm] (A.center);
			
			\end{scope}
			
			\begin{scope}[canvas is xy plane at z=0]
			
			\begin{scope}
			\filldraw[fill=blue!40,  opacity=0.5,dashed, thick]
			(A.center) to [bend right=50, min distance=1cm] (B.center) to [bend left =110, min distance =5cm] coordinate[pos=0.4] (c1) coordinate[pos=0.6] (c2) (A.center);
			
			\draw[-] (B.center) to (c1) node[myNode] () {};
			\draw[-] (A.center) to (c2) node[myNode] () {};
			
			\end{scope}
			
			\begin{scope}
			\filldraw[fill=green!70!blue,  opacity=0.7, dashed, thick]
			(A.center) to [bend left=50, min distance =1cm] (B.center) to [bend right =110, min distance =5cm] coordinate[pos=0.4] (c3) coordinate[pos=0.6] (c4) (A.center); 
			
			\draw[-] (B.center) to (c3) node[myNode] () {};
			\draw[-] (A.center) to (c4) node[myNode] () {};
			
			\end{scope}
			
			\end{scope}
			
			\begin{scope}[canvas is zy plane at x=0]
			
			\path (A.center) to [bend right=60, min distance=0.5cm] (B.center) to [bend left =110, min distance =4cm] coordinate[pos=0.5] (e4) (A.center);
			\end{scope}
			
			\node[myNode] () at (A) {};
			\node[myNode] () at (B) {};
			
			\end{scope}
	\end{tikzpicture}
	\end{subfigure}
	\vspace{-1cm}
		\vspace{-2mm}\caption{The possible configurations of the edges $e_{1}^{y},e_{2}^{y}, e_{1}^{x},$ and $e_{2}^{x}$ in the proof of Subcase 2.1.}\label{striving}
	\end{figure}

	\noindent {\em Subcase 2.2:} $(e_{1}^{x},e_{2}^{x})$ is homologous and $(e_{1}^{y},e_{2}^{y})$ is not homologous.
	
	Let $H\in{\cal C}(G,S)$ be the component such that $e_{1}^{x},e_{2}^{x}\in E(H)$.
	Notice that since $(e_{1}^{y},e_{2}^{y})$ is not homologous, at most one of
	$e_{1}^{y},e_{2}^{y}$ is in $E(H)$. If none of $e_{1}^{y},e_{2}^{y}$ is in $E(H)$, then
	$\hyperref[narcotic]{O_{3}^{0}}\leq G$, while if some of $e_{1}^{y},e_{2}^{y}$ is in $E(H)$, then $\hyperref[daylight]{{O}_{11}^{1}}\leq G,$
	 a contradiction (see \autoref{blocking}).	\vspace{-2mm}
	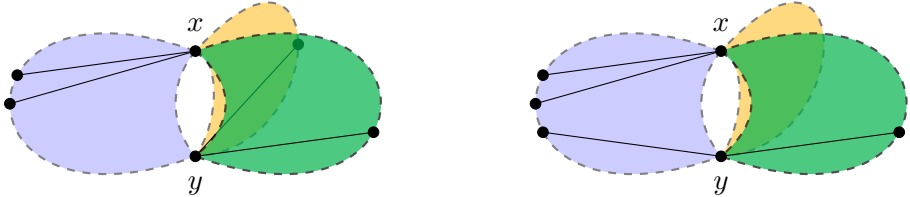
\begin{figure}[H]
	\centering
	\vspace{-1cm}	
	\begin{subfigure}{0.4\textwidth}
		\begin{tikzpicture}[x  = {(0:1cm)}, y  = {(90:1cm)},z= {(30:0.8cm)}, myNode/.style = black node, scale=.7]
		\begin{scope}
		
		\node[label=above:$x$] (A) at (0,2,0) {};
		\node[label=below:$y$] (B) at (0,0,0) {};
		
		\begin{scope}[canvas is yz plane at x=0]
		
		\filldraw[fill=red!30!yellow,  opacity=0.5,dashed, thick] (A.center) to [bend right=60, min distance=0.5cm] (B.center) to [bend left =110, min distance =4cm] coordinate[pos=0.5] (c4) (A.center);
		
		\draw[-] (B.center) to (c4) node[myNode] () {};
		\end{scope}

		\begin{scope}[canvas is xy plane at z=0]
		
		\begin{scope}
		\filldraw[fill=blue!40,  opacity=0.5,dashed, thick]
		(A.center) to [bend right=40, min distance=0.7cm] (B.center) to [bend left =110, min distance =5cm] coordinate[pos=0.5] (c1) coordinate[pos=0.6] (c2) (A.center);
		
		\draw[-] (A.center) to (c1) node[myNode] () {};
		\draw[-] (A.center) to (c2) node[myNode] () {};
		
		\end{scope}
		
		\begin{scope}
		\filldraw[fill=green!70!blue,  opacity=0.7, dashed, thick]
		(A.center) to [bend left=50, min distance =1cm] (B.center) to [bend right =110, min distance =5cm] coordinate[pos=0.4] (c3)  (A.center); 
		
		\draw[-] (B.center) to (c3) node[myNode] () {};
		
		\end{scope}
		
		\end{scope}

		\node[myNode] () at (A) {};
		\node[myNode] () at (B) {};

		\end{scope}
		
		\end{tikzpicture}
	\end{subfigure}
	~
	\begin{subfigure}{0.4\textwidth}
			\begin{tikzpicture}[x  = {(0:1cm)}, y  = {(90:1cm)},z= {(30:0.8cm)}, myNode/.style = black node, scale=.7]
			\begin{scope}[xshift=8cm]
			
			\node[label=above:$x$] (A) at (0,2,0) {};
			\node[label=below:$y$] (B) at (0,0,0) {};
			
			\begin{scope}[canvas is yz plane at x=0]
			
			\filldraw[fill=red!30!yellow,  opacity=0.5,dashed, thick] (A.center) to [bend right=60, min distance=0.5cm] (B.center) to [bend left =110, min distance =4cm] (A.center);
			
			\end{scope}

			\begin{scope}[canvas is xy plane at z=0]
			
			\begin{scope}
			\filldraw[fill=blue!40,  opacity=0.5,dashed, thick]
			(A.center) to [bend right=40, min distance=0.7cm] (B.center) to [bend left =110, min distance =5cm] coordinate[pos=0.4] (e3) coordinate[pos=0.5] (e1) coordinate[pos=0.6] (e2) (A.center);
			
			\draw[-] (A.center) to (e1) node[myNode] () {};
			\draw[-] (A.center) to (e2) node[myNode] () {};
			\draw[-] (B.center) to (e3) node[myNode] () {};

			\end{scope}
			
			\begin{scope}
			\filldraw[fill=green!70!blue,  opacity=0.7, dashed, thick]
			(A.center) to [bend left=50, min distance =1cm] (B.center) to [bend right =110, min distance =5cm] coordinate[pos=0.4] (e4) (A.center); 
			
			\draw[-] (B.center) to (e4) node[myNode] () {};
			
			\end{scope}
			
			\end{scope}
			
			\node[myNode] () at (A) {};
			\node[myNode] () at (B) {};
			
			\end{scope}
			\end{tikzpicture}
			\end{subfigure}
			\vspace{-.7cm}
		\vspace{-2mm}\caption{The possible configurations of the edges $e_{1}^{y},e_{2}^{y}, e_{1}^{x},$ and $e_{2}^{x}$ in the proof of Subcase 2.2.}\label{blocking}
	\end{figure}

	\noindent {\em Subcase 2.3:} $(e_{1}^{x},e_{2}^{x})$ is not homologous and $(e_{1}^{y},e_{2}^{y})$ is homologous. This case is symmetric to the previous one.
	
	\begin{figure}[H]
	\centering
	\vspace{-1.5cm}
	\begin{subfigure}{0.4\textwidth}
			\begin{tikzpicture}[x  = {(0:1cm)}, y  = {(90:1cm)},z= {(30:0.8cm)}, myNode/.style = black node, scale=.7]
			\begin{scope}

			\node[label=above:$x$] (A) at (0,2,0) {};
			\node[label=below:$y$] (B) at (0,0,0) {};
			
			\begin{scope}[canvas is yz plane at x=0]
			
			\filldraw[fill=red!30!yellow,  opacity=0.5,dashed, thick] (A.center) to [bend right=60, min distance=0.5cm] (B.center) to [bend left =110, min distance =4cm] (A.center);

			\end{scope}

			\begin{scope}[canvas is xy plane at z=0]
			
			\begin{scope}
			\filldraw[fill=blue!40,  opacity=0.5,dashed, thick]
			(A.center) to [bend right=40, min distance=0.7cm] (B.center) to [bend left =110, min distance =5cm] coordinate[pos=0.4] (c1) coordinate[pos=0.6] (c2) (A.center);
			
			\draw[-] (A.center) to (c1) node[myNode] () {};
			\draw[-] (A.center) to (c2) node[myNode] () {};
			
			\end{scope}
			
			\begin{scope}
			\filldraw[fill=green!70!blue,  opacity=0.7, dashed, thick]
			(A.center) to [bend left=50, min distance =1cm] (B.center) to [bend right =110, min distance =5cm] coordinate[pos=0.4] (c3) coordinate[pos=0.6] (c4)  (A.center); 
			
			\draw[-] (B.center) to (c3) node[myNode] () {};
			\draw[-] (B.center) to (c4) node[myNode] () {};
			\end{scope}
			
			\end{scope}

			\node[myNode] () at (A) {};
			\node[myNode] () at (B) {};

			\end{scope}
	\end{tikzpicture}
	\end{subfigure}
	~
	\begin{subfigure}{0.4\textwidth}
			\begin{tikzpicture}[x  = {(0:1cm)}, y  = {(90:1cm)},z= {(30:0.8cm)}, myNode/.style = black node, scale=.7]
			\begin{scope}
			
			\node[label=above:$x$] (A) at (0,2,0) {};
			\node[label=below:$y$] (B) at (0,0,0) {};
			
			\begin{scope}[canvas is yz plane at x=0]
			
			\filldraw[fill=red!30!yellow,  opacity=0.5,dashed, thick] (A.center) to [bend right=60, min distance=0.5cm] (B.center) to [bend left =110, min distance =4cm] (A.center);
			
			\end{scope}

			\begin{scope}[canvas is xy plane at z=0]
			
			\begin{scope}
			\filldraw[fill=blue!40,  opacity=0.5,dashed, thick]
			(A.center) to [bend right=40, min distance=0.7cm] (B.center) to [bend left =110, min distance =5cm] coordinate[pos=0.35] (e3) coordinate[pos=0.45] (e4) coordinate[pos=0.55] (e1) coordinate[pos=0.65] (e2) (A.center);
			
			\draw[-] (A.center) to (e1) node[myNode] () {};
			\draw[-] (A.center) to (e2) node[myNode] () {};
			\draw[-] (B.center) to (e3) node[myNode] () {};
			\draw[-] (B.center) to (e4) node[myNode] () {};
			
			\end{scope}
			
			\begin{scope}
			\filldraw[fill=green!70!blue,  opacity=0.7, dashed, thick]
			(A.center) to [bend left=50, min distance =1cm] (B.center) to [bend right =110, min distance =5cm] (A.center);

			\end{scope}
			
			\end{scope}
			
			\node[myNode] () at (A) {};
			\node[myNode] () at (B) {};
			
			\end{scope}
			\end{tikzpicture}
	\end{subfigure}	
		\vspace{-0.5cm}
		\vspace{-2mm}\caption{The possible configurations of the cycles $e_{1}^{y},e_{2}^{y}, e_{1}^{x},$ and $e_{2}^{x}$ in the proof of Subcase 2.4.}\label{gasoline}
	\end{figure}
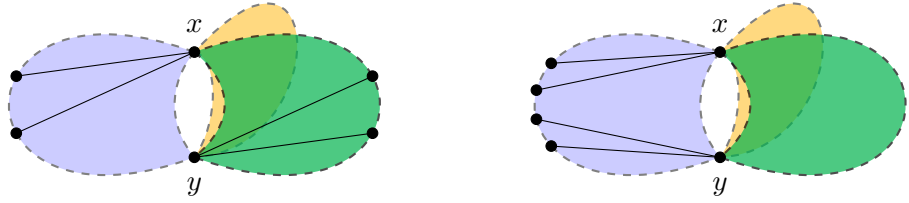
	
	\noindent {\em Subcase 2.4:} $(e_{1}^{x},e_{2}^{x})$ is homologous and $(e_{1}^{y},e_{2}^{y})$ is homologous.
	Let $H\in{\cal C}(G,S)$ be the component such that $e_{1}^{x},e_{2}^{x}\in E(H)$ and
	$H'\in{\cal C}(G,S)$ be the component such that $e_{1}^{y},e_{2}^{y}\in E(H')$.
	If $H\not=H'$, then  $\hyperref[magnetic]{O_{1}^{0}}\leq G$, while if $H=H'$,
	then {$\hyperref[solution]{{O}_{13}^{2}}\leq G.$} In both cases we have a contradiction (see \autoref{gasoline}).
\end{proof}

\section{Confining connectivity}
\label{refusing}

In this section we further restrict the structure of a graph $G\in \obs({\cal A}_{1}({\cal P}))\setminus  {\cal O}$. The first step is to prove that $G$ is biconnected (\autoref{destined}) and the second one is
to prove that $G$ is triconnected (\autoref{personal}).

\subsection{Proving biconnectivity}

In this section we prove that every graph in  $\obs({\cal A}_{1}({\cal P}))\setminus  {\cal O}$ is biconnected (\autoref{destined}). For this we prove a series of lemmata that gradually restrict the structure of such a graph.

We begin by making two observations.
We have that \autoref{generals} implies that every block of a graph $G\in \obs({\cal A}_{1}({\cal P}))\setminus  {\cal O}$ has a cycle and so by the $\hyperref[teaching]{{O}_{10}^{1}}$-freeness of such a graph we derive the following:

\begin{observation}\label{theories}
	If $G\in \obs({\cal A}_{1}({\cal P}))\setminus  {\cal O}$ then every block of $G$ contains at most 2 cut-vertices.
\end{observation}

Also, for a graph $G\in \obs({\cal A}_{1}({\cal P}))\setminus  {\cal O}$ we have that $G\not\in {\cal A}_{1}({\cal P})$ and this implies the following:

\begin{observation}\label{divorces}
	If $G\in \obs({\cal A}_{1}({\cal P}))\setminus  {\cal O}$ is a connected graph, then
	for every cut-vertex $v\in V(G)$ there exists an $H\in {\cal C}(G,v)$ such that
	$H\setminus v\not\in{\cal P}$, or equivalently $\obs({\cal P})\leq H\setminus v$.
\end{observation}

\begin{lemma}  \label{busyness}
	If $G\in \obs({\cal A}_{1}({\cal P}))\setminus  {\cal O}$, then $G$ cannot have more than 1 cut-vertex.
\end{lemma}

\begin{proof}
	Recall that, from~\autoref{hundreds}, $G$ is connected.
	Suppose then, towards a contradiction, that $G$ has at least two cut-vertices.
	Then, there exists a block $B$ containing two cut-vertices $u_{1},u_{2}$, which due to \autoref{generals} is not a bridge.
	Let $$D_{1} = \cupall \{ H \in {\cal C}(G,u_{1}) : u_{2} \notin V(H) \}\mbox{~~and~~}D_{2} = \cupall \{ H \in {\cal C}(G,u_{2}) : u_{1} \notin V(H) \},$$ 
	We now prove a series of claims:
	
	\medskip
	
	\noindent{\em Claim 1:} Both $D_{1}, D_{2}$ are isomorphic to $K_{3}.$
	
	\noindent{\em Proof of Claim 1:}
	Suppose, to the contrary, that one of $D_{1}, D_{2}$, say $D_{1}$,
	contains two cycles, which is equivalent to $\obs({\cal P})\leq D_{1}$,
	since $D_{1}$ is connected. Let $H\in {\cal C}(G,u_{1})$ be the component that contains $u_{2}$.
	We distinguish two cases:
	
	\smallskip
	
	\noindent {\em Case 1:} $D_{1}\setminus u_{1}\in {\cal P}$. Then, by
	\autoref{divorces}, $\obs({\cal P})\leq H\setminus u_{1}$ and therefore, since
	$V(D_{1})\cap V(H\setminus u_{1})=\emptyset$ and $\obs({\cal P})\leq D_{1}$,
	we have that $ {{\cal O}^{0}}\leq G$, a contradiction.
	
	\smallskip
	
	\noindent {\em Case 2:} $D_{1}\setminus u_{1}\not\in {\cal P}$, or equivalently $\obs({\cal P})\leq D_{1}\setminus u_{1}$.
	Observe that, since $H$ contains the cut-vertex $u_{2}$, there exist two blocks $H_{1},H_{2}$ of $G$ in $H$ such that $V(H_{1})\cap V(H_{2})=\{u_{2}\}$.
	Then, since, by \autoref{generals}, each block of $G$ contains a cycle, we have that the butterfly graph $Z$ is a minor of $H$.
	Hence, $\{\hyperref[consiste]{{O}_{2}^{0}},\hyperref[narcotic]{{O}_{3}^{0}}\}\leq G$, a contradiction.
	
	\smallskip
	
	Therefore, both $D_{1}, D_{2}$ contain at most one cycle and, since both are non-empty, \autoref{generals}(1) implies Claim 1.
	
	\medskip
	
	\noindent{\em Claim 2:} Every cycle in $B$ contains either $u_{1}$ or $u_{2}$.
	
	\noindent{\em Proof of Claim 2:}
	Suppose, to the contrary, that there exists a cycle $C$ containing neither $u_{1}$ nor $u_{2}$. By Menger's Theorem there exist two internally vertex disjoint $(u_{1},u_{2})$-paths $P_{1}, P_{2}$. We distinguish the following cases:
	
	\smallskip
	
	\noindent {\em Case 1:} Both of $P_{1},P_{2}$ intersect $C$.
	
	Let $z_{1},z_{2}$ be the vertices where $P_{1},P_{2}$ meet $C$ for the
	first time, respectively. Let, also, $w_{1},w_{2}$ be the vertices
	that $P_{1},P_{2}$ meet $C$ for the last time, respectively. Since
	$V(P_{1})\cap V(P_{2})=\{u_{1},u_{2}\}$ we have that
	$\{z_{1}, w_{1}\} \cap \{z_{2},w_{2}\}=\emptyset$.
	If $z_{1}\neq w_{1}$ or $z_{2}\neq w_{2}$, say $z_{1}\neq w_{1}$,
	then by contracting the edges in the $(w_{1},u_{2})$-subpath of
	$P_{1}$ we form $\hyperref[consists]{O_{1}^{1}}$ as a minor of $G$, a contradiction (see left figure of \autoref{concepts}).
	Therefore, we have that $z_{1}=w_{1}$ and $z_{2}=w_{2}$, 
	{i.e. both $P_{1},P_{2}$ intersect $C$ in exactly one vertex,}
	in which
	case we have again $\hyperref[consists]{O_{1}^{1}}$ as a minor of $G$, a contradiction (see right figure of \autoref{concepts}).	\vspace{-3mm}
	
	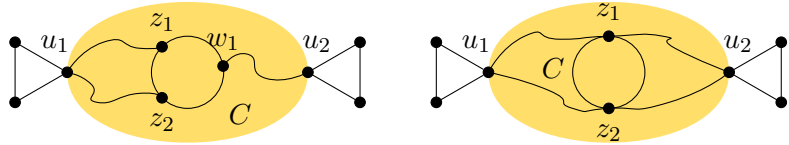
\begin{figure}[H]
		\centering
		\begin{tikzpicture}[myNode/.style = black node, scale=.8]
		
		\begin{scope}
		\node[myNode] (u1) at (0,0) {};
		\node[label=above:$u_1$] () at (-0.2,0) {};
		\node[myNode]  (u2) at (4,0) {};
		\node[label=above:$u_2$] () at (4.15,0) {};
		
		\draw[-] (u1) to (150:1) node[myNode] {} to (210:1) node[myNode] {} to (u1);
		
		\begin{scope}[xscale=-1, xshift=-4cm]
		\draw[-] (0,0) to (150:1) node[myNode] {} to (210:1) node[myNode] {} to (0,0);
		\end{scope}
		
		\begin{scope}[xshift=2cm]
		\draw (0,0) circle (0.6);
		\node[myNode, label=above:$z_1$] (z1) at (135:0.6) {};
		\node[myNode, label=below:$z_2$] (z2) at (225:0.6) {};
		\node[myNode, label=above:$w_1$] (w1) at (10:0.6) {};
		\node[label=below right:$C$] () at (-30:0.4) {};
		\end{scope}
		
		\draw plot [smooth, tension=1.2] coordinates { (u1) (0.6,0.5) (1.2,0.4) (z1)};
		\draw plot [smooth, tension=1.2] coordinates { (u1) (0.4,-0.2) (0.5,-0.5) (1.2,-0.35) (z2)};
		\draw plot [smooth, tension=1.2] coordinates { (w1) (3,0.3) (3.3,-0.1) (u2)};
		
		\begin{scope}[on background layer]
		\fill[mustard!90] (u1.center) to [bend left = 90] (u2.center) to [bend left=90] (u1.center);
		\end{scope}
		\end{scope}
		
		\begin{scope}[xshift=7cm]
		\node[myNode] (u1) at (0,0) {};
		\node[label=above:$u_1$] () at (-0.2,0) {};
		\node[myNode]  (u2) at (4,0) {};
		\node[label=above:$u_2$] () at (4.15,0) {};
		
		\draw[-] (u1) to (150:1) node[myNode] {} to (210:1) node[myNode] {} to (u1);
		
		\begin{scope}[xscale=-1, xshift=-4cm]
		\draw[-] (0,0) to (150:1) node[myNode] {} to (210:1) node[myNode] {} to (0,0);
		\end{scope}
		
		\begin{scope}[xshift=2cm]
		\draw (0,0) circle (0.6);
		\node[myNode, label=above:$z_1$] (z1) at (90:0.6) {};
		\node[myNode, label=below:$z_2$] (z2) at (-90:0.6) {};
		\node[label=left:$C$] () at (170:0.4) {};
		
		\draw plot [smooth, tension=1.2] coordinates { (u1) (-1.4,0.6) (-0.8,0.55) (z1) (0.9,0.6) (1.1,0.4) (u2)};
		\draw plot [smooth, tension=1.2] coordinates { (u1) (-1,-0.3) (-0.6,-0.6) (z2) (1,-0.5) (u2)};
		\end{scope}
		
		\begin{scope}[on background layer]
		\fill[mustard!90] (u1.center) to [bend left = 90] (u2.center) to [bend left=90] (u1.center);
		\end{scope}
		\end{scope}
		
		\end{tikzpicture}
		\vspace{-2mm}\caption{The paths $P_{1}$ and $P_{2}$ in Case 1 of Claim 2.}\label{concepts}
	\end{figure}

	
	\smallskip
	
	\noindent {\em Case 2:} Either $P_{1}$ or $P_{2}$ is disjoint to $C$.
	
	Say, without loss of generality, that $V(P_{1})\cap V(C)=\emptyset$.
	Applying Menger's theorem for the vertex sets $\{u_1,u_2\}$ and $V(C)$ in $B$, we deduce the existence of two vertices $x,y \in V(C)$, an $(x,u_{1})$-path $Q_{1}$ in $B$,
	and a $(y,u_{2})$-path $Q_{2}$ in $B$ such that $V(Q_{1}) \cap V(Q_{2}) = \emptyset$.
	Let $z_{1},z_{2}$ be the vertices where $Q_{1},Q_{2}$ meet $P_{1}$
	for the first time, respectively (starting from $x$ and $y$,
	respectively). 
	{Then, by contracting every edge of $P_1$ except of one edge of the $(z_1,z_2)$-subpath of $P_{1}$, we form
		$\hyperref[presents]{O_{9}^{1}}$ as a minor of $G$, a contradiction (see \autoref{fortress}).}
	Claim 2 follows.	\vspace{-2mm}

	\begin{figure}[H]
		\centering
		\begin{tikzpicture}[myNode/.style = black node]
		
		\begin{scope}
		\node[myNode] (u1) at (0,0) {};
		\node[label=above:$u_1$] () at (-0.2,0) {};
		\node[myNode]  (u2) at (4,0) {};
		\node[label=above:$u_2$] () at (4.15,0) {};;
		
		\draw[-] (u1) to (150:1) node[myNode] {} to (210:1) node[myNode] {} to (u1);
		
		\begin{scope}[xscale=-1, xshift=-4cm]
		\draw[-] (0,0) to (150:1) node[myNode] {} to (210:1) node[myNode] {} to (0,0);
		\end{scope}
		
		\begin{scope}[xshift=2cm, yshift=-0.3cm]
		\draw (0,0) circle (0.5);
		\node[myNode, label=left:$x$] (x) at (150:0.5) {};
		\node[myNode, label=below:$y$] (y) at (50:0.5) {};

		\draw[-, name path=P1] plot [smooth, tension=1.2] coordinates { (u1) (-1.6,0.5) (-0.8,0.7) (0.1,0.8) (1,0.58) (u2)};
		\draw[-, name path=Q1] plot [smooth, tension=1.2] coordinates { (x) (-0.8,0.9) (-1.5,0.2) (u1)};
		\draw[-, name path=Q2] plot [smooth, tension=1.2] coordinates { (y) (1,0.9) (1.5,0) (u2)};
		\node[label=below:$Q_1$] (y) at (-1.4,0.25) {};
		\node[label=below:$Q_2$] (y) at (1.3,0.1) {};
		\node[label=left:$C$] () at (250:0.5) {};
		
		\begin{scope}
		\draw [name intersections={of= P1 and Q1}]
		(intersection-3) node[myNode, label=above:$z_{1}$] {};
		\end{scope}
		
		\begin{scope}
		\draw [name intersections={of= P1 and Q2}]
		(intersection-1) node[myNode, label=above:$z_{2}$] {};
		\end{scope}
		
		\end{scope}
		
		\begin{scope}[on background layer]
		\fill[mustard!90] (u1.center) to [bend left = 90] (u2.center) to [bend left=90] (u1.center);
		\end{scope}
		\end{scope}
		
		\end{tikzpicture}
		\vspace{-2mm}\caption{The paths $P_{1},Q_{1},Q_{2}$ in Case 2 of Claim 2.}\label{fortress}
	\end{figure}
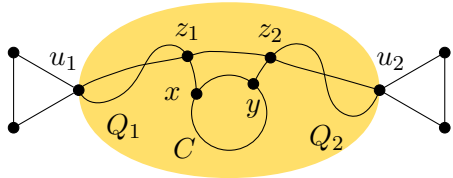

	
	We now return to the proof of the Lemma. Notice that, by \autoref{theories}, $u_{1},u_{2}$ are the only cut-vertices of $G$ contained in $B$. Therefore, Claim 1 implies that
	$D_{1},B,D_{2}$ are the only blocks of $G$.
	
	Let $H_{1}\in {\cal C}(G,u_{2})$ be the component where $u_{1}\in H_{1}$
	and $H_{2}\in {\cal C}(G,u_{1})$ be the component where $u_{2}\in H_{2}$.
	By \autoref{divorces}, we have that $\obs({\cal P})\leq H_{1}\setminus u_{2},H_{2}\setminus u_{1}$,
	or equivalently $H_{1}\setminus u_{2},H_{2}\setminus u_{1}\not\in{\cal P}$.
	Therefore, since $H_{1}\setminus u_{2},H_{2}\setminus u_{1}$ are connected, Claim 2 implies that there exist two cycles $C_{1},C_{2}$ in $B$
	such that $u_{1}\in V(C_{1})$ and $u_{2}\in V(C_{2})$. If $V(C_{1})\cap V(C_{2})=\emptyset$ then $\hyperref[consiste]{{O}_{2}^{0}}\leq G$
	and if $|V(C_{1})\cap V(C_{2})|\geq 2$ then $\hyperref[consists]{{O}_{1}^{1}}\leq G$, a contradiction in both cases.
	Hence, $V(C_{1})\cap V(C_{2})=\{x\}$ for some $x\in V(B)$. As $B$
	is a block of $G$, $x$ is not a cut-vertex of $B$ and so
	there exists a $(u_{1},u_{2})$-path $P$ in $B$ such that $x\not\in V(P)$.	\vspace{-1mm}
	
	

	\begin{figure}[H]
		\centering
		\begin{tikzpicture}[myNode/.style = black node, scale=.8]
		
		\begin{scope}
		\node[myNode] (u1) at (0,0) {};
		\node[label=above:$u_1$] () at (-0.2,0) {};
		\node[myNode]  (u2) at (4,0) {};
		\node[label=above:$u_2$] () at (4.15,0) {};;
		
		\draw[-] (u1) to (150:1) node[myNode] {} to (210:1) node[myNode] {} to (u1);
		\begin{scope}[xscale=-1, xshift=-4cm]
		\draw[-] (0,0) to (150:1) node[myNode] {} to (210:1) node[myNode] {} to (0,0);
		\end{scope}
		
		\node[myNode, label=above:$x$] (x) at (2,0) {};
		
		\draw[-] (u1.center) to [out=30, in=110] (x.center) to [out=-110, in=-30] (u1.center);
		\draw[-] (u2.center) to [out=150, in=70] (x.center) to [out=-70, in=-150] (u2.center);
		
		\draw[-] plot [smooth, tension=1.2] coordinates { (u1) (1.2,0.9) (2.5,0.6) (3.5,0.3) (u2)};
		
		\begin{scope}[on background layer]
		\fill[mustard!90] (u1.center) to [bend left = 90] (u2.center) to [bend left=90] (u1.center);
		\end{scope}
		
		\end{scope}
		
		\begin{scope}[xshift=7cm]
		\node[myNode] (u1) at (0,0) {};
		\node[label=above:$u_1$] () at (-0.2,0) {};
		\node[myNode]  (u2) at (4,0) {};
		\node[label=above:$u_2$] () at (4.15,0) {};;
		
		\draw[-] (u1) to (150:1) node[myNode] {} to (210:1) node[myNode] {} to (u1);
		\begin{scope}[xscale=-1, xshift=-4cm]
		\draw[-] (0,0) to (150:1) node[myNode] {} to (210:1) node[myNode] {} to (0,0);
		\end{scope}
		
		\node[myNode, label=above:$x$] (x) at (2,0) {};
		
		\draw[-, name path=C1] (u1.center) to [out=30, in=110] (x.center) to [out=-110, in=-30] (u1.center);
		\draw[-, name path=C2] (u2.center) to [out=150, in=70] (x.center) to [out=-70, in=-150] (u2.center);
		
		\begin{scope}[xshift=2cm]
		\begin{scope}
		\draw[opacity=0, name path=line] (-0.4,1) -- (-0.4,-1);
		\draw [name intersections={of= C1 and line}]
		(intersection-2) node[myNode, label=above:$z_{1}$] (z1) {};
		\end{scope}
		\draw[-, name path=P] plot [smooth, tension=1.2] coordinates { (u1) (-1.2,0.6) (-0.9,-0.8) (z1) (0.3,-0.3) (1.4,0.5) (u2)};
		
		\begin{scope}
		\draw [name intersections={of= P and C2}]
		(intersection-1) node[myNode, label=below:$z_{2}$] {};
		\end{scope}
		\end{scope}
		
		\begin{scope}[on background layer]
		\fill[mustard!90] (u1.center) to [bend left = 90] (u2.center) to [bend left=90] (u1.center);
		\end{scope}
		
		\end{scope}
		
		\end{tikzpicture}
		\vspace{-2mm}\caption{The cycles $C_{1},C_{2}$ and the path $P$ in the end of the proof of the Lemma.}\label{radiance}
	\end{figure}
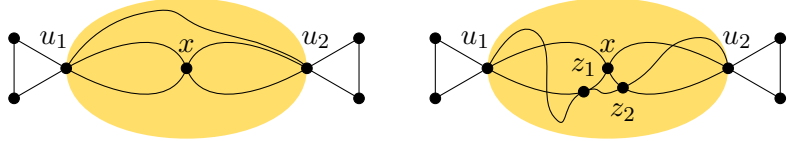
	

	If $V(P)\cap V(C_{1}\cup C_{2})=\{u_{1},u_{2}\}$ then $\hyperref[daylight]{O_{12}^{1}}\leq G$, a contradiction (see the leftmost  figure of \autoref{radiance}).
	Therefore $P$ intersects, without loss of generality, $C_{1}$ at a vertex different from $u_{1}$.
	Let $z_{1}\in V(C_{1})$ be the vertex that $P$ meets $C_{1}$ for the last time.
	Also, let $z_{2}\in V(C_{2})$ be the vertex that the $(z_{1},u_{2})$-subpath of $P$ meets $C_{2}$ for the first time.
	Then, the cycle $C_{1}$, the $(z_{1},z_{2})$-subpath of $P$, the $(z_{2},u_{2})$-path in $C_{2}$ that does not contain $x$ and the
	$(x,u_{2})$-path of $C_{2}$ that does not contain $z_{2}$, along with $D_{1},D_{2}$ form $\hyperref[consists]{{O}_{1}^{1}}$ as a minor of $G$, a contradiction (see the rightmost figure of \autoref{radiance}).
\end{proof}

\begin{lemma}\label{currency}
	Let $G\in \obs({\cal A}_{1}({\cal P}))\setminus  {\cal O}$. If $G$ contains a cut-vertex $x$, then ${\cal C}(G,x)=\{B,D\}$ for some biconnected graph $B$ and $D\cong K_{3}$.
\end{lemma}
%

\begin{proof}
	{First, suppose that $G$ contains a cut-vertex $x$ and recall that,
	by \autoref{hundreds}, $G$ is connected. Now, by
	\autoref{busyness}, we have that $x$ is the only cut-vertex of $G$
	and therefore \autoref{divorces} implies that there exists some block $B\in{\cal C}(G,x)$ such that $\obs({\cal P})\leq B\setminus x$.}
	
	Let $D = \cupall \{ H \in {\cal C}(G,x) :H\not= B \}$, that is $G\setminus (B\setminus x)$. We will prove that $D\cong K_{3}$.
	Suppose, towards a contradiction, that $D$ contains more than one cycle. Then, since $D$ is connected, we have that $\obs({\cal P})\leq D$
	and so the fact that $D\cap (B\setminus x)=\emptyset$ implies that $ {{\cal O}^{0}}\leq G$, a contradiction.
	Therefore, $D$ contains at most one cycle.
	But since $x$ is a cut-vertex we have that $D\neq \emptyset$ and hence, \autoref{generals}(1) implies that $D\cong K_{3}$ which concludes the proof of the Lemma.
\end{proof}

\begin{lemma}
	\label{absolute}
	If $G\in \obs({\cal A}_{1}({\cal P}))\setminus  {\cal O}$ then  $K_{4}\not\leq G$  or $G$ is biconnected.
\end{lemma}

\begin{proof}
	Suppose, towards a contradiction, that $K_{4}\leq G$ and $G$ is not
	biconnected. Then, since by \autoref{hundreds}, $G$ is connected,
	there exists a cut-vertex $x$ of $G$ and so
	\autoref{currency} implies that ${\cal C}(G,x)=\{B,D\}$,
	where $B$ is a biconnected graph and $D\cong K_3$.
	Now, observe that $K_{4}\leq B$ and, following \autoref{etiology},
	consider an ${r}$-wheel-subdivision pair $(H,K)$ of $B$.
	
	We argue that the following holds:
	\medskip
	
	\noindent{\em Claim 1:} $x$ is a branch vertex of $K$.
	
	\noindent{\em Proof of Claim 1:}	
	We first prove that $x\in V(K)$. Suppose, towards a contradiction,
	that $x\not\in V(K)$. Notice that, since $x\in V(B)$ and $B$
	is a biconnected, then there exist two paths $P_{1},P_{2}$ from $x$ to
	some vertex of $K$ such that $V(P_{1})\cap V(P_{2})=\{x\}$.
	
	Let $u_{1},u_{2}$ be the first time $P_{1}$ and $P_{2}$ meet $K$,
	respectively. Let $P_{1}'$ be the $(x,u_{1})$-subpath of $P_{1}$ and
	$P_{2}'$ be the $(x,u_{2})$-subpath of $P_{2}$. Then, the
	$(u_{1},u_{2})$-path $P_{1}'\cup P_{2}'$ intersects $K$ only in its
	endpoints and so \autoref{analysts} implies that there exists an edge $e\in E(H)$ such that
	$u_{1},u_{2}$ are both vertices of the subdivision of $e$ in $K$ (see \autoref{ceremony}).	\vspace{-3mm}
	
	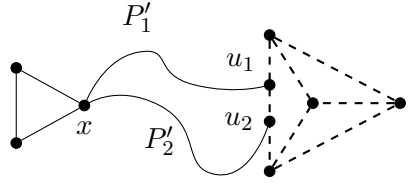
\begin{figure}[H]
		\centering
		\begin{tikzpicture}[myNode/.style = black node]
		
		\node[myNode, label=above:$z_{1}$] (n1) at (10,0) {};
		\node[myNode,  label=below:$z_{2}$] (n2) [below=10ex of n1]  {};
		\coordinate (midway1) at ($(n1)!0.5!(n2)$)  {};
		
		\node[myNode] (n3) [right=10ex of midway1] {};
		\node[myNode] (n4) [right=3ex of midway1] {};
		
		\node[myNode, label=above left:$u_1$] (n5) [above=1ex of midway1] {};
		\node[myNode, label=left:$u_2$] (n6) [below=1ex of midway1] {};
		
		\node[myNode] (n7) [above left=2.5ex and 20ex of midway1] {};
		\node[myNode] (n8) [below of=n7]  {};
		\coordinate (midway2) at ($(n7)!0.5!(n8)$)  {};
		\node[label=below:$x$,myNode] (n9) [right=5ex of midway2] {};
		
		\draw[dashed, thick] (n1) -- (n2);
		\draw[dashed, thick] (n1) -- (n3);
		\draw[dashed, thick] (n1) -- (n4);
		
		\draw[dashed, thick] (n2) -- (n3);
		\draw[dashed, thick] (n2) -- (n4);
		\draw[dashed, thick] (n3) -- (n4);
		
		\draw (n7) -- (n9);
		\draw (n7) -- (n8);
		\draw (n8) -- (n9);
		
		\node (n51) [above right of=n9] {};
		\node (n52) [left of=n5]  {};
		\node (n61) [right of=n9]  {};
		\node (n62) [below left of=n6]  {};
		
		\draw plot [smooth, tension=1.2] coordinates { (n9) (n51) (n52) (n5)};
		
		\node[label=above:$P_1'$] () at (n51) {};
		
		\draw plot [smooth, tension=1.2] coordinates { (n9) (n61) (n62) (n6)};
		
		\node[label=below:$P_2'$] () at (n61) {};
		
		\end{tikzpicture}
		\vspace{-2mm}\caption{The paths $P_{1}', P_{2}'$ from $x$ to $K$ in the proof of Claim 1.}
		\label{ceremony}
	\end{figure}

	Let $P$ be the path corresponding to the subdivision of $e$ in $K$.  
	Let,  also, $z_{1}$ be the branch vertex of $K$ incident to $e$ that is closest
	to $u_{1}$ in $P$ and $z_{2}$ be the other branch vertex of $K$ incident
	to $e$. Then, by contracting all the edges in the
	$(u_{1},z_{1})$-,$(u_{2},z_{2})$-subpaths of $P$ we form
	$\hyperref[skeleton]{{O}_{4}^{1}}$ as a minor of $G$, a contradiction.
	Hence, $x\in V(K)$.
	
	But now, if $x$ is not a branch	vertex of $K$ then
	$\hyperref[skeleton]{{O}_{4}^{1}}\leq G$, a contradiction.
	Claim 1 follows.
	
	\medskip
	We now prove that $H$ is isomorphic to $K_{4}.$ Suppose to the
	contrary, that $H$ is isomorphic to an $r$-wheel for some $r\geq 4$.
	By Claim 1, $x$ is a branch vertex of $K$. Then, $x$ is the center of $K$,
	otherwise $\hyperref[skeleton]{{O}_{4}^{1}}\leq G,$ a contradiction.
	Since $B$ is a block of $G$, by \autoref{patience}, we have that $x$ is a ${\cal P}$-apex vertex of $B$ and so the fact that ${\cal C}(G,x)=\{B,D\}$, where $D\cong K_3$, implies that $x$ is also a ${\cal P}$-apex vertex of $G$, a contradiction. Hence, $r=3$, i.e.
	$H$ is isomorphic to $K_{4}$.
	
	According to Claim 1, we have that $x$ is a branch vertex of $K$. 
	Let $y_{i}, i\in [3],$ be the three other branch vertices of $K,$ as in \autoref{fig24}:\vspace{-3mm}
	
	\begin{figure}[H]
		\centering
		\begin{tikzpicture}[every node/.style = black node, scale=.8]
		
		\draw (0:0.5) node[label=270:$x$] {} -- (120:0.5) node {} -- (240:0.5) node {} -- cycle;
		
		\begin{scope}[on background layer, xshift = 2cm]
		\draw[dashed, thick]  (180:1.5) -- (60:1.5)  (180:1.5) -- (0,0) (180:1.5) -- (300:1.5)  (0,0) -- (60:1.5) -- (300:1.5) -- (0,0) ;
		
		\draw (60:1.5) node[label=60:$y_{2}$] {};
		\draw (0,0) node[label=0:$y_{1}$] {};
		\draw (300:1.5) node[label=300:$y_{3}$] {};
		\end{scope}
		
		\end{tikzpicture}
		\vspace{-2mm}\caption{{Illustration of the branch vertices of $K$.}}\label{fig24}
	\end{figure}
	
%
%
%
	
	\noindent{\em Claim 2:} Every flap of $(H,K)$ is $x$-oriented. 
	
	\smallskip
	
	\noindent{\em Proof of Claim 2:}
We first prove that for every flap $F$ of $(H,K)$ there is a $t\in V(K)$ such that  $F$ is an $(x,t)$-flap.
For that, consider a flap $F$ whose base does not contain $x$, i.e., $F$ is an $(s,t)$-flap for some $s,t\in V(K)$, where $s,t\neq x$ (the set $\{s,t\}$ is the base of $F$).
By \autoref{campaign}(3), there exists an edge
$e\in E(H)$ such that $s,t$ belong both to the subdivision of $e$ in $K$.
Observe that if $e=y_{i} y_{j}$, then by \autoref{campaign}(1),(2),
$\hyperref[friesian]{{O}_{5}^{1}}\leq G$ (see left figure of \autoref{autonomy}),
while if $e=x y_{i}$ and $s,t$ are both different from $x$,
then, similarly, $\hyperref[consists]{{O}_{1}^{1}}\leq G$
(see right figure of \autoref{autonomy}), a contradiction in both cases.

	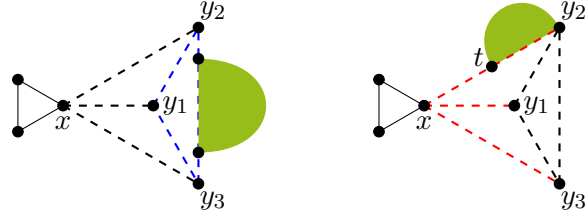
\begin{figure}[H]
		\centering
		\begin{tikzpicture}[every node/.style = black node, scale=.8]
		
		\begin{scope}
		
		\draw (0:0.5) node[label=270:$x$] {} -- (120:0.5) node {} -- (240:0.5) node {} -- cycle;
		
		\begin{scope}[on background layer, xshift=2cm]
		\draw[dashed, thick]  (180:1.5) -- (60:1.5)  (180:1.5) -- (0,0) (180:1.5) -- (300:1.5) ;
		\draw[dashed, thick, blue]  (0,0) -- (60:1.5) -- coordinate[pos=.2] (L) coordinate[pos=.8] (R) (300:1.5) -- cycle ;
		\fill[applegreen!90] (L.center) to [bend left =90, min distance = 1.5cm] (R.center) to (L.center) {}; 
		
		\draw (60:1.5) node[label=60:$y_{2}$] {};
		\draw (0,0) node[label=0:$y_{1}$] {};
		\draw (300:1.5) node[label=300:$y_{3}$] {};
		\node () at (L) {};
		\node () at (R) {};
		\end{scope}
		
		\end{scope}

		\begin{scope}[xshift = 6cm]
				
		\draw (0:0.5) node[label=270:$x$] {} -- (120:0.5) node {} -- (240:0.5) node {} -- cycle;
		
		\begin{scope}[on background layer, xshift=2cm]
	
		\draw[dashed, thick, red]  (180:1.5) -- coordinate[pos=.4] (L) coordinate[pos=.8] (R)  (60:1.5)  (180:1.5) -- (0,0) (180:1.5) -- (300:1.5) ;
		\draw[dashed, thick]  (0,0) -- (60:1.5) -- (300:1.5) -- cycle ;
		\fill[applegreen!90] (L.center) to [bend left =90, min distance = 1cm] (R.center);

		\node () at (R) {};
		\node[label=160:$t$] at (L) {};
		\draw (60:1.5) node[label=60:$y_{2}$] {};
		\draw (0,0) node[label=0:$y_{1}$] {};
		\draw (300:1.5) node[label=300:$y_{3}$] {};
		\end{scope}
		
		\end{scope}
		\end{tikzpicture}
		\vspace{-2mm}\caption{Left figure: An $(s,t)$-flap in the subdivision of some $y_{i} y_{j}$ edge of $H$ (depicted in blue) in the proof of Claim 2. Right figure: An $(s,t)$-flap in the subdivision of some $x y_{i}$ edge of $H$ (depicted in red) in the proof of Claim 2.}\label{autonomy}
		\end{figure}
	
	Therefore, $F$ is an $(x,t)$-flap for some $t\in V(K)$.
	Observe now that for every $t\in V(K)$, every $(x,t)$-flap of $(H,K)$ is $x$-oriented.
	Indeed, for otherwise, if there exists an $(x,t)$-flap $F$, for some
	$t\in V(K)$, that is not $x$-oriented, then, by \autoref{mobility}, $F$ is $t$-oriented.
	Hence, that there exists a cycle in $F$ containing
	$t$ but not $x$ and so $\hyperref[consists]{{O}_{1}^{1}}\leq G$,
	a contradiction. Claim 2 follows.
	
	\medskip
	
	Therefore, since ${\cal C}(G,x) = \{B,K_{3}\}$ and the fact that,
	by Claim 2, every flap of $(H,K)$ is $x$-oriented, 
	\autoref{fanfares} implies that $x$ is a {${\cal P}$-apex} vertex of $G$ and so $G\in {\cal A}_{1}({\cal P}),$ a contradiction.
\end{proof}

\begin{lemma}\label{elicited}
	If $G\in \obs({\cal A}_{1}({\cal P}))\setminus  {\cal O}$ then  $K_{2,3}\not\leq G$  or $G$ is biconnected.
\end{lemma}

\begin{proof}
	Suppose, towards a contradiction, that $K_{2,3}\leq G$ and $G$ is not
	biconnected. Then, since by \autoref{hundreds}, $G$ is connected,
	there exists a cut-vertex $u$ of $G$ and so
	\autoref{currency} implies that ${\cal C}(G,u)=\{B,T\}$,
	where $B$ is a biconnected graph and $T\cong K_{3}$.
	Note that since $G$ contains a cut-vertex, then by \autoref{absolute} it is
	$K_{4}$-free. Note also that by \autoref{negative} we have that there exists a
	b-rich separator $S=\{x,y\}$ of $B$ and observe that $S$ is also a b-rich separator of $G$.
	Since $G\in \obs({\cal A}_{1}({\cal P}))\setminus
	{\cal O}$, by \autoref{princess}, every b-rich separator of $G$ is nice
	and so \autoref{symphony} implies that $S$ is the unique b-rich separator of $G$, and therefore it is also the unique b-rich separator of $B$.
	We distinguish two cases, based on whether $u$ belongs to $S$ or not:
	
	\smallskip
	
	\noindent {\em Case 1:} The cut-vertex $u$ is neither of $x,y$.
	
	Let $H\in {\cal C}(G, S)$ be the component where $u\in V(H)$ and
	let $\hat{H} = H \setminus (T \setminus u)$.
	Observe that $B=G[V(G\setminus H)\cup V(\hat{H})]$ and that $\hat{H}\in {\cal C}(B,S)$.
	Since $B$ is a block, every $H\in {\cal C}(B,S)$ is biconnected.
	Moreover, since $S$ is the unique b-rich separator of $B$, by
	\autoref{outerplanarfl},  every $H\in {\cal C}(B,S)$  is outerplanar.
	Thus, by \autoref{decoding},
	for every $H\in {\cal C}(B,S)$,
	we can consider the Hamiltonian cycle $C_{H}$ of $H$.
	{Notice that $K_{4}$-freeness of $G$ implies that $xy\in E(C_{H})$}.
	Also, keep in mind that by \autoref{generals}, for every $H'\in {\cal C}(G, S)
    \setminus\{H\}$, $G[V(H')]$ contains a cycle.
	
	\medskip
	
	\noindent{\em Observation 1:} $B$ does not contain $(x,y)$-disjoint chords.
	Indeed, if $H$ contains an $(x,y)$-disjoint chord,
	{then $\obs({\cal P})\leq H\setminus S$ (see leftmost and cental figure of \autoref{brandeis})}
	and so
	$\{\hyperref[magnetic]{{O}_{1}^{0}}, \hyperref[narcotic]{{O}_{3}^{0}}\}\leq G$,
	while if some $H'\in{\cal C}(B,S)\setminus\{H\}$ contains a $(x,y)$-disjoint chord,
	then $\hyperref[appetite]{{O}_{2}^{1}}\leq G$ (see rightmost figure of \autoref{brandeis}), a contradiction.	\vspace{-3mm}

	\begin{figure}[H]
		\centering
		\vspace{-1cm}
		\begin{subfigure}{0.3\textwidth}
			\begin{tikzpicture}[x  = {(0:1cm)}, y  = {(90:1cm)},z= {(30:0.8cm)}, myNode/.style = black node, scale=0.6]
			
			\begin{scope}
			
			\node[label=above:$x$] (A) at (0,2,0) {};
			\node[label=below:$y$] (B) at (0,0,0) {};
			
			\begin{scope}[canvas is yz plane at x=0]
			
			\filldraw[fill=red!30!yellow,  opacity=0.5,dashed, thick] (A.center) to [bend right=60, min distance=0.5cm] (B.center) to [bend left =110, min distance =4cm] (A.center);
			
			\end{scope}

			\begin{scope}[canvas is xy plane at z=0]
			
			\begin{scope}
			\filldraw[fill=blue!40,  opacity=0.5,dashed, thick]
			(A.center) to [bend right=40, min distance=0.7cm] (B.center) to [bend left =110, min distance =5cm] coordinate[pos=0.2] (e1) coordinate[pos=0.5] (u) coordinate[pos=0.8] (e2) (A.center);

			\draw[-]  (e1) node[myNode] () {} to (e2) node[myNode] () {};
			\node[myNode] (uu) at ($(u)$) {};
			\coordinate[left=0.5cm of uu] (v) {};
			\node[myNode] (u1) [above =0.25cm of v] {};
			\node[myNode] (u2) [below =0.25cm of v] {};
			
			\draw[-] (uu) to (u1) to (u2) to (uu);
			
			\end{scope}
			
			\begin{scope}
			\filldraw[fill=green!70!blue,  opacity=0.7, dashed, thick]
			(A.center) to [bend left=50, min distance =1cm] (B.center) to [bend right =110, min distance =5cm] (A.center);

			\end{scope}
			
			\end{scope}
			
			\node[myNode] () at (A) {};
			\node[myNode] () at (B) {};
			
			\end{scope}
			\end{tikzpicture}
		\end{subfigure}
		~
		\begin{subfigure}{0.3\textwidth}
			\begin{tikzpicture}[x  = {(0:1cm)}, y  = {(90:1cm)},z= {(30:0.8cm)}, myNode/.style = black node, scale=0.6]
			
			\begin{scope}
			
			\node[label=above:$x$] (A) at (0,2,0) {};
			\node[label=below:$y$] (B) at (0,0,0) {};
			
			\begin{scope}[canvas is yz plane at x=0]
			
			\filldraw[fill=red!30!yellow,  opacity=0.5,dashed, thick] (A.center) to [bend right=60, min distance=0.5cm] (B.center) to [bend left =110, min distance =4cm] (A.center);
			
			\end{scope}

			\begin{scope}[canvas is xy plane at z=0]
			
			\begin{scope}
			\filldraw[fill=blue!40,  opacity=0.5,dashed, thick]
			(A.center) to [bend right=40, min distance=0.7cm] (B.center) to [bend left =110, min distance =5cm] coordinate[pos=0.5] (u) coordinate[pos=0.6] (e1) coordinate[pos=0.9] (e2) (A.center);

			\draw[-]  (e1) node[myNode] () {} to (e2) node[myNode] () {};
			\node[myNode] (uu) at ($(u)$) {};
			\coordinate[left=0.5cm of uu] (v) {};
			\node[myNode] (u1) [above =0.25cm of v] {};
			\node[myNode] (u2) [below =0.25cm of v] {};
			
			\draw[-] (uu) to (u1) to (u2) to (uu);
			
			\end{scope}
			
			\begin{scope}
			\filldraw[fill=green!70!blue,  opacity=0.7, dashed, thick]
			(A.center) to [bend left=50, min distance =1cm] (B.center) to [bend right =110, min distance =5cm] (A.center);

			\end{scope}
			
			\end{scope}
			
			\node[myNode] () at (A) {};
			\node[myNode] () at (B) {};
			
			\end{scope}
		\end{tikzpicture}
		\end{subfigure}
		~
		\begin{subfigure}{0.3\textwidth}
			\begin{tikzpicture}[x  = {(0:1cm)}, y  = {(90:1cm)},z= {(30:0.8cm)}, myNode/.style = black node, scale=0.6]
				
			\begin{scope}
			
			\node[label=above:$x$] (A) at (0,2,0) {};
			\node[label=below:$y$] (B) at (0,0,0) {};
			
			\begin{scope}[canvas is yz plane at x=0]
			
			\filldraw[fill=red!30!yellow,  opacity=0.5,dashed, thick] (A.center) to [bend right=60, min distance=0.5cm] (B.center) to [bend left =110, min distance =4cm] (A.center);
			
			\end{scope}

			\begin{scope}[canvas is xy plane at z=0]
			
			\begin{scope}
			\filldraw[fill=blue!40,  opacity=0.5,dashed, thick]
			(A.center) to [bend right=40, min distance=0.7cm] (B.center) to [bend left =110, min distance =5cm] coordinate[pos=0.5] (u) (A.center);
			
			\node[myNode] (uu) at ($(u)$) {};
			\coordinate[left=0.5cm of uu] (v) {};
			\node[myNode] (u1) [above =0.25cm of v] {};
			\node[myNode] (u2) [below =0.25cm of v] {};
			
			\draw[-] (uu) to (u1) to (u2) to (uu);
			
			\end{scope}
			
			\begin{scope}
			\filldraw[fill=green!70!blue,  opacity=0.7, dashed, thick]
			(A.center) to [bend left=50, min distance =1cm] (B.center) to [bend right =110, min distance =5cm] coordinate[pos=0.2] (e1) coordinate[pos=0.5] (u) coordinate[pos=0.8] (e2) (A.center);

			\end{scope}
			
			\end{scope}
			
			\draw[-]  (e1) node[myNode] () {} to (e2) node[myNode] () {};

			\node[myNode] () at (A) {};
			\node[myNode] () at (B) {};
			
			\end{scope}
			
			\end{tikzpicture}
		\end{subfigure}
		\vspace{-0.5cm}
		\vspace{-2mm}\caption{The $(x,y)$-disjoint chord in the proof of Observation 1}\label{brandeis}
	\end{figure}
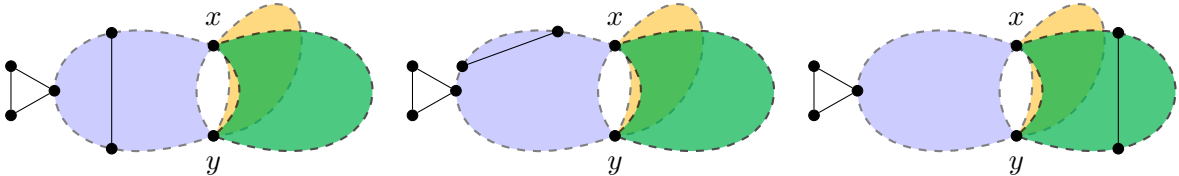
	
	\noindent{\em Claim 1:} $B\setminus (\hat{H}\setminus S)$ does not contain both $x$-chords and $y$-chords.\smallskip
	
	\noindent {\em Proof of Claim 1:}
	Suppose, to the contrary, that there exist an $x$-chord and a $y$-chord in $B\setminus (\hat{H}\setminus S)$, say $e_{x}$ and $e_{y}$, respectively.
	
	{We first consider the case that there exists some $H'\in{\cal C}(B,S)\setminus \{\hat{H}\}$ such that $e_{x}, e_{y}\in E(H')$. Due to $K_{4}$-freeness of $G$ we have that $e_{x}, e_{y}$ are not ``crossing'' i.e., if $e_{x}=xx'$ and $e_{y}=yy'$, then the $(x,y')$-path in $C_{H'}\setminus y$ contains $x'$.
	This implies that $\hyperref[adorning]{{O}_{3}^{1}}\leq G$, a contradiction (see the leftmost figure in \autoref{unmasked}).}
%
%
	{In the case where $e_{x}$ and $e_{y}$ are in different augmented connected components in ${\cal C}(G,S) \setminus \{\hat{H}\}$, we have that $\hyperref[presents]{{O}_{10}^{1}}\leq G,$
	a contradiction (see the rightmost figure in \autoref{unmasked}). Claim 1 follows.}
	\vspace{-3mm}
	
	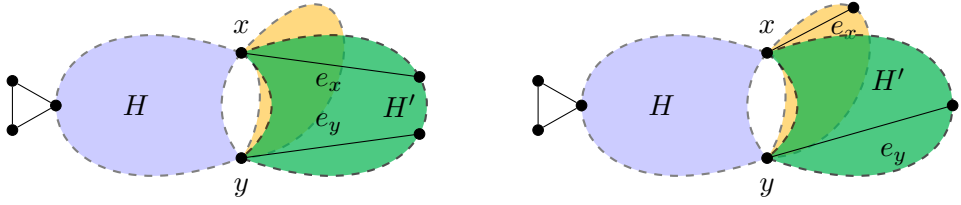
\begin{figure}[H]
		\centering
		\vspace{-1cm}
		\begin{subfigure}{0.4\textwidth}
			\begin{tikzpicture}[x  = {(0:1cm)}, y  = {(90:1cm)},z= {(30:0.8cm)}, myNode/.style = black node, scale=.7]
			
			\begin{scope}
			
			\node[label=above:$x$] (A) at (0,2,0) {};
			\node[label=below:$y$] (B) at (0,0,0) {};
			
			\begin{scope}[canvas is yz plane at x=0]
			
			\filldraw[fill=red!30!yellow,  opacity=0.5,dashed, thick] (A.center) to [bend right=60, min distance=0.5cm] (B.center) to [bend left =110, min distance =4cm] (A.center);
			
			\end{scope}

			\begin{scope}[canvas is xy plane at z=0]
			
			\begin{scope}
			\filldraw[fill=blue!40,  opacity=0.5,dashed, thick]
			(A.center) to [bend right=40, min distance=0.7cm] (B.center) to [bend left =110, min distance =5cm] coordinate[pos=0.5] (u) (A.center);

			\node[myNode] (uu) at ($(u)$) {};
			\coordinate[left=0.5cm of uu] (v) {};
			\node[myNode] (u1) [above =0.25cm of v] {};
			\node[myNode] (u2) [below =0.25cm of v] {};
			
			\node () at (-2,1) {$\hat{H}$};
			
			\draw[-] (uu) to (u1) to (u2) to (uu);
			
			\end{scope}
			
			\begin{scope}
			\filldraw[fill=green!70!blue,  opacity=0.7, dashed, thick]
			(A.center) to [bend left=50, min distance =1cm] (B.center) to [bend right =110, min distance =5cm] coordinate[pos=0.4] (b) coordinate[pos=0.6] (a) (A.center); 
			
			\draw[-] (A.center) to coordinate[pos=0.3] (e1) (a) node[myNode] () {};
			\draw[-] (B.center) to coordinate[pos=0.3] (e2) (b) node[myNode] () {};
			
			\node[label=-20:$e_{x}$] () at (e1) {};
			\node[label=30:$e_{y}$] () at (e2) {};
			
			\node () at (3,1) {$H'$};
			
			\end{scope}
			
			\end{scope}
			
			\node[myNode] () at (A) {};
			\node[myNode] () at (B) {};
			
			\end{scope}
			\end{tikzpicture}
			\end{subfigure}
			~
			\begin{subfigure}{0.4\textwidth}
			\begin{tikzpicture}[x  = {(0:1cm)}, y  = {(90:1cm)},z= {(30:0.8cm)}, myNode/.style = black node, scale=.7]
			
			\begin{scope}
			
			\node[label=above:$x$] (A) at (0,2,0) {};
			\node[label=below:$y$] (B) at (0,0,0) {};
			
			\begin{scope}[canvas is yz plane at x=0]
			
			\filldraw[fill=red!30!yellow,  opacity=0.5,dashed, thick] (A.center) to [bend right=60, min distance=0.5cm] (B.center) to [bend left =110, min distance =4cm] coordinate[pos=0.7] (aa) (A.center);
			
			\draw[-] (A.center) to coordinate[pos=0.5] (c1) (aa) node[myNode] () {};
			\node[label=right:$e_{x}$] () at (c1) {};
			
			\end{scope}

			\begin{scope}[canvas is xy plane at z=0]
			
			\begin{scope}
			\filldraw[fill=blue!40,  opacity=0.5,dashed, thick]
			(A.center) to [bend right=40, min distance=0.7cm] (B.center) to [bend left =110, min distance =5cm] coordinate[pos=0.5] (u) (A.center);

			\node[myNode] (uu) at ($(u)$) {};
			\coordinate[left=.5cm of uu] (v) {};
			\node[myNode] (u1) [above =0.25cm of v] {};
			\node[myNode] (u2) [below =0.25cm of v] {};
			
			\node () at (-2,1) {$\hat{H}$};
			
			\draw[-] (uu) to (u1) to (u2) to (uu);
			
			\end{scope}
			
			\begin{scope}
			\filldraw[fill=green!70!blue,  opacity=0.7, dashed, thick]
			(A.center) to [bend left=50, min distance =1cm] (B.center) to [bend right =110, min distance =5cm] coordinate[pos=0.5] (cc) (A.center);

			\draw[-] (B.center) to coordinate[pos=0.5] (c2) (cc) node[myNode] () {};

			\node[label=-10:$e_{y}$] () at (c2.center) {};
			
			\node () at (2.3,1.5) {$H'$};
			
			\end{scope}
			
			\end{scope}
			
			\node[myNode] () at (A) {};
			\node[myNode] () at (B) {};
			
			\end{scope}
			
			\end{tikzpicture}
			\end{subfigure}
			\vspace{-.5cm}
		\vspace{-2mm}\caption{The chords $e_{x}, e_{y}$ in the proof of Claim 1}\label{unmasked}
	\end{figure}
%
	
	
	We also observe the following:
	
	\smallskip
	
	{\noindent {\em Observation 2:} If there is an $x$-chord (resp. $y$-chord) of $C_{H'}$, for some $H'\in {\cal C}(B,S)\setminus \{\hat{H}\}$, then every chord of $C_{\hat{H}}$ is a $y$-chord (resp. $x$-chord). 
	Indeed, suppose that there exists an $x$-chord of $C_{H'}$, for some $H'\in {\cal C}(B,S)\setminus \{\hat{H}\}$, and notice that if there exist both $x$-chords and a $y$-chords in $\hat{H}$, then $\hyperref[granting]{{O}_{6}^{1}}\leq G,$  a contradiction (see \autoref{recovers}) while, if all chords of $C_{\hat{H}}$ are $x$-chords, then due to Observation 1 and Claim 1, for every $H'\in {\cal C}(B,S)\setminus \{\hat{H}\}$, every chord of $C_{H'}$ is an $x$-chord, and thus $x$ is a ${\cal P}$-apex of $G$, a contradiction.	\vspace{-4mm}}
	
%
	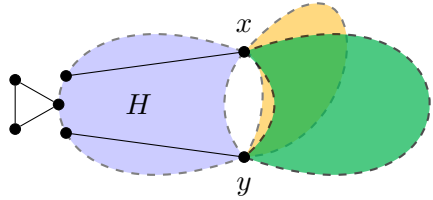
\begin{figure}[H]
		\centering
		\vspace{-1cm}
			\begin{tikzpicture}[x  = {(0:1cm)}, y  = {(90:1cm)},z= {(30:0.8cm)}, myNode/.style = black node, scale=.7]
			
			\begin{scope}
			
			\node[label=above:$x$] (A) at (0,2,0) {};
			\node[label=below:$y$] (B) at (0,0,0) {};
			
			\begin{scope}[canvas is yz plane at x=0]
			
			\filldraw[fill=red!30!yellow,  opacity=0.5,dashed, thick] (A.center) to [bend right=60, min distance=0.5cm] (B.center) to [bend left =110, min distance =4cm] (A.center);

			\end{scope}

			\begin{scope}[canvas is xy plane at z=0]
			
			\begin{scope}
			\filldraw[fill=blue!40,  opacity=0.5,dashed, thick]
			(A.center) to [bend right=40, min distance=0.7cm] (B.center) to [bend left =110, min distance =5cm] coordinate[pos=0.4] (cc) coordinate[pos=0.5]  (u) coordinate[pos=0.6] (bb) (A.center);

			\node[myNode] (uu) at ($(u)$) {};
			\coordinate[left=0.5cm of uu] (v) {};
			\node[myNode] (u1) [above =0.25cm of v] {};
			\node[myNode] (u2) [below =0.25cm of v] {};
			
			\node () at (-2,1) {$\hat{H}$};
			
			\draw[-] (uu) to (u1) to (u2) to (uu);
			
			\end{scope}
			
			\begin{scope}
			\filldraw[fill=green!70!blue,  opacity=0.7, dashed, thick]
			(A.center) to [bend left=50, min distance =1cm] (B.center) to [bend right =110, min distance =5cm] (A.center);

			\end{scope}
			
			\end{scope}
			
			\node[myNode] () at (A) {};
			\node[myNode] () at (B) {};
			
			\draw[-] (B.center) to (cc) node[myNode] () {};
			\draw[-] (A.center) to (bb) node[myNode] () {};
			\end{scope}
			\end{tikzpicture}
		\vspace{-.5cm}
		\vspace{-2mm}\caption{An $x$-chord and a $y$-chord of $C$ in Observation 2 of Case 1.}\label{recovers}
	\end{figure}
	
	\noindent {\em Claim 2:} 
	{Either $C_{H'}$ is chordless for every $H'\in {\cal C}(B,S)\setminus \{\hat{H}\}$, or $C_{\hat{H}}$ is chordless.}
	
	\smallskip
	
	\noindent {\em Proof of Claim 2:} 
	Suppose to the contrary that 
%
%
	{there exist an $H'\in {\cal C}(B,S)\setminus \{\hat{H}\}$, a chord $e$ of $C_{H'}$,
		and a chord $e'$ of $C_{\hat{H}}$.
%
If $e$ is an $x$-chord (resp. $y$-chord), then by Observation 2, $e'$ is a $y$-chord (resp. $x$-chord). 
Thus, $\hyperref[daylight]{{O}_{11}^{1}}\leq G,$  a contradiction (see \autoref{ldsjgalja}). Claim 2 follows.}

	\begin{figure}[H]
	\centering
	\vspace{-1cm}
	\begin{tikzpicture}[x  = {(0:1cm)}, y  = {(90:1cm)},z= {(30:0.8cm)}, myNode/.style = black node, scale=.7]
	
	\begin{scope}
	
	\node[label=above:$x$] (A) at (0,2,0) {};
	\node[label=below:$y$] (B) at (0,0,0) {};
	
	\begin{scope}[canvas is yz plane at x=0]
	
	\filldraw[fill=red!30!yellow,  opacity=0.5,dashed, thick] (A.center) to [bend right=60, min distance=0.5cm] (B.center) to [bend left =110, min distance =4cm] (A.center);

	\end{scope}

	\begin{scope}[canvas is xy plane at z=0]
	
	\begin{scope}
	\filldraw[fill=blue!40,  opacity=0.5,dashed, thick]
	(A.center) to [bend right=40, min distance=0.7cm] (B.center) to [bend left =110, min distance =5cm] coordinate[pos=0.4] (cc) coordinate[pos=0.5]  (u) coordinate[pos=0.6] (bb) (A.center);

	\node[myNode] (uu) at ($(u)$) {};
	\coordinate[left=0.5cm of uu] (v) {};
	\node[myNode] (u1) [above =0.25cm of v] {};
	\node[myNode] (u2) [below =0.25cm of v] {};
	
	\node () at (-2,0.8) {$\hat{H}$};
	
	\draw[-] (uu) to (u1) to (u2) to (uu);
	
	\end{scope}
	
	\begin{scope}
	\filldraw[fill=green!70!blue,  opacity=0.7, dashed, thick]
	(A.center) to [bend left=50, min distance =1cm] (B.center) to [bend right =110, min distance =5cm] coordinate[pos=0.5] (cc) (A.center); 
	
	\draw[-] (B.center) to coordinate[pos=0.5] (c2) (cc) node[myNode] () {};

	\node[label=160:$e'$] () at (c2.center) {};
	
	\node () at (2.3,1.5) {$H'$};
	
	\end{scope}
	
	\end{scope}
	
	\node[myNode] () at (A) {};
	\node[myNode] () at (B) {};
	
	\draw[-] (A.center) to  coordinate[pos=0.8] (e) (bb) node[myNode] () {};
	\node[label=10:$e$] () at (e) {};
	\end{scope}
	\end{tikzpicture}
	\vspace{-.5cm}
	\vspace{-2mm}\caption{The chords $e$, $e'$ in the proof of Claim 2.}\label{ldsjgalja}
\end{figure}
	
	\medskip
	
	According to Claim 2, either every $C_{H'}, \ H'\in {\cal C}(B,S)\setminus \{\hat{H}\}$ is chordless or $C_{\hat{H}}$ is chordless which, together with Observation 1 and Claim 1, implies that either $y$ or $x$, respectively, is a ${\cal P}$-apex of $G$, a contradiction.
	
	\medskip\medskip
	
	\noindent {\em Case 2:} The cut-vertex $u$ is either $x$ or $y$.
	
	Assume, without loss of generality, that $u=y$. We first prove the following:
	
	\medskip
	
	\noindent {\em Claim 3:} There exists a unique augmented connected component in ${\cal C}(B,S)$ that is not isomorphic to a cycle.
	
	\smallskip
	
	\noindent {\em Proof of Claim 3:}
	First, notice that if each augmented connected component in ${\cal C}(B,S)$ is
	isomorphic to a cycle, then $G\in {\cal A}_{1}({\cal P})$, a contradiction.
	Therefore, there exists an augmented connected component in ${\cal C}(B,S)$ that is not isomorphic to a cycle.
	Suppose, towards a contradiction, that ${\cal C}(B,S)$ contains two augmented connected components not isomorphic to a cycle.
	We distinguish the following subcases:
	
	\smallskip
	
	\noindent {\em Subcase 2.1:} $B$ contains an $(x,y)$-disjoint chord.
	
	Let $e$ be an $(x,y)$-disjoint chord in $B$. Then, \autoref{caillois} implies
	that $e$ is the unique $(x,y)$-disjoint chord of $B$ and $B$ does not contain both
	$x$-chords and $y$-chords. Let $H\in {\cal C}(B,S)$ be the component where $e\in E(H)$.
	
	Recall that there exists some $H'\in {\cal C}(B,S)\setminus \{H\}$ that is not
	isomorphic to a cycle and therefore $C_{H'}$ contains a chord $e'$.
	If $e'$ is an $x$-chord, then $\hyperref[presents]{{O}_{9}^{1}}\leq G$
	(see \autoref{ordinary}), {a contradiction. If $e'$ is a $y$-chord,
	then \autoref{caillois} implies that} there does not exist an
	$x$-chord and therefore $y$ is a ${\cal P}$-apex of $G$, a contradiction.	\vspace{-3mm}	\vspace{-2mm}
	
	\begin{figure}[H]
		\centering
		\vspace{-1cm}
			\begin{tikzpicture}[x  = {(0:1cm)}, y  = {(90:1cm)},z= {(30:0.8cm)}, myNode/.style = black node, scale=.7]
			
			\begin{scope}
			
			\node[label=above:$x$] (A) at (0,2,0) {};
			\node[label=below right:$y$] (B) at (0,0,0) {};
			
			\begin{scope}[canvas is yz plane at x=0]
			
			\filldraw[fill=red!30!yellow,  opacity=0.5,dashed, thick] (A.center) to [bend right=60, min distance=0.5cm] (B.center) to [bend left =110, min distance =4cm] (A.center);
			
			\end{scope}

			\begin{scope}[canvas is xy plane at z=0]
			
			\begin{scope}
			\filldraw[fill=blue!40,  opacity=0.5,dashed, thick]
			(A.center) to [bend right=40, min distance=0.7cm] (B.center) to [bend left =110, min distance =5cm] coordinate[pos=0.2] (e1)  coordinate[pos=0.8] (e2) (A.center);

			\draw[-]  (e1) node[myNode] () {} to coordinate[pos=0.5] (e) (e2) node[myNode] () {};
			
			\node[label=left:$e$] () at (e) {};

			\end{scope}
			
			\begin{scope}
			\filldraw[fill=green!70!blue,  opacity=0.7, dashed, thick]
			(A.center) to [bend left=50, min distance =1cm] (B.center) to [bend right =110, min distance =5cm] coordinate[pos=0.5] (e3) (A.center); 
			
			\draw[-] (A.center) to coordinate[pos=0.5] (ee) (e3) node[myNode] () {};
			\node[label=below:$e'$] () at (ee) {};
			\end{scope}
			
			\end{scope}
			
			\node[myNode] () at (A) {};
			\node[myNode] () at (B) {};

			\node[myNode] (u1) [below left= 0.3cm and 0.2 cm of B]  {};
			\node[myNode] (u2) [below right= 0.05cm and 0.2cm of u1]  {};
			\draw[-] (B.center) -- (u1) -- (u2)-- (B.center);

			\end{scope}

			\end{tikzpicture}
		\vspace{-.5cm}
		\vspace{-1mm}\caption{The chords $e,e'$ in the first part the proof of Subcase 2.1.}\label{ordinary}
	\end{figure}
	
	\noindent {\em Subcase 2.2:} $B$ does not contain an $(x,y)$-disjoint chord.
	
	Then \autoref{caillois} implies that $H$ contains at most one $x$-chord or at most one $y$-chord.
	If there exists at most one $x$-chord, then $y$ is a ${\cal P}$-apex of $G$, a contradiction.
	Therefore there exists at most one $y$-chord.
	
	Suppose that there exists a $y$-chord, namely $e_{y}$, and let $H\in {\cal C}(B,S)$ be the component where $e_{y}\in E(H)$.
	Recall, again, that there exists some $H'\in {\cal C}(B,S)\setminus \{H\}$ that is
	not isomorphic to a cycle and therefore $C_{H'}$ contains a chord $e'$.
	Since $e_{y}$ is the only $y$-chord in $B$, $e'$ is an $x$-chord.
	Observe that, since $y$ is not a ${\cal P}$-apex vertex of $G$,
	there exists an $x$-chord $e''$ such that $e''\not=e'$.
	If $e''\notin E(H)$, then $\{\hyperref[consiste]{{O}_{2}^{0}}, \hyperref[narcotic]{{O}_{3}^{0}}\}\leq G,$ while if $e''\in E(H)$, then $\hyperref[daylight]{{O}_{12}^{1}}\leq G$ (see \autoref{classify}),  a contradiction in both cases. Claim 3 follows.
		\vspace{-3mm}
	
	\begin{figure}[H]
		\centering
		\vspace{-1cm}
			\begin{tikzpicture}[x  = {(0:1cm)}, y  = {(90:1cm)},z= {(30:0.8cm)}, myNode/.style = black node, scale=.7]
			\begin{scope}
			
			\node[label=above:$x$] (A) at (0,2,0) {};
			\node[label=below right:$y$] (B) at (0,0,0) {};
			
			\begin{scope}[canvas is yz plane at x=0]
			
			\filldraw[fill=red!30!yellow,  opacity=0.5,dashed, thick] (A.center) to [bend right=60, min distance=0.5cm] (B.center) to [bend left =110, min distance =4cm] (A.center);
			
			\end{scope}

			\begin{scope}[canvas is xy plane at z=0]
			
			\begin{scope}
			\filldraw[fill=blue!40,  opacity=0.5,dashed, thick]
			(A.center) to [bend right=40, min distance=0.7cm] (B.center) to [bend left =110, min distance =5cm] coordinate[pos=0.4] (e1)  coordinate[pos=0.6] (e2) (A.center);

			\draw[-]  (e1) node[myNode] () {} to coordinate[pos=0.8] (ey) (B.center);
			\draw[-]  (e2) node[myNode] () {} to coordinate[pos=0.8] (e) (A.center);
			
			\node[label=160:$e_{y}$] () at (ey) {};
			\node[label=-160:$e''$] () at (e) {};

			\end{scope}
			
			\begin{scope}
			\filldraw[fill=green!70!blue,  opacity=0.7, dashed, thick]
			(A.center) to [bend left=50, min distance =1cm] (B.center) to [bend right =110, min distance =5cm] coordinate[pos=0.5] (e3) (A.center); 
			
			\draw[-] (A) to coordinate[pos=0.5] (ee) (e3) node[myNode] () {};
			\node[label=below:$e'$] () at (ee) {};
			\end{scope}
			
			\end{scope}
			
			\node[myNode] () at (A) {};
			\node[myNode] () at (B) {};

			\node[myNode] (u1) [below left= 0.3cm and 0.2 cm of B]  {};
			\node[myNode] (u2) [below right= 0.05cm and 0.2cm of u1]  {};
			\draw[-] (B.center) -- (u1) -- (u2)-- (B.center);

			\end{scope}
			\end{tikzpicture}
		\vspace{-.5cm}
		\vspace{-1mm}\caption{The chords $e_{y},e', e''$ in the last part of the proof of Subcase 2.2.}\label{classify}
	\end{figure}

	We now proceed in order to conclude Case 2 of the Lemma. According to Claim 3, let $H$ be the unique augmented connected component of $B$ that is not isomorphic to a cycle. Therefore, due to \autoref{generals}, every $H'\in {\cal C}(B,S)\setminus \{H\}$ is isomorphic to $K_{3}$.

	We have that $H$ is outerplanar and due to \autoref{currency}, $H$ is also biconnected. Let $C$ be the Hamiltonian cycle of $H$, which exists by \autoref{decoding}.
	Notice again that, $K_{4}$-freeness of $G$ implies we have that $xy\in E(C).$
    Therefore, the graph $G$ is as in \autoref{construe} below:	\vspace{-3mm}
	\begin{figure}[H]
		\centering
		\vspace{-.5cm}
		\begin{tikzpicture}[myNode/.style = black node, scale=.7]
		
		\node[myNode, label=above:$x$] (A) at (0,2) {};
		\node[myNode, label=below:$y$] (B) at (0,0) {};
		
		\node[myNode] (v1) [below left= 0.1cm and 0.5cm of A] {};
		\node[myNode] (v2) [below left= 0.9cm and 0.5cm of A] {};
		\foreach \from/\to in {v1/A,v1/B,v2/A,v2/B, A/B}
		\draw (\from) -- (\to);
		
		\coordinate (Middle) at ($(A)!0.5!(B)$);
		\node (H) [right =0.3cm of Middle] {$H$};
		
		\node[myNode] (u1) [below left= 0.cm and 0.6 cm of B]  {};
		\node[myNode] (u2) [below right= 0.4cm and 0.25cm of u1]  {};
		\draw[-] (B.center) -- (u1) -- (u2)-- (B.center);
		
		\begin{scope}[on background layer]
		\filldraw[fill=green!70!blue,  opacity=0.7, dashed, thick]
		(B.center) to [bend right =110, min distance =5cm] coordinate[pos=0.2] (E11) coordinate[pos=0.3] (E21) coordinate[pos=0.7] (E22) coordinate[pos=0.8] (E12) (A.center);
		\end{scope}
		\end{tikzpicture}
		\vspace{-.5cm}
		\vspace{-2mm}\caption{{Illustration of the structure of the graph $G$ in Case 2.}}\label{construe}
	\end{figure}

	Observe that every $(x,y)$-disjoint chord, $x$-chord, and $y$-chord of $B$ is a chord of $C$.
	\smallskip
	
	\noindent{\em Claim 4:} There do not exist $(x,y)$-disjoint chords in $B$.
	
	\smallskip
	
	\noindent{\em Proof of Claim 4:}
	Suppose, towards a contradiction, that there exists an $(x,y)$-disjoint chord
	in $B$, namely $e.$ We observe the following:
	
	\smallskip
	\noindent{\em Observation 3:} There exists at most one $(x,y)$-disjoint chord in $B$.
	Indeed, if the contrary holds then $\{\hyperref[magnetic]{{O}_{1}^{0}},
	\hyperref[narcotic]{{O}_{3}^{0}}\}\leq G$ (see \autoref{minority}),
	a contradiction. 	\vspace{-3mm}
	\begin{figure}[H]
		\centering
		\vspace{-0.5cm}
		\begin{tikzpicture}[myNode/.style = black node, scale=.7]
		
		\begin{scope}
		\node[myNode, label=above:$x$] (A) at (0,2) {};
		\node[myNode, label=below:$y$] (B) at (0,0) {};
		
		\node[myNode] (v1) [below left= 0.1cm and .5cm of A] {};
		\node[myNode] (v2) [below left= 0.9cm and .5cm of A] {};
		\foreach \from/\to in {v1/A,v1/B,v2/A,v2/B, A/B}
		\draw (\from) -- (\to);
		
		\coordinate (Middle) at ($(A)!0.5!(B)$);
		\node (H) [right =0.3cm of Middle] {$H$};
		
		\node[myNode] (u1) [below left= 0.cm and 0.6 cm of B]  {};
		\node[myNode] (u2) [below right= 0.4cm and 0.25cm of u1]  {};
		\draw[-] (B.center) -- (u1) -- (u2)-- (B.center);

		\begin{scope}[on background layer]
		\filldraw[fill=green!70!blue,  opacity=0.7, dashed, thick]
		(B.center) to [bend right =110, min distance =5cm] coordinate[pos=0.2] (E11) coordinate[pos=0.3] (E21) coordinate[pos=0.7] (E22) coordinate[pos=0.8] (E12) (A.center);
		\end{scope}
		
		\node[myNode] (E11) at (E11) {};
		\node[myNode] (E22) at (E22) {};
		\node[myNode] (E21) at (E21) {};
		\node[myNode] (E12) at (E12) {};
		
		\draw[-] (E11) -- (E12) (E21) -- (E22);
		
		\end{scope}
		
		\begin{scope}[xshift=6cm]
		\node[myNode, label=above:$x$] (A) at (0,2) {};
		\node[myNode, label=below:$y$] (B) at (0,0) {};
		
		\node[myNode] (v1) [below left= 0.1cm and .5cm of A] {};
		\node[myNode] (v2) [below left= 0.9cm and .5cm of A] {};
		\foreach \from/\to in {v1/A,v1/B,v2/A,v2/B, A/B}
		\draw (\from) -- (\to);
		
		\coordinate (Middle) at ($(A)!0.5!(B)$);
		\node (H) [right =0.3cm of Middle] {$H$};
		
		\node[myNode] (u1) [below left= 0.cm and 0.6 cm of B]  {};
		\node[myNode] (u2) [below right= 0.4cm and 0.25cm of u1]  {};
		\draw[-] (B.center) -- (u1) -- (u2)-- (B.center);

		\begin{scope}[on background layer]
		\filldraw[fill=green!70!blue,  opacity=0.7, dashed, thick]
		(B.center) to [bend right =110, min distance =5cm] coordinate[pos=0.1] (E11) coordinate[pos=0.4] (E21) coordinate[pos=0.6] (E22) coordinate[pos=0.9] (E12) (A.center);
		\end{scope}
		
		\node[myNode] (E11) at (E11) {};
		\node[myNode] (E22) at (E22) {};
		\node[myNode] (E21) at (E21) {};
		\node[myNode] (E12) at (E12) {};
		
		\draw[-] (E11) -- (E21) (E12) -- (E22);
		
		\end{scope}
		
		\end{tikzpicture}
		\vspace{-.5cm}
		\vspace{-0mm}\caption{The two $(x,y)$-disjoint chords in Observation 1 of Case 2.}
		\label{minority}
	\end{figure}

	Observation 3 implies that $e$ is the unique $(x,y)$-disjoint chord in $B$.
	Now, if there exists some $x$-chord in $B$ then 
	$\{\hyperref[segments]{{O}_{7}^{1}},\hyperref[immature]{{O}_{11}^{1}}\} \leq G$
	(see \autoref{cadillac}),
	a contradiction. Therefore, every edge $e'\in E(H)$ except for $e$ that is a chord of
	$C$, is a $y$-chord and thus $y$ is a {${\cal P}$-apex} vertex of $G,$
	a contradiction. Claim 4 follows.	\vspace{-3mm}
	
	\begin{figure}[H]
		\centering
		\vspace{-.5cm}
		\begin{tikzpicture}[myNode/.style = black node, scale=.7]
		
		\begin{scope}
		\node[myNode, label=above:$x$] (A) at (0,2) {};
		\node[myNode, label=below:$y$] (B) at (0,0) {};
		
		\node[myNode] (v1) [below left= 0.1cm and 0.5cm of A] {};
		\node[myNode] (v2) [below left= 0.9cm and 0.5cm of A] {};

		\node[myNode] (u1) [below left= 0.cm and 0.6 cm of B]  {};
		\node[myNode] (u2) [below right= 0.4cm and 0.25cm of u1]  {};
		
		\begin{scope}[on background layer]
		\filldraw[fill=green!70!blue,  opacity=0.7, dashed, thick]
		(B.center) to [bend right =110, min distance =5cm] coordinate[pos=0.3] (A') coordinate[pos=0.7] (B') coordinate[pos=0.8] (C) (A.center);
		\end{scope}  
		
		\coordinate (Middle) at ($(A)!0.5!(B)$);
		\node (H) [right =0.3cm of Middle] {$H$};

		\foreach \from/\to in {v1/A,v1/B,v2/A,v2/B}
		\draw (\from) -- (\to);
		
		\draw[-] (B.center) to (u1);
		\draw[-] (B.center) to (u2);
		\draw[-] (u1) to (u2);
		
		\draw[-] (A)    to (B);
		\draw[-]  (A') node[myNode] () {} to (B') node[myNode] (){};
		\node (e) [right =1.5cm of Middle] {$e$};
		\draw[-] (A) to (C) node[myNode] (){};

		\end{scope}
		
		\begin{scope}[xshift=6cm]
		\node[myNode, label=above:$x$] (A) at (0,2) {};
		\node[myNode, label=below:$y$] (B) at (0,0) {};
		
		\node[myNode] (v1) [below left= 0.1cm and 0.5cm of A] {};
		\node[myNode] (v2) [below left= 0.9cm and 0.5cm of A] {};

		\node[myNode] (u1) [below left= 0.cm and 0.6 cm of B]  {};
		\node[myNode] (u2) [below right= 0.4cm and 0.25cm of u1]  {};
		
		\begin{scope}[on background layer]
		\filldraw[fill=green!70!blue,  opacity=0.7, dashed, thick]
		(B.center) to [bend right =110, min distance =5cm] coordinate[pos=0.3] (A') coordinate[pos=0.7] (B') coordinate[pos=0.2] (C) (A.center);
		\end{scope}  
		
		\coordinate (Middle) at ($(A)!0.5!(B)$);
		\node (H) [above right =0.3cm and 1.3cm of Middle] {$H$};

		\foreach \from/\to in {v1/A,v1/B,v2/A,v2/B}
		\draw (\from) -- (\to);
		
		\draw[-] (B.center) to (u1);
		\draw[-] (B.center) to (u2);
		\draw[-] (u1) to (u2);
		
		\draw[-] (A)    to (B);
		\draw[-]  (A') node[myNode] () {} to (B') node[myNode] (){};
		\node (e) [right =1.5cm of Middle] {$e$};
		\draw[-] (A) to (C) node[myNode] (){};

		\end{scope}

		\end{tikzpicture}
		\vspace{-3mm}\caption{A $(x,y)$-disjoint chord $e$ and an $x$-chord in the proof of Claim 4 of Case 2.}\label{cadillac}
	\end{figure}

	
	
	We now conclude the proof of the Lemma.
	Since $x$ is not an {${\cal P}$-apex} vertex of $G,$ there exists some chord of $C$
	not incident to $x$, namely $e$. By Claim 4, $e$ is a $y$-chord.
		\smallskip
	
%
%
%
	{Also, there exists at most one $x$-chord of $C$, since otherwise we have $\hyperref[immature]{{O}_{11}^{1}}\leq G$
    (see \autoref{insecure}), a contradiction.}
    
\vspace{-4mm}	
	
	\begin{figure}[H]
		\centering
		\vspace{-0.5cm}
		\begin{tikzpicture}[myNode/.style = black node, scale=.7]
		
		\node[myNode, label=above:$x$] (A) at (0,2) {};
		\node[myNode, label=below:$y$] (B) at (0,0) {};

		\node[myNode] (v1) [below left= 0.1cm and 0.5cm of A] {};
		\node[myNode] (v2) [below left= 0.9cm and 0.5cm of A] {};
		
		\coordinate (Middle) at ($(A)!0.5!(B)$);
		\node (H) [right =0.3cm of Middle] {$H$};
		
		\node[myNode] (u1) [below left= 0.cm and 0.6 cm of B]  {};
		\node[myNode] (u2) [below right= 0.4cm and 0.25cm of u1]  {};

		\foreach \from/\to in {v1/A,v1/B,v2/A,v2/B}
		\draw (\from) -- (\to);
		
		\draw[-] (B.center) to (u1);
		\draw[-] (B.center) to (u2);
		\draw[-] (u1) to (u2);
		
		\draw[-] (A) to (B);

		\begin{scope}[on background layer]
		\filldraw[fill=green!70!blue,  opacity=0.7, dashed, thick]
		(B.center) to [bend right =110, min distance =5cm] coordinate[pos=0.35] (A') coordinate[pos=0.65] (B') coordinate[pos=0.5] (C) (A.center);
		\end{scope}
		
		\draw[-] (B) to (A') node[myNode] () {};
		\draw[-] (A) to (B') node[myNode] () {};
		\draw[-] (A) to (C) node[myNode] () {};
		
		\end{tikzpicture}
		\vspace{-.5cm}
		\vspace{-0mm}\caption{A $y$-chord and two $x$-chords in the last part of the proof of Case 2.}
		\label{insecure}
	\end{figure}

	Therefore, it holds that $y$ is a {${\cal P}$-apex} vertex of $G,$
	a contradiction. This completes the proof of the Lemma.
\end{proof}

\begin{lemma}\label{destined}
	If $G\in \obs({\cal A}_{1}({\cal P}))\setminus  {\cal O}$ then $G$ is biconnected.
\end{lemma}

\begin{proof}
	{Suppose, towards a contradiction, that $G$ is not biconnected.
	Then, since by \autoref{hundreds}, $G$ is connected, there exists a cut-vertex $x$
    of $G$}.
Due to  \autoref{currency} we have that ${\cal C}(G,x) = \{B,D\}$,
where $B$ is biconnected and $D\cong K_{3}$. Also, \autoref{absolute} and \autoref{elicited} imply that $K_{4}\not\leq G$
and $K_{2,3}\not\leq G$ and so $G$ is an outerplanar graph. Thus, $B$ is also outerplanar.
	Following \autoref{decoding}, we can consider the Hamiltonian cycle $C$ of $B$. Therefore, the structure of $G$ is as in \autoref{parallel}:
	\vspace{-5mm}
	
	\begin{figure}[H]
		\centering
		\begin{tikzpicture}[myNode/.style = black node, scale=.8]
		
		\draw (0:0.5) node[myNode, label=above:$x$] {} -- (120:0.5) node[myNode] {} -- (240:0.5) node[myNode] {} -- cycle;
		
		\begin{scope}[on background layer]
		\filldraw[fill=mustard!90, dashed, thick] (0:0.5) to [out=60, in=90] (0:4.5) to [out=270, in=-60] (0:0.5);
		\end{scope}   
		
		\end{tikzpicture}
		\vspace{-2mm}\caption{{Illustration of the structure of the graph $G$ in the proof of \autoref{destined}.}}
	\label{parallel}
	\end{figure}

    We make the following observation:
		
		\smallskip
		
		\noindent{\em Observation:} Every two chords not incident to $x$ share a vertex.
		Indeed, suppose that there exist two disjoint chords $e_{1}=u_{1}v_{1},e_{2}=u_{2}v_{2}$
		of $C$ not incident to $x$. The possible configurations of $e_1$,$e_2$, due to outerplanarity of $B$, are depicted in \autoref{splendid}. Thus, $\{\hyperref[appetite]{{O}_{2}^{1}}, \hyperref[adorning]{{O}_{8}^{1}}\}\leq G$, a contradiction.
		\vspace{-3mm}
		
		\begin{figure}[H]
			\centering
			\begin{tikzpicture}[myNode/.style = black node, scale=.8]
			
			\draw (0:0.5) node[myNode, label=above:$x$] {} -- (120:0.5) node[myNode] {} -- (240:0.5) node[myNode] {} -- cycle;
			\begin{scope}[on background layer]
			\filldraw[fill=mustard!90, dashed, thick] (0:0.5) to [out=60, in=90] coordinate[pos=0.4] (u1) coordinate[pos=0.7] (u2) (0:4.5) to [out=270, in=-60] coordinate[pos=0.3] (v2) coordinate[pos=0.6] (v1) (0:0.5);
			\end{scope}
			
			\draw[-] (u1) node[myNode, label=above:$u_{1}$] {} to coordinate[pos=.5] (e1) (v1) node[myNode, label=below:$v_{1}$] {};
			\draw[-] (u2) node[myNode, label=above:$u_{2}$] {} to coordinate[pos=.5] (e2) (v2) node[myNode, label=below:$v_{2}$] {};
			
			\node[label=left:$e_{1}$] () at ($(e1)$) {};
			\node[label=left:$e_{2}$] () at ($(e2)$) {};
			
			\begin{scope}[xshift=6cm]
			
			\draw (0:0.5) node[myNode, label=above:$x$] {} -- (120:0.5) node[myNode] {} -- (240:0.5) node[myNode] {} -- cycle;
			\begin{scope}[on background layer]
			\filldraw[fill=mustard!90, dashed, thick] (0:0.5) to [out=60, in=90] coordinate[pos=0.4] (u1) coordinate[pos=0.9] (v1) (0:4.5) to [out=270, in=-60] coordinate[pos=0.1] (v2) coordinate[pos=0.6] (u2) (0:0.5);
			\end{scope}
			
			\draw[-] (u1) node[myNode, label=above:$u_{1}$] {} to coordinate[pos=.5] (e1) (v1) node[myNode, label=right:$v_{1}$] {};
			\draw[-] (u2) node[myNode, label=below:$u_{2}$] {} to coordinate[pos=.5] (e2) (v2) node[myNode, label=right:$v_{2}$] {};
			
			\node[label=below:$e_{1}$] () at ($(e1)$) {};
			\node[label=above:$e_{2}$] () at ($(e2)$) {};
			\end{scope}
			
			\end{tikzpicture}
			\vspace{-2mm}\caption{The two chords of $C$ disjoint to $x,y$ in the proof of Observation 1.}
			\label{splendid}
		\end{figure}

		We now continue with the proof of the Lemma. Since $x$ is not a ${\cal P}$-apex
		vertex of $G$ there exist two chords $e_{1},e_{2}$ not incident to $x$.
		By the above Observation, both $e_{1},e_{2}$ share a vertex $u$.
		But then, since $u$ is not a ${\cal P}$-apex vertex
		of $G$, there exists a chord $e$ not incident to $u$. 
		
		If $e$ is not incident to $x$ then Observation 1 implies that
		$e$ shares a vertex with both $e_{1},e_{2}$, in which case,
		$\hyperref[adorning]{{O}_{3}^{1}}\leq G$ (see \autoref{maintain}), a contradiction.
		
		\vspace{-5mm}
		
		\begin{figure}[H]

			\centering
			
			\begin{tikzpicture}[myNode/.style = black node, scale=.8]
			
			\draw (0:0.5) node[myNode, label=above:$x$] {} -- (120:0.5) node[myNode] {} -- (240:0.5) node[myNode] {} -- cycle;
			\begin{scope}[on background layer]
			\filldraw[fill=mustard!90, dashed, thick] (0:0.5) to [out=60, in=90] coordinate[pos=0.4] (u1) coordinate[pos=1] (v1) (0:4.5) to [out=270, in=-60]
			coordinate[pos=0.6] (u2) (0:0.5);
			\end{scope}
			
			\draw[-] (u1) node[myNode] {} to coordinate[pos=.4] (e1) (v1) node[myNode, label=right:$u$] {};
			\draw[-] (u2) node[myNode] {} to coordinate[pos=.4] (e2) (v1) node[myNode] {};
			\draw[-] (u1) to coordinate[pos=.5] (e3) (u2);
			
			\node[label=$e_{1}$] () at (-2:3) {};
			\node[label=$e_2$] () at (-16:3.1) {};
			\node[label=$e$] () at (-10:2) {};
			
			\end{tikzpicture}
			\vspace{-2mm}\caption{The case where $e$ is not incident to $x$.}\label{maintain}
		\end{figure}

		If now $e$ is incident to $x$ then,
		$\{\hyperref[segments]{{O}_{7}^{1}},\hyperref[immature]{{O}_{11}^{1}}, \hyperref[narcotic]{{O}_{3}^{0}}\}\leq G$ (see \autoref{waitress}), a contradiction.
	
		\begin{figure}[H]
		\centering
		\begin{tikzpicture}[myNode/.style = black node, scale=.8]
		\begin{scope}

		\draw (0:0.5) node[myNode, label=above:$x$] {} -- (120:0.5) node[myNode] {} -- (240:0.5) node[myNode] {} -- cycle;
		\begin{scope}[on background layer]
		\filldraw[fill=mustard!90, dashed, thick] (0:0.5) to [out=60, in=90] coordinate[pos=0.8] (u) coordinate[pos=0.4] (v1) (0:4.5) to [out=270, in=-60] coordinate[pos=0.4] (v2) coordinate[pos=0.5] (u2) (0:0.5);
		\end{scope}
		
		\node[myNode, label=above right:$u$] (u) at ($(u)$) {};
		\draw[-] (v1) node[myNode] {} to (u) to (v2) node[myNode] {};
		\draw[-] (0:0.5) to (u2) node[myNode] {};
		\node[label=$e$] () at (-20:2) {};
		
		\end{scope}
		
		\begin{scope}[xshift=6cm]
		
		\draw (0:0.5) node[myNode, label=above:$x$] {} -- (120:0.5) node[myNode] {} -- (240:0.5) node[myNode] {} -- cycle;
		\begin{scope}[on background layer]
		\filldraw[fill=mustard!90, dashed, thick] (0:0.5) to [out=60, in=90] coordinate[pos=0.8] (u) (0:4.5) to [out=270, in=-60] coordinate[pos=0.3] (v2) coordinate[pos=0.4] (v1) coordinate[pos=0.5] (u2) (0:0.5);
		\end{scope}
		
		\node[myNode, label=above right:$u$] (u) at ($(u)$) {};
		\draw[-] (v1) node[myNode] {} to (u) to (v2) node[myNode] {};
		\draw[-] (0:0.5) to (u2) node[myNode] {};
				\node[label=$e$] () at (-20:2) {};
		\end{scope} 
		
		\begin{scope}[xshift = 12cm]
		
		\draw (0:0.5) node[myNode, label=above:$x$] {} -- (120:0.5) node[myNode] {} -- (240:0.5) node[myNode] {} -- cycle;
		\begin{scope}[on background layer]
		\filldraw[fill=mustard!90, dashed, thick] (0:0.5) to [out=60, in=90] coordinate[pos=0.5] (u2) coordinate[pos=0.8] (u) (0:4.5) to [out=270, in=-60] coordinate[pos=0.3] (v2) coordinate[pos=0.4] (v1)  (0:0.5);
		\end{scope}
		
		\node[myNode, label=above right:$u$] (u) at ($(u)$) {};
		\draw[-] (v1) node[myNode] {} to (u) to (v2) node[myNode] {};
		\draw[-] (0:0.5) to (u2) node[myNode] {};
		\node[label=$e$] () at (-8:2) {};
		\end{scope}
		
		\end{tikzpicture}
		\vspace{-2mm}\caption{The possible configurations of the chords in the case where $e$ is incident to $x$.}
		\label{waitress}
	\end{figure}
\end{proof}

\subsection{Proving triconnectivity}

The purpose of this subsection is to prove that 
all graphs in $ \obs({\cal A}_{1}({\cal P}))\setminus  {\cal O}$ are triconnected
(\autoref{personal}).

\begin{lemma}\label{atrocity}
	If $G\in \obs({\cal A}_{1}({\cal P}))\setminus  {\cal O}$ then $K_{4}\not\leq G$  or $G$ is triconnected.
\end{lemma}

\begin{proof}
	Suppose, to the contrary, that $G$ is not triconnected and $K_{4}\leq G,$ which, by \autoref{absolute}, implies that $G$ is biconnected.
	Let $(H,K)$ be an ${r}$-wheel-subdivision pair of $G.$
	Since $G$ is biconnected and $G\not\in {\cal A}_{1}({\cal P})$, \autoref{patience} implies that $H$ is isomorphic to $K_{4}.$
	{We stress that, due to \autoref{campaign}(2), for every flap $F$ of $(H,K)$, $G[V(F)]$ contains a cycle.}

%
%
%
%
%
%
	
	
	{\smallskip
	\noindent{\em Observation:} Every flap of $(H,K)$ is an $(x,u)$-flap where $x$ is a
	branch vertex of $K$ and $u$ is a vertex in the subdivision of an edge
	$e=xy\in E(H)$ in $K$. Indeed, let $F$ be a $(u,v)$-flap of $(H,K)$. By \autoref{campaign}(3),
	there exists an edge $e\in E(H)$ such that $u,v$ belong
	both to the subdivision of $e$ in $K$. If neither of $u,v$ is a
	branch vertex of $K$ then, since, by \autoref{campaign}(2), $G[V(F)]$ contains a
	cycle, we have that $\hyperref[holiness]{{O}_{9}^{2}}\leq G$
	(see \autoref{reproved}), a contradiction.
    }

	\begin{figure}[H]
		\centering
		\begin{tikzpicture}[every node/.style = black node, scale=.8]
		
		\draw[dashed, thick] (30:1.5) -- (0,0) -- (270:1.5) -- cycle;
		\draw[dashed, thick] (0,0) -- (150:1.5) -- (270:1.5);
		
		\fill[applegreen!80] (130:1) .. controls +(0,0.9) and +(0,0.9) .. (50:1) {};
		
		\draw[dashed, thick] (130:1) --  (50:1) -- (30:1.5);
		\draw[dashed, thick] (130:1) -- (150:1.5) ;
		
		\draw (150:1.5) node {};
		\draw (0,0) node {};
		\draw (270:1.5) node {};
		\draw (30:1.5) node {};
		\draw (50:1) node[label=30:$v$] {};
		\draw (130:1) node[label=120:$u$] {};
		
		\end{tikzpicture}
		\vspace{-1mm}\caption{An example of a $(u,v)$-flap such that both $u,v$ are subdividing vertices of the corresponding path.}
		\label{reproved}
	\end{figure}
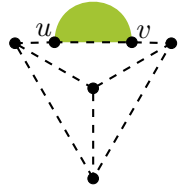

	{Following the above Observation}, we distinguish the following two cases, depending on whether there exists an $(x,u)$-flap of $(H,K)$ such that $x$ is a branch vertex of $K$ and $u$ is a subdividing vertex in the subdivision of an  $e=x y \in E(H)$ in $K$ or for every flap its base is a subset of the branch vertices of $K$:

	\smallskip
	
	\noindent {\em Case 1:} There exists an $(x,u)$-flap $F$, where $x$ is a branch vertex of $K$ and $u$ is a subdividing vertex in the subdivision of an edge $e=x y \in E(H)$ in $K$.
	
	Let $z_{1},z_{2}$ be the two other branch vertices of $K$ (as
	shown in \autoref{backdrop}).

	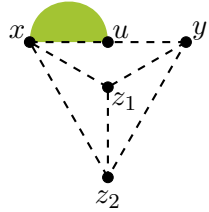
\begin{figure}[H]
		\centering
		\begin{tikzpicture}[every node/.style = black node, scale=.8]
		
		\draw[dashed, thick] (30:1.5) -- (0,0) -- (270:1.5) -- cycle;
		\draw[dashed, thick] (0,0) -- (150:1.5) -- (270:1.5);
		
		\fill[applegreen!80] (150:1.5) .. controls +(0,0.9) and +(0,0.9) .. (0,0.75) {};
		
		\draw[dashed, thick]  (150:1.5) -- (0,0.75);
		\draw[dashed, thick] (30:1.5) --  (0,0.75) ;
		
		\draw (150:1.5) node[label=150:$x$] {};
		\draw (0,0) node[label=-30:$z_{1}$] {};
		\draw (270:1.5) node[label=270:$z_{2}$] {};
		\draw (30:1.5) node[label=30:$y$] {};
		\draw (0,0.75) node[label=30:$u$] {};

		\end{tikzpicture}
		\vspace{-2mm}\caption{The structure of $G$ in Case 1.}\label{backdrop}
	\end{figure}
	
	\noindent{\em Claim 1:} All flaps of $(H,K)$ are $(x,w)$-flaps where $w$ is a vertex in the subdivision of an edge $e^{\prime}\in \{xy,xz_{1},xz_{2}\}$ in $K.$
	\smallskip
	
	\noindent{\em Proof of Claim 1:} Suppose, to the contrary, that Claim 1 does not hold. {Following the above Observation,} we distinguish the following subcases:
	
	\smallskip

	\noindent{\em Subcase 1.1:} There exists a $(y,v)$-flap $F'$
	where $v$ is a subdividing vertex in the subdivision of $e$ in $K.$
	{Observe that $F\neq F'$.}
	By \autoref{campaign}, it holds that $V(F)\cap V(F')\subseteq \{x,u\}\cap \{y,v\}$.
	Therefore, if $u=v$ then $V(F)\cap V(F')=\{u\}$ and so
	$\hyperref[begrudge]{{O}_{4}^{2}}\leq G$, while if $u\neq v$ then
	$V(F)\cap V(F')=\emptyset$ and so
	$\hyperref[consists]{{O}_{1}^{1}}\leq G$ (see \autoref{dialogic}),
	a contradiction in both cases.

	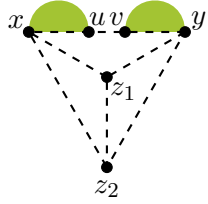
\begin{figure}[H]
		\centering
		\begin{tikzpicture}[every node/.style = black node, scale=.8]
		\draw[dashed, thick] (30:1.5) -- (0,0) -- (270:1.5) -- cycle;
		\draw[dashed, thick] (0,0) -- (150:1.5) -- (270:1.5);
		
		\fill[applegreen!80] (150:1.5) .. controls +(0,0.7) and +(0,0.7) .. (-0.3,0.75) {};
		
		\fill[applegreen!80] (30:1.5) .. controls +(0,0.7) and +(0,0.7) .. (0.3,0.75) {};
		
		\draw[dashed, thick] (30:1.5) -- (150:1.5);
		
		\draw (150:1.5) node[label=150:$x$] {};
		\draw (0,0) node[label=-30:$z_{1}$] {};
		\draw (270:1.5) node[label=270:$z_{2}$] {};
		\draw (30:1.5) node[label=30:$y$] {};
		\draw (-0.3,0.75) node[label=80:$u$] {};
		\draw (0.3,0.75) node[label=100:$v$] {};
		
		\end{tikzpicture}
		\vspace{-2mm}\caption{The configuration of the flaps of $(H,K)$ in Subcase 1.1 of Claim 1.}\label{dialogic}
	\end{figure}

	\noindent{\em Subcase 1.2:} There exists a flap {whose base is} in the subdivision
	of the edges $y z_{1}$ or $y z_{2}.$ But, then $\hyperref[operetta]{{O}_{6}^{2}}\leq G$ (see \autoref{executed}), a contradiction.
	
	\begin{figure}[H]
		\centering
		\vspace{-.5cm}
		\begin{tikzpicture}[every node/.style = black node, scale=.8]
		
		\begin{scope}
		\draw[dashed, thick, red] (30:1.5) -- (0,0) ;
		\draw[dashed, thick, red] (30:1.5) -- (270:1.5) ;
		\draw[dashed, thick] (0,0) -- (150:1.5) -- (270:1.5)--cycle;
		\draw[dashed, thick] (150:1.5) -- (30:1.5);
		
		\begin{scope}[on background layer]
		\fill[applegreen!80] (150:1.5) .. controls +(0,0.9) and +(0,0.9) .. (0,0.75) {};
		
		\fill[applegreen!80] (30:1.5) .. controls +(330:1.8) and +(330:1.8) .. (270:1.5) {};
		\end{scope}

		\draw (150:1.5) node[label=150:$x$] {};
		\draw (0,0) node[label=90:$z_{1}$] {};
		\draw (270:1.5) node[label=270:$z_{2}$] {};
		\draw (30:1.5) node[label=30:$y$] {};
		\draw (0,0.75) node[label=30:$u$] {};
		\end{scope}
		
		\end{tikzpicture}
		\vspace{-2mm}\caption{A flap whose base is in the subdivision of some edge $yz_{i}, \ i \in[2]$ (depicted in red) in Subcase 1.2 of Claim 1.}
		\label{executed}
	\end{figure}
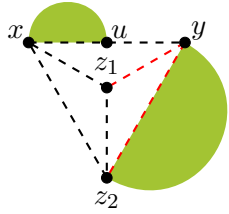

	\noindent{\em Subcase 1.3:} There exists a flap whose base is in the subdivision of the
	edge $z_{1} z_{2}.$ But, then $\hyperref[bringing]{{O}_{5}^{2}}\leq G$
	(see \autoref{infested}), a contradiction.
	
	\begin{figure}[H]
		\centering
		
		\begin{tikzpicture}[every node/.style = black node, scale=.8]
		
		\begin{scope}
		
		\draw[dashed, thick] (30:1.5) -- (270:1.5) ;
		
		\fill[applegreen!80] (150:1.5) .. controls +(0,0.9) and +(0,0.9) .. (0,0.75) {};
		\fill[applegreen!80] (0,0) .. controls +(355:0.5) and  +(320:0.3).. (270:1.5) {};

		\draw[dashed, thick] (0,0) -- (150:1.5) -- (270:1.5) -- cycle;
		\draw[dashed, thick] (150:1.5) -- (30:1.5)-- (0,0) ;
		
		\draw (150:1.5) node[label=150:$x$] {};
		\draw (0,0) node[label=90:$z_{1}$] {};
		\draw (270:1.5) node[label=270:$z_{2}$] {};
		\draw (30:1.5) node[label=30:$y$] {};
		\draw (0,0.75) node[label=30:$u$] {};
		\end{scope}
		
		\end{tikzpicture}
		\vspace{-2mm}\caption{A flap whose base is in the subdivision of $z_{1}z_{2}$ in Subcase 1.3 of Claim 1.}
		\label{infested}
	\end{figure}
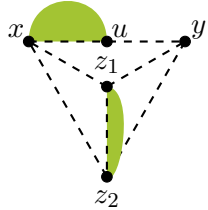

	\noindent{\em Subcase 1.4:} There exists a $(z_{i},v)$-flap such that $v$ is a subdividing vertex in the subdivision
	of the edge $z_{i} x$, $i\in [2]$. But, then $\hyperref[appetite]{{O}_{2}^{1}}\leq G$ (see \autoref{goethean}), a contradiction.
	
	\begin{figure}[H]
		\centering
		\begin{tikzpicture}[every node/.style = black node, scale=.8]
		
		\fill[applegreen!80] (150:1.5) .. controls +(0,0.9) and +(0,0.9) .. (0,0.75) {};
		
		\fill[applegreen!80] (210:0.75) .. controls +(205:1.15) and +(215:1.2)  .. (270:1.5) {};
		
		\draw[dashed, thick] (270:1.5) -- (30:1.5) -- (0,0);
		\draw[dashed, thick, blue] (0,0) -- (150:1.5) -- (270:1.5) ;
		\draw[dashed, thick] (30:1.5) -- (150:1.5) ;
		\draw[dashed, thick] (270:1.5) -- (0,0);
		
		\draw (150:1.5) node[label=150:$x$] {};
		\draw (0,0) node[label=-30:$z_{1}$] {};
		\draw (270:1.5) node[label=270:$z_{2}$] {};
		\draw (30:1.5) node[label=30:$y$] {};
		\draw (0,0.75) node[label=30:$u$] {};
		\draw (210:0.75) node[label=180:$v$] {};
		
		\end{tikzpicture}
		\vspace{-2mm}\caption{A flap whose base is in the subdivision of some edge $z_{i} x, \ i \in[2]$ (depicted in blue) in Subcase 1.4 of Claim 1.}
		\label{goethean}
	\end{figure}
	
	Since we have exhausted all possible cases, Claim 1 follows.
	
	\medskip
	
	Now, to conclude Case 1 we prove the following:
	
	\medskip
	
	\noindent{\em Claim 2:} Every flap of $(H,K)$ is $x$-oriented.
	
	\smallskip
	
	\noindent{\em Proof of Claim 2:}
	Suppose, towards a contradiction, that there exists a flap $F'$ that is not $x$-oriented.
	Then, by Claim 1, this is an $(x,w)$-flap where $w$ is
	some vertex in the subdivision of an edge
	$e'\in\{xy,xz_{1},xz_{2}\}.$ Since $F'$ is not $x$-oriented then \autoref{mobility} implies that there exists a cycle $C$ in $F'$ that contains $w$ but not $x$.
	If $e'=x z_{i},$ for some $i\in \{1,2\},$ then $\hyperref[appetite]{{O}_{2}^{1}}\leq G$ (see the left figure of \autoref{fiendish}), a
	contradiction. Thus, $e'=xy$. Then, if $w\not=y$ we have that
	$\hyperref[skeleton]{{O}_{4}^{1}}\leq G$ (see the central figure of \autoref{fiendish}), a contradiction. Therefore, $w=y$ and so
	$F'\neq F$, which, by \autoref{campaign}, implies that $V(F)\cap V(F')=\{x\}$. Therefore, by
	contracting $F$ to a cycle we get a cycle containing $x$ but not $y$
	disjoint to $C$. Hence $\hyperref[consists]{{O}_{1}^{1}}\leq G$
	 (see the right figure of \autoref{fiendish}), a contradiction. Claim 2 follows.

	\begin{figure}[H]
		\centering
		\vspace{-.5cm}
		\begin{tikzpicture}[every node/.style = black node, scale=.8]
		
		\begin{scope}[xshift=5cm]
		
		\coordinate (x) at (150:1.5);
		\coordinate (y) at (30:1.5);
		\coordinate (u) at ($(x.center)!0.6!(y.center)$) {};
		\fill[applegreen!80] (x.center)to [out= 90, in=90, min distance=1cm]  (u.center) {};
		\draw[-] (u.center).. controls ++(180:1.7) and ++(110:1.5).. (u.center);
		
		\draw[dashed, thick] (270:1.5) -- (30:1.5) -- (0,0);
		\draw[dashed, thick] (0,0) -- (150:1.5) -- (270:1.5) ;
		\draw[dashed, thick] (30:1.5) -- (150:1.5) ;
		\draw[dashed, thick] (270:1.5) -- (0,0);
		
		\draw (150:1.5) node[label=150:$x$] {};
		\draw (0,0) node[label=-30:$z_{1}$] {};
		\draw (270:1.5) node[label=270:$z_{2}$] {};
		\draw (30:1.5) node[label=30:$y$] {};
		\draw (u) node[label=30:$w$] {};

		\end{scope}
		
		\begin{scope}
		
		\coordinate (x) at (150:1.5);
		\coordinate (y) at (30:1.5);
		\coordinate (z2) at (270:1.5);
		\node (m1) at ($(x.center)!0.6!(y.center)$) {};
		\node (m2) at ($(x.center)!0.6!(z2.center)$) {};
		
		\begin{scope}[on background layer]
		\fill[applegreen!80] (x.center) to [out= 90, in=90, min distance=1cm] (m1.center) {};
		\fill[applegreen!80] (x.center) to [bend right = 90, min distance=1cm] (m2.center) {};
		\path (m2) edge [out=190,in=130, looseness=27] (m2);
		\end{scope}
		
		\draw[dashed, thick] (270:1.5) -- (30:1.5) -- (0,0) -- (150:1.5) -- (270:1.5);
		\draw[dashed, thick] (30:1.5) -- (150:1.5) ;
		\draw[dashed, thick] (270:1.5) -- (0,0);

		\draw (m1.center) node[label=above right:$u$] {};
		
		\draw (150:1.5) node[label=150:$x$] {};
		\draw (0,0) node[label=-30:$z_{1}$] {};
		\draw (270:1.5) node[label=270:$z_{2}$] {};
		\draw (30:1.5) node[label=30:$y$] {};

		\end{scope}
		
		\begin{scope}[xshift=10cm]
		
		\coordinate (x) at (150:1.5);
		\coordinate (y) at (30:1.5);
		\node (m) at ($(x.center)!0.5!(y.center)$) {};

		\fill[applegreen!80] (150:1.5) to [out= 60, in=90, min distance=2cm] (30:1.5) {};
		\fill[red!30!yellow, opacity=0.7] (150:1.5) to [out= 100, in=120, min distance=1cm] (m.center) {};

		\draw[dashed, thick] (270:1.5) -- (30:1.5) -- (0,0) -- (150:1.5) -- (270:1.5);
		\draw[dashed, thick] (30:1.5) -- (150:1.5) ;
		\draw[dashed, thick] (270:1.5) -- (0,0);
		
		\draw[-] (30:1.5).. controls ++(170:2.2) and ++(110:2).. (30:1.5);
		
		\draw (m.center) node[label=above:$u$] {};
		
		\draw (150:1.5) node[label=150:$x$] {};
		\draw (0,0) node[label=-30:$z_{1}$] {};
		\draw (270:1.5) node[label=270:$z_{2}$] {};
		\draw (30:1.5) node[label=30:$y$] {};
		
		\end{scope}
		\end{tikzpicture}
		\vspace{-2mm}\caption{The possible configurations of the non $x$-oriented flap $F'$ in the proof of Claim 2.}
		\label{fiendish}
	\end{figure}
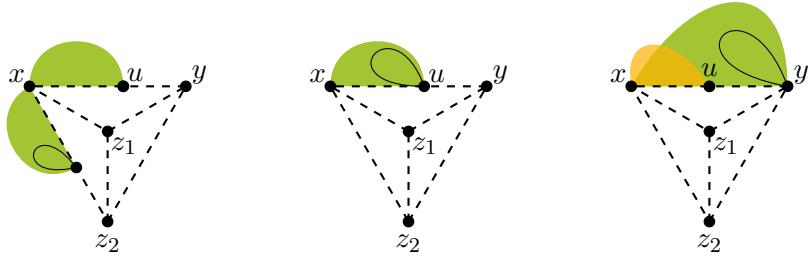

	Claim 2 {implies that every cycle of every flap contains $x$. Hence,
    \autoref{fanfares} implies that every cycle in $G$ except for one contains $x$}.
     Therefore, $x$ is a {${\cal P}$-apex} vertex of $G$ and so we arrive at a contradiction.
	
	\medskip
	
	\noindent {\em Case 2:} Every flap is a $(u,v)$-flap, where $u,v$ are branch vertices of $K.$
	
	\medskip
	
	We argue that the following holds:
	
	
	
	
	
	\smallskip
	
	\noindent{\em Claim 3:} 
	There is a branch vertex $x$ of  $K$ such that the bases of all flaps of $(H,K)$ share $x$.
	
	\smallskip
	
	\noindent{\em Proof of Claim 3:} To prove Claim 3 we make the following observation:
	
	\smallskip
	
	\noindent{\em Observation:} 
	For every two flaps of $(H,K)$, there exists a branch vertex $x$ of $K$ such that their bases share $x$.
	Indeed, suppose to the contrary, that there exist two flaps of $(H,K)$, an $(x,y)$-flap and an $(x',y')$-flap,
	such that $xy,x'y'$ are two non-incident edges of $H$. But then,
	$\hyperref[bringing]{{O}_{5}^{2}}\leq G,$ a contradiction (see \autoref{greyness}).
	
	\begin{figure}[H]
		\centering
		\vspace{-.5cm}
		\begin{tikzpicture}[every node/.style = black node, scale=.8]
		
		\node[label=150:$x$] (x) at (150:1.5)  {};
		\node[label=90:$x'$] (z1) at (0,0) {};
		\node[label=270:$y'$] (z2) at (270:1.5)  {};
		\node[label=30:$y$] (y) at (30:1.5)  {};
		
		\begin{scope}[on background layer]
		\draw[dashed, thick] (z2) -- (y) -- (z1) -- (x) -- (z2);
		\fill[applegreen!80] (x.center) .. controls +(0,1.8) and +(0,1.8) .. (y.center) {};
		\fill[applegreen!80] (z2.center) .. controls +(355:0.5) and  +(320:0.3).. (z1.center) {};
		\end{scope}
		
		\draw[dashed, thick] (z2) -- (z1);
		\draw[dashed, thick] (x) -- (y) ;
		
		\end{tikzpicture}
		\vspace{-2mm}\caption{The configuration of the flaps of $G$ in the proof of Observation.}
		\label{greyness}
	\end{figure}
	
{If $F_{1}$, $F_{2}$ are two different flaps of $(H,K)$, then, by the above Observation, there exist two branch vertices $y,z$ of $K$ {different from $x$} such that $F_{1}$ is an $(x,y)$-flap and $F_{2}$ is an $(x,z)$-flap.
Recall that a 2-separator $S\subseteq V(K)$ of $G$ defines at most one flap,
and hence we have that $y\neq z$.
Moreover, if there exists a flap $F_{3}$ of $(H,K)$ different than $F_1, F_2$, then its base also contains $x$, since otherwise, the above Observation implies that $F_{3}$ is a $(y,z)$-flap and therefore $\hyperref[furrowed]{{O}_{8}^{2}}\leq G$ (see \autoref{sublated}),
	a contradiction. Claim 3 follows.
	}

	
	\begin{figure}[H]
		\centering
		\vspace{-.5cm}
		\begin{tikzpicture}[every node/.style = black node, scale=.8]
		
		\begin{scope}
		\draw[dashed, thick] (30:1.5) -- (0,0) ;
		\draw[dashed, thick] (270:1.5) -- (0,0);
		\draw[dashed, thick] (150:1.5) -- (0,0);
		
		\fill[applegreen!80] (150:1.5) .. controls +(0,1.8) and +(0,1.8) .. (30:1.5) {};
		\fill[applegreen!80] (30:1.5) .. controls +(330:1.8) and  +(330:1.8)  .. (270:1.5) {};
		\fill[applegreen!80] (150:1.5) .. controls +(210:1.8) and  +(210:1.8)  .. (270:1.5) {};
		
		\draw[dashed, thick] (30:1.5) -- (270:1.5) ;
		\draw[dashed, thick] (150:1.5) -- (270:1.5);
		\draw[dashed, thick] (30:1.5) -- (150:1.5) ;
		
		\draw (150:1.5) node[label=150:$x$] {};
		\draw (0,0) node[] {};
		\draw (270:1.5) node[label=270:$z$] {};
		\draw (30:1.5) node[label=30:$y$] {};
		\end{scope}
		
		\end{tikzpicture}
		\vspace{-.5cm}
		\vspace{-0mm}\caption{The configuration of the flaps of $G$ in the end of the proof of Claim 3.}\label{sublated}
	\end{figure}
%
%
	
	{According to Claim 3, there exists a branch vertex $x\in V(K)$ such that every 
	flap of $(H,K)$ is an $(x,y)$-flap, where $y$ is a branch vertex of $K$ different
	from $x$. The following claim will conclude Case 2 and thus the Lemma.} 

	\medskip
	
	\noindent{\em Claim 4:} Every flap of $(H,K)$ is $x$-oriented.
	
	\smallskip
	
	\noindent{\em Proof of Claim 4:}
    Suppose, towards a contradiction, that there exists an $(x,y)$-flap $F$ of
	$(H,K)$ that is not $x$-oriented.
	Then, \autoref{mobility} implies that $F$ is
	$y$-oriented and so there exists a cycle in $F$ that contains $y$ but not $x$. 
	
	We will prove that $F$ is the unique flap of $(H,K).$ Indeed, if there
	exists a flap $F'$ of $(H,K)$ that is different than $F$, then it is an $(x,y')$-flap, where $y'$ is a
	branch vertex of $K$ different from both $x$ and $y$.
	As above, we have that $y\neq y'$, since $F\neq F'$. Thus,
	$\hyperref[friesian]{{O}_{5}^{1}} \leq G$ (see \autoref{ancients}),  a contradiction. 
	
	Therefore, $F$ is the only flap of $(H,K)$ and since it is $y$-oriented, 
	{\autoref{fanfares} implies that} $y$ is a {${\cal P}$-apex} vertex of $G,$ a contradiction. Claim 4 follows.
	
	\begin{figure}[H]
		\centering
		\vspace{-.5cm}
		\begin{tikzpicture}[every node/.style = black node, scale=.8]
		\begin{scope}
		\fill[applegreen!80] (150:1.5) .. controls +(0,1.8) and +(0,1.8) .. (30:1.5) {};
		\fill[applegreen!80] (150:1.5) .. controls +(210:1.8) and  +(210:1.8)  .. (270:1.5) {};
		
		\draw[dashed, thick] (270:1.5) -- (30:1.5) -- (0,0) -- (150:1.5) -- (270:1.5);
		\draw[dashed, thick] (30:1.5) -- (150:1.5) ;
		\draw[dashed, thick] (270:1.5) -- (0,0);
		
		\draw[-] (30:1.5).. controls ++(170:2.2) and ++(110:2).. (30:1.5);
		
		\draw (150:1.5) node[label=150:$x$] {};
		\draw (0,0) node {};
		\draw (270:1.5) node[label=270:$y'$] {};
		\draw (30:1.5) node[label=30:$y$] {};
		
		\end{scope}
		\end{tikzpicture}
		\vspace{-.4cm}
		\vspace{-0mm}\caption{An example of the $(x,y)$-flap and the $(x,y')$-flap in the proof of Claim 4.} \label{ancients}
	\end{figure}
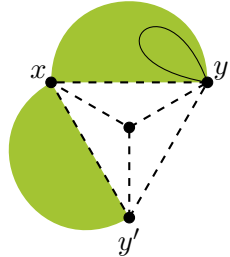
	
	Claim 4 {and \autoref{fanfares}}, taking account that $G$ is biconnected,
	imply that $x$ is a {${\cal P}$-apex} vertex of $G,$  a contradiction. 
	This concludes Case 2 and hence \autoref{atrocity}.
\end{proof}

\begin{lemma}\label{toulouse}
	If $G\in \obs({\cal A}_{1}({\cal P}))\setminus  {\cal O}$ 
	then $K_{2,3}\not\leq G$  or $G$ is triconnected.
\end{lemma}

\begin{proof}
	
	Suppose, to the contrary, that $G$ is not triconnected and $K_{2,3}\leq G,$ which, by \autoref{elicited}, implies that $G$ is biconnected.
	Also, since $G$ is not triconnected, by  \autoref{atrocity}, it is $K_{4}$-free.
	
	Since $G\in \obs({\cal A}_{1}({\cal P}))\setminus  {\cal O}$, by \autoref{princess}, every b-rich separator of $G$ is nice. Also, by \autoref{negative} we have that there exists a b-rich separator $S=\{x,y\}$ of $G$, that is unique due to \autoref{symphony}.
	 We argue that the following holds:
	
	\medskip
	
	\noindent{\em Claim 1:} There exists a unique augmented component in ${\cal C}(G,S)$ not isomorphic to a cycle.
	
	\smallskip
	
	\noindent{\em Proof of Claim 1:} Observe that there exists an augmented connected component in ${\cal C}(G,S)$ not isomorphic to a cycle, otherwise $G\in {\cal A}_{1}({\cal P})$.
	Suppose then, towards a contradiction, that there exist two augmented connected components in ${\cal C}(G,S)$ that are not isomorphic to a cycle. Since $S$ is the unique b-rich separator of $G$, then by \autoref{outerplanarfl}, we have that every $H\in {\cal C}(G,S)$ is outerplanar. Thus, by	
	\autoref{caillois}, one of the following holds:
	
	\begin{itemize}
		\item There exists a unique $(x,y)$-disjoint chord of $G$ and there do not exist both $x$-chords and $y$-chords in $G$. But then, if there do not exist $x$-chords (or $y$-chords) of $G$, then $y$ (or $x$, respectively) is a ${\cal P}$-apex of $G$, a contradiction.
		
		\item There do not exist $(x,y)$-disjoint chords of $G$ and there exists at most one $x$-chord or at most one $y$-chord of $G$. But then, if there exists at most one $x$-chord (or $y$-chord) of $G$, then $y$ (or $x$, respectively) is a ${\cal P}$-apex of $G$, a contradiction.
	\end{itemize}
	Since we arrived at a contradiction in both cases Claim 1 follows.
	
	\medskip
	
	According to Claim 1, let $H$ be the unique augmented connected component in ${\cal C}(G,S)$ that
	is not isomorphic to a cycle. Note that by \autoref{generals} every
	$H'\in{\cal C}(G,S)\setminus\{H\}$ is isomorphic to $K_{3}$.
	Also, since $H$ is biconnected and outerplanar, following \autoref{decoding},
	we can consider the Hamiltonian cycle $C$ of $H$.
	Notice that, due to $K_{4}$-freeness of $G$, $xy\in E(C).$
	
	\medskip
	
	\noindent{\em Claim 2:} Every chord of $C$ is either an $x$-chord or a $y$-chord of $G$.
	
	\smallskip
	
	\noindent {\em Proof of Claim 2:}
	Suppose, towards a contradiction, that there exists some $(x,y)$-disjoint chord $uv$ of $G$.
	Observe that there is a unique such chord, since otherwise $\{\hyperref[magnetic]{{O}_{1}^{0}},\hyperref[narcotic]{{O}_{3}^{0}}\} \leq G$. 
	
	Suppose, without loss of generality, that the $(x,u)$-subpath of $C \setminus xy$ does not contain $v$
	-- we denote this path by $P_1.$ Let $P_2$ be the $(y,v)$-path in $C \setminus xy$ (as shown in \autoref{fidelity}).

	\begin{figure}[H]
		\centering
		\vspace{-.5cm}
		\begin{tikzpicture}[myNode/.style = black node, scale=.7]
		
		\node[myNode, label=above:$x$] (A) at (0,2) {};
		\node[myNode, label=below:$y$] (B) at (0,0) {};

		\node[myNode] (v1) [below left= 0.1cm and 0.5cm of A] {};
		\node[myNode] (v2) [below left= 0.9cm and 0.5cm of A] {};
		
		\coordinate (Middle) at ($(A)!0.5!(B)$);
		\node (H) [right =0.15cm of Middle] {$H$};
		
		\begin{scope}[on background layer]
		\filldraw[fill=green!70!blue,  opacity=0.7, dashed, thick]
		(B.center) to [bend right =110, min distance =5cm] coordinate[pos=0.15] (p2) coordinate[pos=0.3] (v) coordinate[pos=0.7] (u) coordinate[pos=0.85] (p1) (A.center);
		\end{scope}

		\node[myNode, label=above:$u$] (C) at ($(u)$) {};
		\node[myNode, label=below:$v$] (D) at ($(v)$) {};
		
		\node[label=above:$P_{1}$] () at ($(p1)$) {};
		\node[label=below:$P_{2}$] () at ($(p2)$) {};
		
		\foreach \from/\to in {v1/A,v1/B,v2/A,v2/B}
		\draw (\from) -- (\to);
		
		\draw[-] (A)    to (B);
		\draw[-] (C)    to (D);

		\end{tikzpicture}
		\vspace{-.5cm}
		\vspace{-0mm}\caption{The chord $uv$ and paths $P_{1},$ $P_{2}$ as in the proof of Claim 1.}\label{fidelity}
	\end{figure}
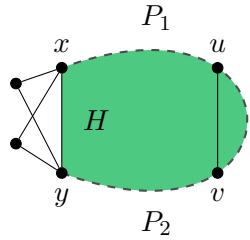

	We argue that the following holds:
	
	\smallskip 
	
	\noindent{\em Subclaim:} All chords of $C$, other than $uv$, are incident to the same vertex of the separator $\{x,y\}$.
	
	\smallskip
	
	\noindent{\em Proof of Subclaim:}
	Suppose, towards a contradiction, that there exists an $x$-chord $xx'$ of $G$ and a $y$-chord $yy'$ of $G$.
	{Then, due to outerplanarity of $H$, $x'$ and $y'$ are vertices in $V(P_{1})\cup V(P_2)$.
	If $x' \in V(P_{1})$ and $y' \in V(P_{2}),$ then $\hyperref[softened]{{O}_{3}^{2}} \leq G$ (see leftmost figure of  \autoref{estrange}) while if  both $x'$ and $y'$ belong to the
	same $P_{i}, i \in [2]$, then $\hyperref[supports]{{O}_{10}^{2}} \leq G$ (see rightmost figure of \autoref{estrange}). Subclaim follows.}
	Then, it cannot be the case that $x' \in V(P_{1})$ and $y' \in V(P_{2}),$ otherwise $\hyperref[softened]{{O}_{3}^{2}} \leq G.$
	On the other hand, it also cannot be the case that both $x'$ and $y'$ belong to the
	same $P_{i}, i \in [2]$ since that implies that
	$\hyperref[supports]{{O}_{10}^{2}} \leq G$ (see \autoref{estrange}).
	Subclaim follows.

	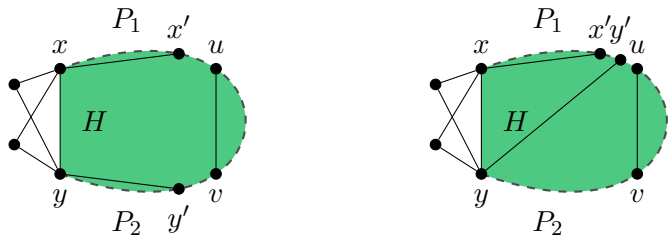
\begin{figure}[H]
		\centering
		\vspace{-.5cm}
		\begin{tikzpicture}[myNode/.style = black node, scale=.7]
		\begin{scope}
		
		\node[myNode, label=above:$x$] (A) at (0,2) {};
		\node[myNode, label=below:$y$] (B) at (0,0) {};

	    \node[myNode] (v1) [below left= 0.1cm and 0.5cm of A] {};
		\node[myNode] (v2) [below left= 0.9cm and 0.5cm of A] {};
		
		\coordinate (Middle) at ($(A)!0.5!(B)$);
		\node (H) [right =0.15cm of Middle] {$H$};
		
		\begin{scope}[on background layer]
		\filldraw[fill=green!70!blue,  opacity=0.7, dashed, thick]
		(B.center) to [bend right =110, min distance =5cm] coordinate[pos=0.1] (p2) coordinate[pos=0.2] (yy) coordinate[pos=0.3] (v) coordinate[pos=0.7] (u) coordinate[pos=0.8] (xx) coordinate[pos=0.9] (p1) (A.center);
		\end{scope}  
		
		\node[myNode, label=above:$u$] (C) at ($(u)$) {};
		\node[myNode, label=below:$v$] (D) at ($(v)$) {};
		
		\node[label=above:$P_{1}$] () at ($(p1)$) {};
		\node[label=below:$P_{2}$] () at ($(p2)$) {};
		\node[myNode, label=above:$x'$] () at ($(xx)$) {};
		\node[myNode, label=below:$y'$] () at ($(yy)$) {};

		\foreach \from/\to in {v1/A,v1/B,v2/A,v2/B}
		\draw (\from) -- (\to);
		
		\draw[-] (A)    to (B);
		\draw[-] (C)    to (D);
		\draw[-] (A) to (xx);
		\draw[-] (B) to (yy);
		
		\end{scope}
		
		\begin{scope}[xshift=8cm]
		
		\node[myNode, label=above:$x$] (A) at (0,2) {};
		\node[myNode, label=below:$y$] (B) at (0,0) {};
		
		\node[myNode] (v1) [below left= 0.1cm and 0.5cm of A] {};
		\node[myNode] (v2) [below left= 0.9cm and 0.5cm of A] {};
		
		\coordinate (Middle) at ($(A)!0.5!(B)$);
		\node (H) [right =0.15cm of Middle] {$H$};
		
		\begin{scope}[on background layer]
		\filldraw[fill=green!70!blue,  opacity=0.7, dashed, thick]
		(B.center) to [bend right =110, min distance =5cm] coordinate[pos=0.1] (p2)  coordinate[pos=0.3] (v) coordinate[pos=0.7] (u) coordinate[pos=0.75] (yy) coordinate[pos=0.8] (xx) coordinate[pos=0.9] (p1) (A.center);
		\end{scope}  
		
		\node[myNode, label=above:$u$] (C) at ($(u)$) {};
		\node[myNode, label=below:$v$] (D) at ($(v)$) {};
		
		\node[label=above:$P_{1}$] () at ($(p1)$) {};
		\node[label=below:$P_{2}$] () at ($(p2)$) {};
		\node[myNode, label=above:$x'$] () at ($(xx)$) {};
		\node[myNode, label=above:$y'$] () at ($(yy)$) {};

		\foreach \from/\to in {v1/A,v1/B,v2/A,v2/B}
		\draw (\from) -- (\to);
		
		\draw[-] (A)    to (B);
		\draw[-] (C)    to (D);
		\draw[-] (A) to (xx);
		\draw[-] (B) to (yy);
		
		\end{scope}
		\end{tikzpicture}
		\vspace{-.5cm}
		\vspace{-0mm}\caption{Possible configurations of chords $xx',$ $yy',$ as in the proof of Subclaim.}\label{estrange}
	\end{figure}
	
	According to the Subclaim, all chords of $C$, other than $uv,$ are incident to the same vertex of the separator, say $x$ - but then $x$ is a
	{${\cal P}$-apex} vertex of $G,$ a contradiction. Therefore, an $(x,y)$-disjoint chord of $G$ cannot exist and this concludes the proof of Claim 2.\vspace{-4mm}

	\begin{figure}[H]
		\centering
		\vspace{-.5cm}
		\begin{tikzpicture}[myNode/.style = black node, scale=0.7]
		\node[myNode, label=above:$x$] (A) at (0,2) {};
		\node[myNode, label=below:$y$] (B) at (0,0) {};

		\node[myNode] (v1) [below left= 0.1cm and 0.5cm of A] {};
		\node[myNode] (v2) [below left= 0.9cm and 0.5cm of A] {};
		
		\coordinate (Middle) at ($(A)!0.5!(B)$);
		\node (H) [right =0.15cm of Middle] {$H$};
		
		\begin{scope}[on background layer]
		\filldraw[fill=green!70!blue,  opacity=0.7, dashed, thick]
		(B.center) to [bend right =110, min distance =5cm] coordinate[pos=0.35] (y2) coordinate[pos=0.45] (y1) coordinate[pos=0.55] (x1) coordinate[pos=0.65] (x2) (A.center);
		\end{scope}

		\node[myNode] () at ($(x1)$) {};
		\node[myNode] () at ($(x2)$) {};
		\node[myNode] () at ($(y1)$) {};
		\node[myNode] () at ($(y2)$) {};
		
		\foreach \from/\to in {v1/A,v1/B,v2/A,v2/B, A/x1, A/x2, B/y1, B/y2}
		\draw (\from) -- (\to);
		
		\draw[-] (A)    to (B);
		\end{tikzpicture}
		\vspace{-.5cm}
		\vspace{-2mm}\caption{An example of $H$ having at least two $x$-chords and at least two $y$-chords.} 
		\label{outsized}
	\end{figure}
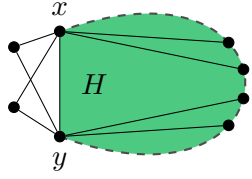
	
	Now, since $x$ is not an {${\cal P}$-apex} there exist two chords of $C$ not incident to $x.$ By Claim 2, these are $y$-chords of $G$. Symmetrically, there exist two $x$-chords of $G$.
	Therefore, we have that $\hyperref[solution]{{O}_{13}^{2}}\leq G$ 
	(as shown in \autoref{outsized}), a contradiction. 
\end{proof}

\begin{lemma}\label{personal}
	If $G\in \obs({\cal A}_{1}({\cal P}))\setminus  {\cal O}$ 
	then $G$ is triconnected.
\end{lemma}

\begin{proof}
	Suppose, to the contrary, that $G$ is not triconnected. Then, by \autoref{destined},
	\autoref{atrocity}, and \autoref{toulouse},
	$G$ is biconnected and outerplanar and so, due to \autoref{decoding}, it contains a Hamiltonian cycle, namely $C.$
	
	Observe that $G$ has at most 3 vertices of degree 2. Indeed, if there exist 4 vertices $v_{1},
	v_{2}, v_{3}, v_{4}\in V(G)$ of degree 2, then, by \autoref{generals},
	they are simplicial, which also implies that no pair of them is adjacent.
	Thus, by contracting all edges of $C,$ except those that are
	incident to $v_{1}, v_{2}, v_{3}, v_{4},$ we can form 
	$\hyperref[softened]{{O}_{3}^{2}}$ as a minor of
	$G$ (as depicted in \autoref{backpack}), a contradiction.

	\begin{figure}[H]
		\centering
		\begin{tikzpicture}[myNode/.style = black node, scale = 0.4]
		
		\begin{scope}
		\begin{scope}[xshift=2.5cm]
		\node[myNode, label=right:$v_{2}$] (v2) at (0:0.4) {};
		\node[myNode] (r2) at (120:0.9) {};
		\node[myNode] (l2) at (240:0.9) {};
		\draw (v2) -- (r2) -- (l2) -- (v2);
		\end{scope}
		
		\begin{scope}[xshift=-2.5cm]
		\node[myNode, label=left:$v_{4}$] (v4) at (180:0.4) {};
		\node[myNode] (r4) at (300:0.9) {};
		\node[myNode] (l4) at (60:0.9) {};
		\draw (v4) -- (r4) -- (l4) -- (v4);
		\end{scope}
		
		\begin{scope}[yshift=-2.5cm]
		\node[myNode, label=below:$v_{3}$] (v3) at (270:0.4) {};
		\node[myNode] (r3) at (30:0.9) {};
		\node[myNode] (l3) at (150:0.9) {};
		\draw (v3) -- (r3) -- (l3) -- (v3);
		\end{scope}
		
		\begin{scope}[yshift=2.5cm]
		\node[myNode, label=above:$v_{1}$] (v1) at (90:0.4) {};
		\node[myNode] (r1) at (210:0.9) {};
		\node[myNode] (l1) at (330:0.9) {};
		\draw (v1) -- (r1) -- (l1) -- (v1);
		\end{scope}
		
		\begin{scope}[on background layer]
		\draw[dashed, thick, fill=mustard!90] (l1.center) to [out=-30, in=120] (r2.center);
		\draw[dashed, thick, fill=mustard!90] (l2.center) to [out=240, in=30] (r3.center);
		\draw[dashed, thick, fill=mustard!90] (l3.center) to [out=150, in=300] (r4.center);
		\draw[dashed, thick, fill=mustard!90] (l4.center) to [out=60, in=210] (r1.center);
		\draw[color=mustard!90, ultra thick] (l1) -- (r2);
		\draw[color=mustard!90, ultra thick] (l2) -- (r3);
		\draw[color=mustard!90, ultra thick] (l3) -- (r4);
		\draw[color=mustard!90, ultra thick] (l4) -- (r1);
		\filldraw[color=mustard!90, fill=mustard!90] (l1.center) -- (r2.center) -- (l2.center) -- (r3.center) -- (l3.center) -- (r4.center) -- (l4.center) -- (r1.center) -- cycle;
		\end{scope}
		
		\end{scope}
		
		\end{tikzpicture}
		\vspace{-2mm}\caption{An outerplanar graph $G\in \obs({\cal A}_{1}({\cal P}))\setminus {{\cal O}}$ having 4 vertices of degree 2.}\label{backpack}
	\end{figure}
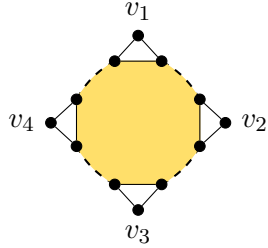

	On the other hand, by \autoref{hegelian}, $G$ has at least two vertices of degree 2. Thus, we distinguish the following two cases:
	
	\medskip
	
	\noindent {\em Case 1:} $G$ has exactly 2 vertices $u,v$ of degree 2.
	Note that, by \autoref{generals}, $u,v$ are simplicial. 
	
	Observe that $C \setminus u \setminus v$ is the union of two vertex disjoint paths $P_{1}, P_{2}$.
	Let $u_{1}, u_{2}$ be the neighbors of $u$ in $P_{1}, P_{2},$ respectively and $v_{1}, v_{2}$ be the neighbors of $v$ in $P_{1}, P_{2},$ respectively. Therefore, the structure of the graph $G$ is as follows:

	\begin{figure}[H]
		\centering
		\begin{tikzpicture}[myNode/.style = black node, scale=.6]
		
		\begin{scope}
		\draw (60:1) {} -- (180:0.5) {} -- (300:1) {} -- cycle;
		\node[myNode, label=above:$u_{1}$] (u1) at (60:1) {};
		\node[myNode, label=left:$u$] (u) at (180:0.5) {};
		\node[myNode, label=below:$u_{2}$] (u2) at (300:1) {};
		
		\begin{scope}[xshift=4.5cm]
		\node[myNode, label=right:$v$] (v) at (0:0.5) {};
		\node[myNode, label=above:$v_{1}$] (v1) at (120:1) {};
		\node[myNode, label=below:$v_{2}$] (v2) at (240:1) {};
		\draw[-] (v) -- (v1) -- (v2) -- (v);
		\end{scope}

		\begin{scope}[on background layer]
		\draw[dashed, thick, fill=mustard!90] (u1.center) to [out=25, in=155]  coordinate[pos=.5] (A) coordinate[pos=.3] (B) coordinate[pos=.7] (C) (v1.center);
		\draw[dashed, thick, fill=mustard!90] (u2.center) to [out=-25, in=-155]  coordinate[pos=.5] (B) (v2.center);
		\filldraw[color=mustard!90, fill=mustard!90] (v2) rectangle (u1);
		\end{scope}
		
		\node[label=above:$P_{1}$] () at (A) {};
		\node[label=below:$P_{2}$] () at (B) {};
		
		\end{scope}

		\end{tikzpicture}
		\vspace{-2mm}\caption{The structure of the graph $G$ in Case 1 of \autoref{personal}.} \label{matthias}
	\end{figure}
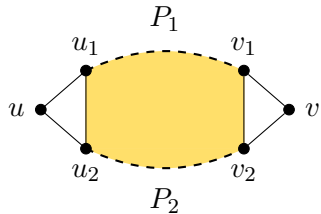

	We prove the following claim concerning the chords of $C.$ 
	
	\medskip
	
	\noindent{\em Claim 1:} Every chord of $C$ is between a vertex of $P_{1}$ and a vertex of $P_{2}$.
	
	\smallskip
	
	\noindent{\em Proof of Claim 1:}
	Suppose, to the contrary, that there exists an edge connecting non-consecutive
	vertices of $P_1$ or $P_2,$ say $P_1,$ and let $e = xy$ be such an edge whose 
	endpoints have the smallest possible distance in $P_{1}.$
	
	Let $P_{1}'$ be the subpath of $P_{1}$ between $x$ and $y.$ 
	Since $e$ is a chord of $C$, $P_{1}'$ contains an internal vertex $w.$ 
	Note then that, since $u,v$ are the only vertices of $G$ of degree 2, $w$ is of degree greater than $2$
	and so there exists a neighbour of $w,$ say $z,$ such that $wz\not\in E(P_{1})$.
	Observe that $K_{4}$-freeness of $G$ implies that $z$ is a vertex of $P_{1}'$
	(see leftmost figure in \autoref{exhibits}).
	But then $wz$ is an edge connecting non-consecutive vertices of $P_{1}$ 
	whose endpoints have smaller distance (in $P_{1}$) than that of $x,y$ 
	(see rightmost figure in \autoref{exhibits}), a 
	contradiction to the minimality of $P_{1}'$.
	This concludes the proof of Claim 1.
	
	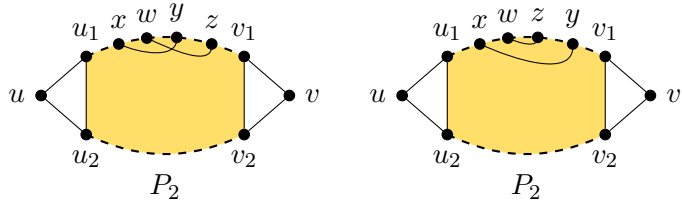
\begin{figure}[H]
		\centering
		\begin{tikzpicture}[myNode/.style = black node, scale=.6]
		
		\begin{scope}
		\draw (60:1) {} -- (180:0.5) {} -- (300:1) {} -- cycle;
		\node[myNode, label=above:$u_{1}$] (u1) at (60:1) {};
		\node[myNode, label=left:$u$] (u) at (180:0.5) {};
		\node[myNode, label=below:$u_{2}$] (u2) at (300:1) {};
		
		\begin{scope}[xshift=4.5cm]
		\node[myNode, label=right:$v$] (v) at (0:0.5) {};
		\node[myNode, label=above:$v_{1}$] (v1) at (120:1) {};
		\node[myNode, label=below:$v_{2}$] (v2) at (240:1) {};
		\draw[-] (v) -- (v1) -- (v2) -- (v);
		\end{scope}
		
		\begin{scope}[on background layer]
		\draw[dashed, thick, fill=mustard!90] (u1.center) to [out=25, in=155]  coordinate[pos=.2] (V1) 
		coordinate[pos=.38] (V2) 
		coordinate[pos=.5] (M1) 
		coordinate[pos=.575] (V3) 
		coordinate[pos=.8] (V4) 
		(v1.center);
		\draw[dashed, thick, fill=mustard!90] (u2.center) to [out=-25, in=-155]  coordinate[pos=.5] (M2) (v2.center);
		\filldraw[color=mustard!90, fill=mustard!90] (v2) rectangle (u1);
		\end{scope}
		
		\node[myNode, label=above:$x$] (k1) at (V1) {};
		\node[myNode, label=above:$w$] (k2) at (V2) {};
		\node[myNode, label=above:$y$] (k3) at (V3) {};
		\node[myNode, label=above:$z$] (k4) at (V4) {};
		
		\draw (k1.center) to [out=-25, in=-85]  (k3.center);
		\draw (k2.center) to [out=-25, in=-85]  (k4.center);
		
		\node[label=below:$P_{2}$] () at (M2) {};
		\end{scope}

		\begin{scope}[xshift=8cm]
		\draw (60:1) {} -- (180:0.5) {} -- (300:1) {} -- cycle;
		\node[myNode, label=above:$u_{1}$] (u1) at (60:1) {};
		\node[myNode, label=left:$u$] (u) at (180:0.5) {};
		\node[myNode, label=below:$u_{2}$] (u2) at (300:1) {};
		
		\begin{scope}[xshift=4.5cm]
		\node[myNode, label=right:$v$] (v) at (0:0.5) {};
		\node[myNode, label=above:$v_{1}$] (v1) at (120:1) {};
		\node[myNode, label=below:$v_{2}$] (v2) at (240:1) {};
		\draw[-] (v) -- (v1) -- (v2) -- (v);
		\end{scope}
		
		\begin{scope}[on background layer]
		\draw[dashed, thick, fill=mustard!90] (u1.center) to [out=25, in=155]  coordinate[pos=.2] (V1) 
		coordinate[pos=.38] (V2) 
		coordinate[pos=.5] (M1) 
		coordinate[pos=.575] (V3) 
		coordinate[pos=.8] (V4) 
		(v1.center);
		\draw[dashed, thick, fill=mustard!90] (u2.center) to [out=-25, in=-155]  coordinate[pos=.5] (M2) (v2.center);
		\filldraw[color=mustard!90, fill=mustard!90] (v2) rectangle (u1);
		\end{scope}
		
		\node[myNode, label=above:$x$] (k1) at (V1) {};
		\node[myNode, label=above:$w$] (k2) at (V2) {};
		\node[myNode, label=above:$z$] (k3) at (V3) {};
		\node[myNode, label=above:$y$] (k4) at (V4) {};
		
		\draw (k1.center) to [out=-25, in=-85]  (k4.center);
		\draw (k2.center) to [out=-25, in=-85]  (k3.center);
		
		\node[label=below:$P_{2}$] () at (M2) {};
		\end{scope}
		\end{tikzpicture}

		\vspace{-2mm}\caption{The chord $wz$ in the proof of Claim 1.}
		\label{exhibits}
	\end{figure}
	
	We now make a series of observations:
	
	\smallskip
	
	\noindent{\em Observation 1:} Every internal vertex of $P_{1}$ and $P_{2}$ is 
	incident to a chord. This follows immediately from \autoref{generals} and the
	fact that $u,v$ are the only vertices of degree $2$.
	
	\smallskip
	
	\noindent{\em Observation 2:} Every internal vertex of $P_{1}$ is adjacent to $u_{2}$ or
	$v_{2}$. Respectively, every internal vertex of $P_{2}$ is adjacent to $u_{1}$ or $v_{1}$.
	Indeed, if there exists an internal vertex $x$ of $P_{1}$ or $P_{2}$, say $P_{1}$, 
	not incident to $u_{2}$ or $v_{2}$ then by Observation 1 and Claim 1 $x$ is adjacent 
	to an internal vertex of $P_{2}$ and hence $\hyperref[straight]{{O}_{15}^{2}}\leq G$
	(as shown in \autoref{brawling}), a contradiction.

	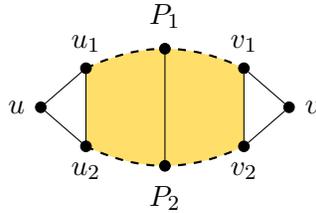
\begin{figure}[H]
		\centering
		\vspace{-0.5cm}
		\begin{tikzpicture}[myNode/.style = black node, scale=.6]
		
		\begin{scope}
		\draw (60:1) {} -- (180:0.5) {} -- (300:1) {} -- cycle;
		\node[myNode, label=above:$u_{1}$] (u1) at (60:1) {};
		\node[myNode, label=left:$u$] (u) at (180:0.5) {};
		\node[myNode, label=below:$u_{2}$] (u2) at (300:1) {};
		
		\begin{scope}[xshift=4.5cm]
		\node[myNode, label=right:$v$] (v) at (0:0.5) {};
		\node[myNode, label=above:$v_{1}$] (v1) at (120:1) {};
		\node[myNode, label=below:$v_{2}$] (v2) at (240:1) {};
		\draw[-] (v) -- (v1) -- (v2) -- (v);
		\end{scope}
		
		\begin{scope}[on background layer]
		\draw[dashed, thick, fill=mustard!90] (u1.center) to [out=25, in=155]  coordinate[pos=.5] (M1) coordinate[pos=.3] (L1) coordinate[pos=.7] (R1) (v1.center);
		\draw[dashed, thick, fill=mustard!90] (u2.center) to [out=-25, in=-155]  coordinate[pos=.5] (M2) (v2.center);
		\filldraw[color=mustard!90, fill=mustard!90] (v2) rectangle (u1);
		\end{scope}
		
		\draw (M1) node[myNode] {} -- (M2) node[myNode] {};
		
		\node[label=above:$P_{1}$] () at (M1) {};
		\node[label=below:$P_{2}$] () at (M2) {};
		
		\end{scope}
		
		\end{tikzpicture}
		\vspace{-2mm}\caption{A chord connecting two internal vertices of $P_{1},P_{2}$ in the proof of Observation 2 of Case 1.}
		\label{brawling}
	\end{figure}

	\noindent{\em Observation 3:} One of $P_{1}, P_{2}$ must be of length at most 1. Indeed, suppose
	to the contrary, that both $P_{1},P_{2}$ are of length at least $2$. Then, both $P_{1},P_{2}$
	contain an internal vertex, say $x_{1},x_{2}$, respectively.
	Then, by Observation 1, they are both incident to some chord of $C$.
	By Observation 2, $x_{1}$ is adjacent to $u_{2}$ or $v_{2}$, say $u_{2}$. Then, again by Observation
	2 and $K_{4}$-freeness  of $G$, $x_{2}$ is adjacent to $v_{1}$, as shown in \autoref{stifling}.
	But then, $\hyperref[magnetic]{{O}_{1}^{0}}\leq G$, a contradiction.
	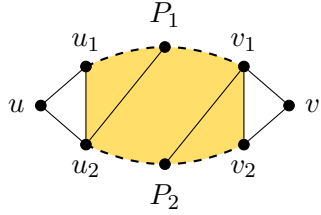
\begin{figure}[H]
		\centering
		\begin{tikzpicture}[myNode/.style = black node, scale=.6]
		
		\begin{scope}
		\draw (60:1) {} -- (180:0.5) {} -- (300:1) {} -- cycle;
		\node[myNode, label=above:$u_{1}$] (u1) at (60:1) {};
		\node[myNode, label=left:$u$] (u) at (180:0.5) {};
		\node[myNode, label=below:$u_{2}$] (u2) at (300:1) {};
		
		\begin{scope}[xshift=4.5cm]
		\node[myNode, label=right:$v$] (v) at (0:0.5) {};
		\node[myNode, label=above:$v_{1}$] (v1) at (120:1) {};
		\node[myNode, label=below:$v_{2}$] (v2) at (240:1) {};
		\draw[-] (v) -- (v1) -- (v2) -- (v);
		\end{scope}
		
		\begin{scope}[on background layer]
		\draw[dashed, thick, fill=mustard!90] (u1.center) to [out=25, in=155]  coordinate[pos=.5] (M1) coordinate[pos=.3] (L1) coordinate[pos=.7] (R1) (v1.center);
		\draw[dashed, thick, fill=mustard!90] (u2.center) to [out=-25, in=-155]  coordinate[pos=.5] (M2) (v2.center);
		\filldraw[color=mustard!90, fill=mustard!90] (v2) rectangle (u1);
		\end{scope}
		
		\draw (u2) -- (M1) node[myNode] {};
		\draw (v1) -- (M2) node[myNode] {};
		
		\node[label=above:$P_{1}$] () at (M1) {};
		\node[label=below:$P_{2}$] () at (M2) {};
		
		\end{scope}

		\end{tikzpicture}
		\vspace{-2mm}\caption{An example of the graph $G$ in the proof of Observation  3 of Case 1.}
		\label{stifling}
	\end{figure}

	\noindent{\em Observation 4:} Either $u_{1}$ or $v_{1}$ is adjacent to every internal vertex of
	$P_{2}$. Respectively, either $u_{2}$ or $v_{2}$ is adjacent to every internal vertex of $P_{1}$.
	Indeed, if otherwise, then Observations 1 and 2 imply, without loss of generality, that there exist
	$x,y\in V(P_{1})\setminus \{u_{1},v_{1}\}$ such that $x,y$ are adjacent to $u_{2},v_{2}$,
	respectively. Hence, $\hyperref[solution]{{O}_{13}^{2}}\leq G$ (see \autoref{murkiest}). 
	
	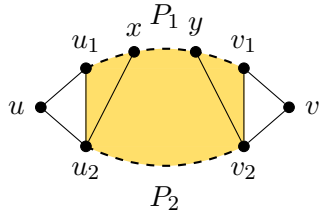
\begin{figure}[H]
		\centering
		\begin{tikzpicture}[myNode/.style = black node, scale=.6]
		
		\begin{scope}
		\draw (60:1) {} -- (180:0.5) {} -- (300:1) {} -- cycle;
		\node[myNode, label=above:$u_{1}$] (u1) at (60:1) {};
		\node[myNode, label=left:$u$] (u) at (180:0.5) {};
		\node[myNode, label=below:$u_{2}$] (u2) at (300:1) {};
		
		\begin{scope}[xshift=4.5cm]
		\node[myNode, label=right:$v$] (v) at (0:0.5) {};
		\node[myNode, label=above:$v_{1}$] (v1) at (120:1) {};
		\node[myNode, label=below:$v_{2}$] (v2) at (240:1) {};
		\draw[-] (v) -- (v1) -- (v2) -- (v);
		\end{scope}
		
		\begin{scope}[on background layer]
		\draw[dashed, thick, fill=mustard!90] (u1.center) to [out=25, in=155]  coordinate[pos=.5] (M1) coordinate[pos=.3] (L1) coordinate[pos=.7] (R1) (v1.center);
		\draw[dashed, thick, fill=mustard!90] (u2.center) to [out=-25, in=-155]  coordinate[pos=.5] (M2) (v2.center);
		\filldraw[color=mustard!90, fill=mustard!90] (v2) rectangle (u1);
		\end{scope}
		
		\draw (u2) -- (L1) node[myNode, label=above:$x$] {};
		\draw (v2) -- (R1) node[myNode, label=above:$y$] {};
		
		\node[label=above:$P_{1}$] () at (M1) {};
		\node[label=below:$P_{2}$] () at (M2) {};
		
		\end{scope}
		
		\end{tikzpicture}
		\vspace{-2mm}\caption{Two chords $x u_{2} , y v_{2} $ where $x,y$ are internal vertices of $P_{1}$ in Observation 4 of Case 1.}
		\label{murkiest}
	\end{figure}
	
	Now, by Observation 3, we may assume that $P_{2}$ is of length $j \leq 1.$
	Then, by Observation 4, $u_{2}$ or $v_{2}$, say $u_{2}$, is adjacent to every 
	internal vertex of $P_{1}$.
	This implies that every chord of $C,$ except for $v_{1} v_{2}$ (if $v_{2}\not=u_{2}$),
	is incident to $u_{2}$ and hence $u_{2}$ is a {${\cal P}$-apex} vertex of $G,$ a contradiction. This concludes Case 1.

	\medskip

	\noindent {\em Case  2:} G has exactly 3 vertices $u,v,w$ of degree 2.
	Note that, by \autoref{generals}, $u,v,w$ are simplicial.

	Note that if all three vertices have pairwise disjoint closed neighbourhoods,
	then $\hyperref[barracks]{{O}_{14}^{2}}\leq G$,  which is a contradiction. 
	Therefore, at least two of them, say $u$ and $v,$ have non-disjoint neighbourhoods.
	We argue that the following holds:
	
	\medskip
	
	\noindent{\em Claim 2:} $N_{G}(u) \cap N_{G}(v) = \{x\}$ for some $x \in V(G).$\medskip
	
	\noindent{\em Proof of Claim 2:}
	Since $u,v$ have non-disjoint neighbourhoods, then either $N_{G}(u) \cap N_{G}(v) = e$ for some
	edge $e \in E(G)$ or $N_{G}(u) \cap N_{G}(v) =\{x\}$ for some vertex $x \in V(G).$
	
	Suppose that $N_{G}(u) \cap N_{G}(v) = \{a,b\}$ for some edge $e = ab \in E(G)$
	and consider any two internally vertex disjoint paths
	(which exist due to biconnectivity of $G$) from $w$ to, say, $v.$ Observe that one of the paths
	contains $a$ and the other contains $b$. Therefore, $K_{2,3} \leq G,$ with $\{a,b\}$ forming
	one part of the $K_{2,3}$ minor and $\{u,v,w\}$ forming the other, a contradiction. Claim 2 follows.
	
	\medskip
	
	By Claim 2, $C\setminus \{u, v, w, x\}$ is the union of two vertex disjoint paths $R_{1}, R_{2}$.
	We can assume that $u$ has a neighbor $u_{1}$ in $R_{1}$, $v$ has a neighbor $v_{2}$ in $R_{2},$ and $N_{G}(w)=\{w_{1},w_{2}\},$ where $w_{1}\in R_{1}$ and $w_{2}\in R_{2}.$
	By arguments identical to the proof of Claim 1 in Case 1, we have that every chord of $C$ is between a vertex in $R_{1} \cup \{ x \}$ and a vertex in $R_{2} \cup \{ x \}$.
	Therefore, the structure of the graph $G$ is as follows (\autoref{overcome}):

	\begin{figure}[H]
		\centering
		\vspace{-.5cm}
		\begin{tikzpicture}[myNode/.style = black node, scale=.6]
		
		\begin{scope}
		\begin{scope}[yshift=0.8cm]
		\node[myNode, label=above:$u_{1}$] (u1) at (40:0.7) {};
		\node[myNode, label=left:$u$] (u) at (150:0.8) {};
		\node[myNode, label=left:$x$] (x) at (270:0.8) {};
		\draw (x) -- (u1) -- (u) -- (x);
		\end{scope}
		
		\begin{scope}[yshift=-0.8cm]
		\node[myNode, label=left:$v$] (v) at (210:0.8) {};
		\node[myNode, label=below:$v_{2}$] (v2) at (320:0.7) {};
		\draw (90:0.8) {} -- (v) {} -- (v2) {} -- (90:0.8);
		\end{scope}
		
		\begin{scope}[xshift=5cm]
		\node[myNode, label=right:$w$] (w) at (0:0.5) {};
		\node[myNode, label=above:$w_{1}$] (w1) at (120:1) {};
		\node[myNode, label=below:$w_{2}$] (w2) at (240:1) {};
		\draw (w) {} -- (w1) {} -- (w2) {} -- (w);
		\end{scope}
		
		\begin{scope}[on background layer]
		\draw[dashed, thick, fill=mustard!90] (u1.center) to [out=15, in=150]  coordinate[pos=.5] (M1) coordinate[pos=.3] (L1) coordinate[pos=.7] (R1) (w1.center);
		\draw[dashed, thick, fill=mustard!90] (v2.center) to [out=-15, in=-150]  coordinate[pos=.5] (M2) (w2.center);
		\filldraw[color=mustard!90, fill=mustard!90] (v2.center) -- (x.center) --  (u1.center) -- (w1.center) -- (w2.center) -- cycle;
		\end{scope}
		
		\node[label=above:$R_{1}$] () at (M1) {};
		\node[label=below:$R_{2}$] () at (M2) {};
		\end{scope}

		\end{tikzpicture}
		\vspace{-2mm}\caption{The structure of the graph $G$ in Case 2 of \autoref{personal}.}\label{overcome}
	\end{figure}
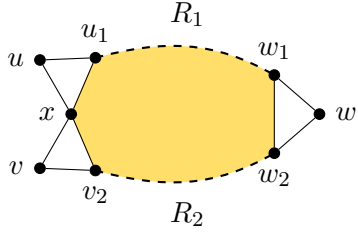

	We now observe the following about the paths $R_1,R_2$:\vspace{-2mm}
	
	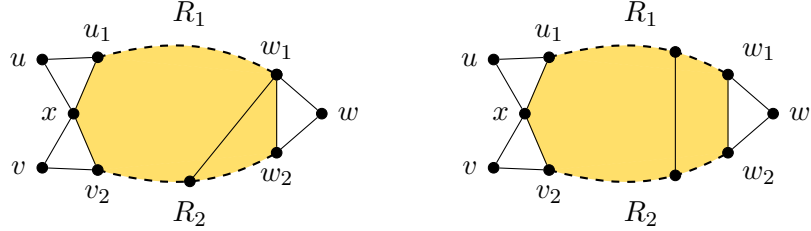
\begin{figure}[H]
		\centering
		\begin{tikzpicture}[myNode/.style = black node, scale=.6]
		
		\begin{scope}
		\begin{scope}[yshift=0.8cm]
		\node[myNode, label=above:$u_{1}$] (u1) at (40:0.7) {};
		\node[myNode, label=left:$u$] (u) at (150:0.8) {};
		\node[myNode, label=left:$x$] (x) at (270:0.8) {};
		\draw (x) -- (u1) -- (u) -- (x);
		\end{scope}
		
		\begin{scope}[yshift=-0.8cm]
		\node[myNode, label=left:$v$] (v) at (210:0.8) {};
		\node[myNode, label=below:$v_{2}$] (v2) at (320:0.7) {};
		\draw (90:0.8) {} -- (v) {} -- (v2) {} -- (90:0.8);
		\end{scope}
		
		\begin{scope}[xshift=5cm]
		\node[myNode, label=right:$w$] (w) at (0:0.5) {};
		\node[myNode, label=above:$w_{1}$] (w1) at (120:1) {};
		\node[myNode, label=below:$w_{2}$] (w2) at (240:1) {};
		\draw (w) {} -- (w1) {} -- (w2) {} -- (w);
		\end{scope}
		
		\begin{scope}[on background layer]
		\draw[dashed, thick, fill=mustard!90] (u1.center) to [out=15, in=150]  coordinate[pos=.5] (M1) coordinate[pos=.3] (L1) coordinate[pos=.7] (R1) (w1.center);
		\draw[dashed, thick, fill=mustard!90] (v2.center) to [out=-15, in=-150]  coordinate[pos=.5] (M2) (w2.center);
		\filldraw[color=mustard!90, fill=mustard!90] (v2.center) -- (x.center) --  (u1.center) -- (w1.center) -- (w2.center) -- cycle;
		\end{scope}
		
		\draw (w1) node[myNode] {} -- (M2) node[myNode] {};
		
		\node[label=above:$R_{1}$] () at (M1) {};
		\node[label=below:$R_{2}$] () at (M2) {};
		\end{scope}
		
		\begin{scope}[xshift=10cm]
		\begin{scope}[yshift=0.8cm]
		\node[myNode, label=above:$u_{1}$] (u1) at (40:0.7) {};
		\node[myNode, label=left:$u$] (u) at (150:0.8) {};
		\node[myNode, label=left:$x$] (x) at (270:0.8) {};
		\draw (x) -- (u1) -- (u) -- (x);
		\end{scope}
		
		\begin{scope}[yshift=-0.8cm]
		\node[myNode, label=left:$v$] (v) at (210:0.8) {};
		\node[myNode, label=below:$v_{2}$] (v2) at (320:0.7) {};
		\draw (90:0.8) {} -- (v) {} -- (v2) {} -- (90:0.8);
		\end{scope}
		
		\begin{scope}[xshift=5cm]
		\node[myNode, label=right:$w$] (w) at (0:0.5) {};
		\node[myNode, label=above right:$w_{1}$] (w1) at (120:1) {};
		\node[myNode, label=below right:$w_{2}$] (w2) at (240:1) {};
		\draw (w) {} -- (w1) {} -- (w2) {} -- (w);
		\end{scope}
		
		\begin{scope}[on background layer]
		\draw[dashed, thick, fill=mustard!90] (u1.center) to [out=15, in=150]  coordinate[pos=.5] (M1) coordinate[pos=.3] (L1) coordinate[pos=.7] (R1) (w1.center);
		\draw[dashed, thick, fill=mustard!90] (v2.center) to [out=-15, in=-150]  coordinate[pos=.5] (M2) coordinate[pos=.7] (R2) (w2.center);
		\filldraw[color=mustard!90, fill=mustard!90] (v2.center) -- (x.center) --  (u1.center) -- (w1.center) -- (w2.center) -- cycle;
		\end{scope}
		
		\draw (R1) node[myNode] {} -- (R2) node[myNode] {};
		
		\node[label=above:$R_{1}$] () at (M1) {};
		\node[label=below:$R_{2}$] () at (M2) {};
		\end{scope}
		
		\end{tikzpicture}
		\vspace{-2mm}\caption{The two configurations of the chord not incident to $x$ in the proof of Observation 5 of Case 2.}\label{rebutted}
	\end{figure}

	\noindent{\em Observation 5:} One of $R_{1},R_{2}$ is of length $0$. Indeed, suppose, towards a
	contradiction, that both $R_{1},R_{2}$ are of length at least $1$. Then,
	since $x$ is not a ${\cal P}$-apex vertex of $G$, there exists a chord $e$ between a vertex in $R_{1}$ and a vertex in $R_{2}$ such that
	$e\neq w_{1}w_{2}.$ Then, if $e$ is incident to $w_{1}w_{2}$ we have that $\hyperref[typifies]{{O}_{12}^{2}}\leq G,$ while if $e$ is disjoint from $w_{1}w_{2}$ we have that $\hyperref[uttering]{{O}_{7}^{2}}\leq G,$ a contradiction in both
	cases (see \autoref{rebutted}).\vspace{-4mm}

	\begin{figure}[H]
		\centering
		\begin{tikzpicture}[myNode/.style = black node, scale=.6]
		
		\begin{scope}
		\begin{scope}
		\node[myNode, label=below:$v$] (v) at (270:0.9) {};
		\node[myNode, label=220:$x$] (x) at (150:0.9) {};
		\draw (v) -- (x) -- (30:0.9) -- (v);
		\end{scope}
		
		\begin{scope}[shift=(30:1.8cm)]
		\node[myNode, label=330:$w_{2}$] (w2) at (210:0.9) {};
		\node[myNode, label=330:$w$] (w) at (330:0.9) {};
		\node[myNode, label=right:$w_{1}$] (w1) at (90:0.9) {};
		\draw (w) -- (w1) -- (w2) -- (w);
		\end{scope}
		
		\begin{scope}[shift=(150:1.8cm)]
		\node[myNode, label=left:$u_{1}$] (u1) at (90:0.9) {};
		\node[myNode, label=210:$u$] (u) at (210:0.9) {};
		\draw (330:0.9) -- (u) -- (u1) -- (330:0.9);
		\end{scope}
		
		\begin{scope}[on background layer]
		\draw[dashed, thick, fill=mustard!90] (u1.center) to [out=50, in=130]  coordinate[pos=.5] (M) coordinate[pos=.3] (L) coordinate[pos=.7] (R) (w1.center);
		\filldraw[color=mustard!90, fill=mustard!90] (x.center) -- (u1.center) --  (w1.center) -- (w2.center) -- cycle;
		\end{scope}
		
		\node[label=above:$R_{1}$] () at (M) {};
		
		\draw (w2) -- (R) node[myNode, label=60:$z$] {};
		\draw (x) -- (L) node[myNode, label=140:$y$] {};
		\end{scope}

		\end{tikzpicture}
		\vspace{-2mm}\caption{An example of the graph $G$ in the last part of the proof of \autoref{personal}.}
		\label{undercut}
	\end{figure}
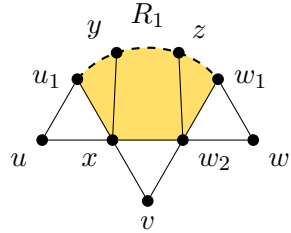

	By Observation 5, we can assume that $R_{2}$ is of length $0$, i.e. $v_{2}=w_{2}$. 
	Then, every chord of $C$, except for $xw_{2}$, is between a vertex of $R_{1}$ and a 
	vertex in $\{x,w_{2}\}$. Then, since $x$ and $w_{2}$ are not ${\cal P}$-apex 
	vertices of $G$ there exist $y,z$ (possibly with $y=z$)
	internal vertices of $R_{1}$ incident to $x$ and $w_{2},$ respectively.
	Hence, $\hyperref[solution]{{O}_{13}^{2}}\leq G$  (see \autoref{undercut}),
	a contradiction. The proof of \autoref{personal} is complete. 
\end{proof}

We are now in position to prove~\autoref{ruthless}.

\begin{proof}[{Proof} of~\autoref{ruthless}]
	As we mentioned in  the end of \autoref{writings}, $\obs({\cal A}_{1}({\cal P}))\supseteq  {\cal O}$ and what remains is 
	to prove that ${\cal O}\supseteq  \obs({\cal A}_{1}({\cal P}))$ or alternative that $\obs({\cal A}_{1}({\cal P}))\setminus {\cal O} = \emptyset.$
	For this assume, towards a contradiction, that there exists a graph $G\in  \obs({\cal A}_{1}({\cal P}))\setminus {\cal O}.$
	From \autoref{personal}, $G$ should be triconnected. Therefore, from \autoref{reliably}, either  ${{\cal O}^{3}}\leq G,$ a contradiction, or $G$ is  isomorphic to $W_{r},$ for some $r\geq 3,$ again a {contradiction}, as $W_{r}\in {\cal A}_{1}({\cal P})$ for all $r\geq 3$.
\end{proof}


\end{document}